%% file: arXiv-version1.tex
\documentclass[11pt,a4paper]{preprint}
\usepackage[utf8]{inputenc}
\usepackage[english]{babel}
\usepackage{helvet, stmaryrd, ulem}
\usepackage{amsmath}
\usepackage{appendix}
\usepackage{amsthm}
\usepackage{amsfonts}
\usepackage{amssymb}
\usepackage{mathtools}
\usepackage{mathrsfs}
\usepackage{dsfont}
\usepackage{enumitem}
\usepackage{tabularx}
\usepackage{xcolor}
\usepackage{bbm}
\usepackage{mhequ}
\usepackage{todonotes}
\usepackage{booktabs}
\usepackage[paper=a4paper,left=25mm,right=25mm,top=25mm,bottom=35mm]{geometry}
\usepackage[noabbrev,capitalize]{cleveref}

\newcommand{\cE}{\mathcal{E}}
\newcommand{\cK}{\mathcal{K}}
\def\scal#1{\langle #1 \rangle}

\newtheorem{theorem}{Theorem}[section]

\crefname{definition}{Definition}{Definitions}
\newtheorem{proposition}{Proposition}[section]
\crefname{proposition}{Proposition}{Propositions}
\newtheorem{lemma}{Lemma}[section]
\crefname{lemma}{Lemma}{Lemmata}

\newtheorem{corollary}{Corollary}[section]
\crefname{corollary}{Corollary}{Corollaries}

\crefname{example}{Example}{Examples}
\newtheorem{assumption}{Assumption}[section]
\crefname{assumption}{Assumption}{Assumptions}

\theoremstyle{definition}
\newtheorem{remark}{Remark}[section]
\crefname{remark}{Remark}{Remarks}

\numberwithin{equation}{section}

\newcommand{\supp}{\operatorname{supp}}
\newcommand{\R}{\mathbb{R}}

\newcommand{\cQ}{\mathcal{Q}}
\newcommand{\cF}{\mathcal{F}}

\newcommand{\T}{\mathbb{T}}
\newcommand{\N}{\mathbb{N}}
\newcommand{\E}{\mathds{E}}

\newcommand{\F}{\mathcal{F}}
\newcommand{\cB}{\mathcal{B}}
\newcommand{\cC}{\mathcal{C}}
\newcommand{\cM}{\mathcal{M}}
\newcommand{\bL}{\mathbb{L}}
\newcommand{\bS}{\mathbb{S}}

\newcommand{\calC}{\mathscr{C}}

\renewcommand{\mathcal}[1]{\mathscr{#1}}
\newcommand{\Qx}[2]{Q(#1,#2)}
\renewcommand{\epsilon}{\varepsilon}
\newcommand{\eps}{\varepsilon}
\newcommand{\bF}{\mathbb{F}}
\newcommand{\Z}{\mathbb{Z}}
\newcommand{\C}{\mathbb{C}}

\let\leqslant\leq
\let\geqslant\geq

\newcommand{\squeeze}[2][0]{\mbox{$\medmuskip=#1mu\displaystyle#2$}}

\newcommand{\Id}{\operatorname{Id}}

\DeclarePairedDelimiter{\abs}{\lvert}{\rvert}
\DeclarePairedDelimiter{\norm}{\lVert}{\rVert}

\DeclarePairedDelimiter{\paren}{(}{)}

\begin{document}

\title{Overcoming the spatial order barrier for nonlinear SPDEs with additive space-time white noise}

	\date{\today}
\author{Lukas Anzeletti\footnote{TU Wien,
\email{lukas.anzeletti@asc.tuwien.ac.at}}, Máté Gerencsér\footnote{TU Wien,
\email{mate.gerencser@asc.tuwien.ac.at}}, Helena Kremp\footnote{TU and WIAS Berlin,
\email{kremp@wias-berlin.de}}}

\maketitle

	\begin{abstract}
We introduce a fully discrete numerical scheme for semilinear SPDEs with additive space-time white noise that overcomes the previous order barrier for the spatial convergence rate. This scheme, which we refer to as the doubly accelerated exponential Euler scheme, achieves a strong convergence rate of $M^{-1+\eps}$ in time and $N^{-3/2+\eps}$ in space for any $\eps>0$, where $M^{-1}$ and $N^{-1}$ are the temporal, respectively the spatial, meshsizes. This substantially improves the standard spatial error bounds of order $N^{-1/2}$ in the literature. Numerical simulations support the findings.
		
		\bigskip
		
		 {{\sc Mathematics Subject Classification (2020):  60H15, 60H35 }
		
		}

	\end{abstract}

\tableofcontents

\section{Introduction}
When approximating solutions of stochastic differential equations, a fundamental challenge is that
the rate of convergence is often limited due to the low regularity of the solutions. 
When it comes to stochastic \textit{partial} differential equations (SPDEs), this issue appears both in the time and space variable,
limiting both temporal and spatial convergence rates. 
While the temporal rate has been significantly studied and in many cases successfully improved
(see below for more details on the literature),
to our best knowledge no discretisation scheme of an SPDE has yet been able to achieve a spatial strong rate of convergence that
is superior to the spatial regularity of the solution\footnote{To be more precise, this is valid for equations with function-valued solutions and approximations thereof in a (possibly discretised) function space. For equations whose solutions are distribution-valued and/or errors measured in a distributional space, the analogue of this order barrier would be the difference between the regularity of the solution and the error topology.}. Going beyond this barrier is the upshot of the present paper.

We consider this problem in the context of additive space-time white noise-driven semilinear SPDEs:
\begin{equ}\label{eq:main}
    \partial_t u=\Delta u+ f( u)+\xi, \qquad u(0,x)=u_0(x)
\end{equ}
on a time horizon $[0,T]$ and on the spatial domain $\T=\R/\Z$.
Under classical assumptions on the nonlinearity $f$ and the initial condition $u_0$ it is known that the equation admits a unique solution that is (almost) $1/4$-H\"older continuous in time and (almost) $1/2$-H\"older continuous in space.
The same exponents appear as the order barriers of a large class of numerical schemes.
 
To discuss these order barriers more precisely,
it is convenient to view numerical schemes as consisting of two qualitatively different steps\footnote{The description of numerical schemes given here applies for methods based on fixed observables and does not cover adaptive schemes.}. First, a set of Gaussian random variables are generated that are linear functionals of the noise, that is, are of the form $\xi(h)$ for some $h\in L^2([0,T]\times\T)$.
Second, a deterministic function (for example, a recursion given by a time-stepping method) is applied to the generated random variables.
With this in mind, one way to formalise the meaning of order barriers is as follows.
Choose a pseudometric $d$ on $C(\T)$. For a given finite set $\mathbb{S}=\{h_1,\ldots,h_{|\bS|}\}\subset L^2([0,T]\times \T)$ define
 \begin{equs}
     \mathbf{c}_d(\mathbb{S})=\inf_{\substack{\varphi:\R^{|\mathbb{S}|}\to C(\T)\\\text{\footnotesize measurable}}}
     \sqrt{\E \Big|d\Big(u(1,\cdot),\varphi\big(\xi(h_1),\ldots,\xi(h_{|\bS|})\big)\Big)\Big|^2}.
 \end{equs}
A lower bound $\delta$ for $\mathbf{c}_d(\mathbb{S})$ expresses that there is no numerical scheme based on the chosen random variables that makes an error less than $\delta$. If $\delta$ can be expressed as a power of $|\mathbb{S}|$, this exponent can be informally referred to as an order barrier.
 
\begin{remark}
    One reason why the comparison of various convergence results in the literature is tedious is that there is no  canonical choice of $d$. Common choices are the $L^p(\T)$ distances for $p\in[1,\infty]$, the discrete $L^p$ distances between functions restricted to a grid, or the absolute difference at a given point.
    In fact, it is also conceivable to measure differently in the time and $\omega$ variables, leading to a wide choice of error criteria.
\end{remark}

The numerical analysis of space-time white noise-drive SPDEs goes back to the work \cite{Gy}, which considered equations even more general than \eqref{eq:main}, where the noise term was also allowed to have a dependence on the solution of the form $g(u)\xi$. Both explicit in time and implicit in time finite difference schemes were considered on the space-time lattice $I_M\times \Pi_N$, where
\begin{equ}\label{eq:lattice}
    I_M\coloneqq\{t_k\coloneqq kh\coloneqq k\frac{T}{M},\quad k=0,\dots,M\},\quad \Pi_N\coloneqq\{x_n\coloneqq n\frac{1}{N},\quad n=0,\ldots,N\},
\end{equ}
with $M,N\in\N$, and a strong error bound of order $M^{-1/4}+N^{-1/2}$ was proven, matching the regularity of the solution.
These approximations can be written as functions of the random variables $\xi(h)$ for $h\in \mathbb{S}_{M,N}^{\mathrm{lat}}:=\{\mathbf{1}_{[t_{k},t_{k+1}]\times[x_n,x_{n+1}]}:\,\,k=0,\dots,M,\,n=0,\ldots,N-1\}$. With the asymptotically optimal choice $M\sim N^2$, the error bound is therefore $O(|\mathbb{S}_{M,N}|^{-1/6})$.
In the multiplicative case $f=0$, $g(u)=u$ this is shown to be optimal in a very strong sense in \cite[Theorem~3.1]{Davie-Gaines}:
with $d(v,w)=\Big|\int_{\T} v(x)-w(x)\,dx\,\Big|$, the authors prove that
\begin{equ}\label{eq:Davie-Gaines-very-strong}
    \inf_{|\mathbb{S}|\leq K}\mathbf{c}_d(\mathbb{S})\geq c K^{-1/6}
\end{equ}
with a universal constant $c>0$.
That is, no matter how one chooses the samples and no matter what scheme is applied to these samples, the order in \cite{Gy} cannot be improved.

In the additive case,  \cite[Section~2.1]{Davie-Gaines} establishes a weaker lower bound. Fixing the choice of linear functionals determined by $\mathbb{S}_{M,N}^{\mathrm{lat}}$, the rate of \cite{Gy} is shown to be optimal. Indeed, with the choice $d(v,w)=|v(0)-w(0)|$, the authors prove that with $f=0$ in \eqref{eq:main} one has
\begin{equ}\label{eq:Davie-Gaines-weaker}
  \mathbf{c}_d(\mathbb{S}_{M,N}^{\mathrm{lat}})\geq c (M^{-1/4}+N^{-1/2}).  
\end{equ}

Since no such strong lower bound as in \eqref{eq:Davie-Gaines-very-strong} is available for the additive case, it was natural to conjecture that replacing $\mathbb{S}_{M,N}^{\mathrm{lat}}$ by another set, an appropriate scheme could achieve higher rates.
A concrete proposal was put forth in
\cite{Jentzen-Kloden}: the so-called \textit{accelerated} exponential Euler scheme, which uses the set
\begin{equ}
    \mathbb{S}_{M,N}^{\mathrm{acc}}=\big\{t,x\mapsto \mathbf{1}_{[0,t_k]}(t)e^{-4\pi^2n^2(t_k-t)}e^{2\pi i n x}:\,\, k=0,\dots,M,\,n=0,\dots,N \big\},
\end{equ}
i.e. the scheme is based on truncating $O$, where $O$ is the solution of the linearised problem $\partial_t O=\Delta O+\xi$ with $O_0=0$,
to Fourier modes $\leq N$ and sampling the truncated solution on the temporal grid $I_M$.
This scheme provides a massive improvement by achieving temporal rate $1$ as opposed to the earlier $1/4$, which was referred to as \textit{``Overcoming the order barrier''}. Unfortunately, \cite{Jentzen-Kloden} imposes very restrictive conditions of $f$. The assumptions therein are phrased in terms of a condition on maps $F:L^2(\T)\to L^2(\T)$ 
and the only $f$ such that $\big(F(u)\big)(x)=f(u(x))$ satisfies the assumption \cite[Assumption~2.4]{Jentzen-Kloden} are the affine linear ones. In the genuinely nonlinear case, \cite{Jentzen12} established an improvement from $1/4$ to $1/2$, and the conjectured rate $1$ was open until very recently, when \cite{DjGK} proved that the temporal rate $1$ (with a loss of an arbitrarily small $\eps>0$) is achieved for any sufficiently regular and globally Lipschitz $f$  by the accelerated exponential Euler scheme. Moreover, a matching lower bound is established, that is, with $d(u,v)=\|u-v\|_{L^2(\T)}$ and with $f(u)=u$ in \eqref{eq:main} one has
\begin{equ}
    \mathbf{c}_d(\mathbb{S}_{M,N}^{\mathrm{acc}})\geq c (M^{-1}+N^{-1/2}).
\end{equ}
Notably, despite the substantial improvement in the temporal rate, the spatial rate remained $1/2$. The same spatial order barrier appears in \cite{BGJK}, which uses yet another set of linear functionals of the noise.
In the present paper we go substantially beyond the spatial order $1/2$ and construct a scheme that achieves spatial rate (arbitrarily close to) $3/2$.
Although the scheme is different from the one considered in \cite{Jentzen-Kloden, DjGK}, the temporal improvement proved in \cite{DjGK} is retained by our scheme, resulting in a full error estimate of order $M^{-1+\eps}+N^{-3/2+\eps}$ for any $\eps>0$. The following table summarises the evolution in the literature of the total computational effort (ignoring the losses of the arbitrarily small $\eps$) required to guarantee an error of order $\delta>0$.

\begin{center}
\begin{tabular}{|c|c|c|c|}
 \hline
 & temporal rate & spatial rate & total cost to ensure error $\lesssim\delta$ \\
\hline    
\cite{Gy}    & $M^{-1/4}$ & $N^{-1/2}$      & $\delta^{-6}$          \\
\hline
\cite{Jentzen12} & $M^{-1/2}$ & $N^{-1/2}$        & $\delta^{-4}$          \\
\hline
\cite{DjGK}                            & $M^{-1}$ & $N^{-1/2}$        & $\delta^{-3}$          \\
\hline
           Theorem \ref{thm:main1}           & $M^{-1}$ & $N^{-3/2}$             & $\delta^{-1.666\cdots}$  \\
           \hline

\end{tabular}

\end{center}

While in the works mentioned above the nonlinearity is assumed to be (among other assumptions) globally Lipschitz continuous, interesting examples such as the stochastic Allen-Cahn equation, with $f(r)=r-r^3$, fail this assumption.
The extension of classical results to such nonlinearities is a delicate matter, due to the known phenomenom of blow-up of explicit schemes for stochastic equations with superlinearly growing one-sided Lipschitz drift, see \cite{HutzJentz, beccari2019strong}.
The rate obtained by \cite{Gy} in the Lipschitz case, together with lower bounds, were proven by \cite{becker2023strong} in the Allen-Cahn case, albeit for a rather different discretisation. Yet another approach was employed in \cite{BrehierJH18, BrehierG19}, achieving rate $M^{-1/4}$ for a temporal discretisation via a splitting scheme. The splitting method will also be relevant in the scheme proposed in the present paper for nonlinearities with superlinear growth. The rate $M^{-1/2}$ obtained by \cite{Jentzen12} in the Lipschitz case was proven by \cite{Wang} in the Allen-Cahn case, modulo taming the scheme.
In the work \cite{DjGK}, as well as in the present paper, both the Lipschitz case and the Allen-Cahn case are treated. Our result on equations with superlinearly growing drift is Theorem \ref{thm:main2} below, matching the rate of Theorem \ref{thm:main1}. Finally, we mention that the reason why our error estimates do not follow a parabolic scaling (i.e. the spatial rate is not twice the temporal rate) arises in a fairly subtle point of the proof and we comment on it after the proof in Remark \ref{rem:rates}.

We comment on a few key aspects of the strategy of the article. 
The first novelty lies in the construction of the scheme: since we wish to avoid the spatial regularity of $O$ limiting the rate, we base the scheme on observables of $O$ that can be simulated \textit{exactly} and therefore at this stage no discretisation error whatsoever occurs. 
The quantities chosen are the values of $O$ on the grid $I_M\times \Pi_N$. In the notation of the discussion above, this means that we make the choice
\begin{equ}\label{eq:Snew}
\mathbb{S}_{M,N}^{\mathrm{new}}=\big\{t,x\mapsto  \mathbf{1}_{[0,t_k]}(t)p_{t_k-t}(x_n-x):\,\,k=0,\dots,M,\,n=1,\dots,N\big\},
\end{equ}
where $p_t(x)$ is the heat kernel on the torus $\T.$
To obtain an approximation for the solution $u$ of \eqref{eq:main} on the grid, it then suffices to approximate the remainder $v=u-O$, and we can benefit from the higher temporal and spatial regularities of $v$. Indeed, classical regularity estimates show that $v$ is twice differentiable in space and once in time (and their derivatives are (almost) $1/4$-H\"older continuous in time and (almost) $1/2$-H\"older continuous in space).

However, the improved regularity on its own does not improve the rate. To illustrate this issue, and the second novelty of the paper that overcomes it, we consider the following simplified question. Suppose that we are given the whole temporal trajectories of the restriction of $O$ to $\Pi_N$, henceforth denoted by $\Theta_NO$, and we wish to approximate the integral
\begin{equ}
    \int_0^t P_{t-s}f(O_s)\,ds
\end{equ}
on $\Pi_N.$ One can recognise the latter integral as a simplified version of the integral in the Duhamel formulation of the equation for $v$. For the approximation let us choose the (not particularly ambitious) discrete norm $L^2(\Pi_N)$, which has the convenient property that there is a simple isometry $\Psi_N$ from $L^2(\Pi_N)$ to $L^2(\T)$ given by mapping the discrete Fourier basis to continuum Fourier modes, which reduces the objective to estimating
\begin{equ}
    \norm[\bigg]{\int_{0}^{t} P_{t-s} (\operatorname{Id}- \Psi_N \Theta_N)f(O_s) \,ds}_{L^2(\T)}.
\end{equ}
We emphasise that the integrand itself is \textit{still} of order $O(N^{-1/2})$ due to the spatial regularity of $O$, and therefore it may seem that we have gained nothing. However, we show (see Lemma \ref{lem:space-ss} below) that the integral admits a much better bound $O(N^{-3/2-\eps})$ for any $\eps>0$. This is done by ``trading'' temporal regularity of the integral process for negative powers of $N$, which in turn is achieved via stochastic sewing ideas originating from \cite{Khoa}.
It is far from obvious that a sewing technique can trade \textit{temporal} regularity into a \textit{spatial} rate. This is one of the main technical upshots of the present paper showing that this tradeoff is indeed possible.

As for the temporal rate, the bound $M^{-1+\epsilon}$ is obtained
by similar, but more involved, techniques as in \cite{DjGK}. While this is not the main focus of the paper, it is also not a straightforward adaptation, since the employed scheme is fairly different.
Among others, one requires regularity bounds for $O$ in spaces of negative regularity on the \textit{discrete} torus (cf. Lemma \ref{lem:Otreg} below). This does not follow from standard bounds on $O$, since the restriction operator $\Theta_N$ clearly does not admit bounds in negative regularity spaces.

Finally, we mention that the obtained rate $M^{-1+\eps}+N^{-3/2+\eps}$ also holds for a larger class of linear functionals than given in \eqref{eq:Snew}. In other words, we can base our scheme on a slight perturbation $\tilde O$ of $O$.
The necessity for this generalisation stems from the observation that the covariance function of $\{\xi(h):\,h\in\mathbb{S}^{\mathrm{new}}_{M,N}\}$ is not in a closed form.
In particular, thanks to this generalisation, we show based on \cite{Davie-Gaines} that there exists a $\tilde{\mathbb{S}}^{\mathrm{new}}_{M,N}$ for which the error bound $M^{-1+\eps}+N^{-3/2+\eps}$ holds true and which admits an explicit implementable covariance function.
At the end of the manuscript we confirm the superior performance of our so-called doubly accelerated exponential Euler scheme by numerical results.

The paper is organized as follows. We introduce the setting, numerical schemes and main results in \cref{subsec:setting} below. Within this section, we distinguish between the cases of bounded nonlinearity, which is treated in \cref{subsec:setting1} with main \cref{thm:main1} and a nonlinearity with possible superlinear growth in \cref{subsec:setting2} with main \cref{thm:main2}. In \cref{sec:prelim}, we derive regularity bounds for the process $O$ solving the linear system, consider well-posedness questions and prove all needed auxiliary results. \cref{bigsec:bounded} is devoted to the proof of \cref{thm:main1}, where we treat the spatial error terms in \cref{sec:spatialerror} and the temporal error terms in \cref{sec:temporalerror}. \cref{bigsec:superlinear} is concerned with the proof of \cref{thm:main2}, where we first derive a priori estimates for the numerical scheme in \cref{subsec:apriori}. Afterwards we consider the spatial error terms in \cref{subsec:spatial} and the temporal error terms in \cref{subsec:temporal}. \cref{section:input-noise} is divided into \cref{section:tildeO}, where we construct a process $\tilde{O}$ that is employed in our numerical scheme and that satisfies the required \cref{ass:tildeO} below and in \cref{sec:implement} we detail the numerical algorithm.
In \cref{sec:numerics} the implementation is carried out on a concrete example and the numerical results confirm a significant speedup.

\subsection{Setting and main results}\label{subsec:setting}

Let $(\Omega,\F,\mathbb{P})$ be a probability space. We fix a time horizon $T>0$. The torus $\T$ is defined as $\T=\R/\Z$. Similarly, we consider a periodic version of the spatial grid $\Pi_N$ in that we identify $x_N\equiv x_0$.
The space-time white noise $\xi$ is defined as a mapping from the Borel sets $\cB([0,T]\times\mathbb{T})$ into $L^{2}(\Omega)$, such that for any collection $A_{1},\dots,A_{k}\in \cB([0,T]\times\mathbb{T})$, the vector $(\xi(A_{1}),\dots,\xi(A_{k}))$ is Gaussian with zero mean and covariance $\E[\xi(A_{i})\xi(A_{j})]=\lambda(A_{i}\cap A_{j})$, where $\lambda$ denotes the Lebesgue measure.
We also fix a filtration $\bF=(\F_{t})_{t\in[0,T]}$, such that $\F_0$ is complete, $\F_T\subset \F$, and such that for any $t\in[0,T]$, $A\in \cB([0,t]\times\mathbb{T})$, $B\in \cB([t,T]\times\mathbb{T})$, $\xi(A)$ is $\F_{t}$-measurable and $\xi(B)$ is independent of $\F_{t}$. An example for $\bF$ would be the completed filtration generated by $\xi$.
The predictable $\sigma$-algebra on $\Omega\times[0,T]$ will be denoted by $\mathcal{P}$. Stochastic integrals $\int_{0}^{T}\int_{\mathbb{T}}g(s,y)\xi(ds,dy)$ against $\xi$ can be defined for all $\mathcal{P}\otimes \cB(\mathbb{T})$-measurable integrands $g:\Omega\times[0,T]\times\mathbb{T}\to\C$ with $g\in L^{2}(\Omega\times[0,T]\times\mathbb{T})$.
We refer to \cite{DPZ} for more details, but remark that for deterministic $g$, the stochastic integral is simply the unique isometric and linear extension of the map $\mathbf{1}_A\mapsto\xi(A)$ to $L^2([0,T]\times\T)$.

Let $(P_t)_{t\geq 0}$ be the heat semigroup and let $p_t$ be the associated heat kernel. That is, denoting by $e_{j}(x)\coloneqq e^{2\pi ijx}$ the Fourier modes for $j\in\Z$ (here $i=\sqrt{-1}$) and by $\F f (j)\coloneqq\hat{f}(j)\coloneqq\int_{\mathbb{T}}f(x)e^{-2\pi i j x}\,dx$, for $j\in\Z$, the Fourier transform of $f\in L^1(\T,\C)$,
set
\begin{equ}
  P_{t}f\coloneqq\F^{-1}\big(j\mapsto e^{-4\pi^2j^2t}\hat f (j)\big)\coloneqq  \sum_{j\in\mathbb{Z}}e^{-4\pi^2j^2t}\hat f(j)e_j.
\end{equ}
One can equivalently write $P_t f=p_t\ast f$, where
\begin{equ}
p_{t}(x)\coloneqq\sum_{j\in\mathbb{Z}}e^{-4\pi^2j^2t}e^{2\pi i j x}=\frac{1}{\sqrt{4\pi t}}\sum_{m\in\Z}e^{-(x-m)^2/4t}.
\end{equ}
We define a mild solution $u$ of \eqref{eq:main} as a $\mathcal{P}\otimes\mathcal{B}(\T)$-measurable continuous process that satisfies almost surely $$u_t=P_{t}u_0 + \int_{0}^{t} P_{t-s} f(u_s) \, ds + \int_{0}^{t}\int_{\T} p_{t-s} (\cdot-y)\, \xi(ds,dy)$$ 
for $t\in[0,T]$.
The existence and uniqueness of a mild solution to \eqref{eq:main} is classical under standard assumptions on $u_0$ and $f$, see below in \cref{prop:wp}. 
The stochastic convolution will be denoted by 
\begin{align} \label{eq:OUdef}
    O_t:= \int_{0}^{t}\int_{\T} p_{t-s} (\cdot-y)\xi(ds,dy), \quad t\in[0,T]
\end{align}
and will be referred to as
the Ornstein-Uhlenbeck (OU) process.

We decompose the mild solution $u$ to \eqref{eq:main} as $u=v+O$
where for $t\in[0,T]$
\begin{align}
    v_t&:=P_{t} u_0 + \int_{0}^{t} P_{t-s} f(u_s) \,ds=P_{t} u_0 + \int_{0}^{t} P_{t-s} f(v_s+O) \,ds.
\end{align}
for $t\in[0,T]$. Therefore, $v$ is the mild solution of the PDE
\begin{align}
    \partial_t v =\Delta v + f(v+O), \quad v_0=u_0. \label{eq:v}
\end{align}
Thanks to the high regularity of $v$ (see \cref{prop:wp} below), it is actually also a strong solution. 

We fix $M,N\in\N$. Our aim is to build a numerical scheme for the equation for $v$ on the space-time lattice $I_M\times \Pi_N$, see \eqref{eq:lattice}, whose approximation rates are improved due to the improved regularity of $v$ compared to $u$.
To construct such a scheme and to transfer these improved rates to the scheme for $u$, one would like to be able to simulate the irregular noise part $O_{t}(x)$ explicitly for $(t,x)\in I_M\times \Pi_N$. However, since the covariance of $O$ on $I_M\times\Pi_N$ is not in a closed form, we instead consider an intermediate Gaussian process $\tilde{O}$ (see \cref{section:tildeO} below), for which we have an explicit expression for the covariance on the time space-grid (and therefore is genuinely implementable) and for which the remainder $O-\tilde{O}$ decays at a sufficiently high temporal and spatial rate. In our scheme, we then use $\tilde O_t(x)$ instead of $O_t(x)$
for $(t,x)\in I_{M}\times\Pi_N$.
In \cref{section:tildeO}, expanding on the method detailed in \cite{Davie-Gaines}, we construct such a process $\tilde{O}$, for which the remainder $O-\tilde{O}$ actually decays with an exponential rate under the condition that there exist $c, \epsilon>0$ with $M^{-1}\geq cN^{-2+\epsilon}$. The latter condition is harmless in our case, since our error estimate of order $M^{-1+\eps}+N^{-3/2+\eps}$ is optimised when $M^{-1}\sim N^{-3/2}\gg N^{-2+\eps}$.
At first reading, the reader might think of $O$ in place of $\tilde{O}$ throughout the article. The results remain true in this case, but the resulting scheme (to our best knowledge) is not implementable.

We summarize the required properties of $\tilde{O}$ in \cref{ass:tildeO} below. 
We define the restriction operator $\Theta_N$ acting on $C(\T)$ by
\begin{equ}\label{eq:restriction}
    \Theta_N g(x)=g(x),\qquad x\in\Pi_N,
\end{equ}
as well as the $L^2$-norm on the discrete torus $\Pi_N$ by
\begin{equ}
    \|f\|_{L^2(\Pi_N)}^2=N^{-1}\sum_{x\in\Pi_N}|f(x)|^2.
\end{equ}
In the assumptions below, we denote by $C_T^{\theta}\calC^{\beta}(\T)=C^{\theta}([0,T],\calC^{\beta}(\T))$, for $\theta \in (0,1]$ and $\beta\in\R$, 
the space of $\theta$-Hölder continuous functions on $[0,T]$ taking values in $\calC^{\beta}(\T)$, 
where $\calC^{\beta}(\T)$ denotes the H\"older-Besov space that we formally introduce in \cref{subsec:Besov-space} below. We equip the space with the norm
\begin{align*}
    \norm{u}_{C_T^{\theta}\calC^{\beta}(\T)}=\sup_{t\in[0,T]}\norm{u_t}_{\calC^{\beta}(\T)}+\sup_{0\leq s<t\leq T}\frac{\norm{u_t-u_s}_{\calC^{\beta}(\T)}}{\abs{t-s}^{\theta}}.
\end{align*}
For $\theta=0$, we let $C_T\calC^{\beta}(\T)$ be the space of continuous functions on $[0,T]$ with values in $\calC^{\beta}(\T)$ equipped with the supremum norm.

\begin{assumption}\label{ass:tildeO}
Assume that $(\tilde O_t)_{t\in[0,T]}$ is a $\mathcal{P}\otimes \cB(\mathbb{T})$-measurable process that satisfies:
    \begin{enumerate}[label=\alph*)]
        \item \label{en:b}
        for any $p\geq 1$ and $\epsilon>0$, there exists a constant $C=C(p, \epsilon)$ such that 
        \begin{align}\label{eq:O-tildeO-close}
            \max_{t\in I_{M}}\big(\E\|\Theta_N O_t-\Theta_N\tilde O_t\|_{L^2(\Pi_N)}^p\big)^{1/p}&\leq C (M^{-1+\epsilon} + N^{-\frac{3}{2}+\epsilon});
        \end{align} 
        \item \label{en:c}for any $p\geq 1$, $\lambda\in [0,1)$, and $\epsilon>0$, there exists a constant $C=C(\epsilon, p, \lambda)$ such that
        \begin{align}\label{eq:Otimespace}
        \E\norm{\tilde O}_{C_T^{\frac{\lambda}{2}}\calC^{\frac{1}{2}-\lambda-\epsilon}(\T)}^p\leq C.
    \end{align} 
    \end{enumerate}
\end{assumption}

\begin{remark}
    The rate assumed in \eqref{eq:O-tildeO-close} is simply chosen so that the replacement of $O$ by $\tilde O$ does not produce a larger error than the rest of the argument.
    In practice, in \eqref{eq:O-tildeO-close} a much faster rate is achievable: the construction from \cite{Davie-Gaines} that we detail in \cref{sec:implement} below yields $\tilde O$ with exponentially small error to $O$.
\end{remark}
Our assumption on the initial condition is as follows.

\begin{assumption}\label{asn:u0}
 Assume that $u_0$ is an $\F_0$-measurable random variables with values in $C(\T)$ and that for all $p\geq 1$, $\eps>0$ there exists $\mathscr{M}(p,\eps)<\infty$ with $\E\norm{u_0}_{\calC^{5/2-\epsilon}(\T)}^p\leq \mathcal{M}(p,\eps)$.
\end{assumption}
\begin{remark}
    The assumed regularity on the initial condition is relatively high, but this is merely a choice of convenience: it would be a fairly straightforward but lengthy task to modify the article to accommodate initial conditions of the form $u_0+Z_0$, where $u_0$ satisfies \cref{asn:u0}, $Z_0$ is $\cF_0$-measurable and $Z_0\stackrel{\mathrm{law}}{=}O_1$.
    This is a natural class of initial conditions since the solution (even if starting from a merely continuous initial condition) admits such a decomposition at any positive time.
    In order to not overwhelm the already heavy notation of the article, we refrain from this generalisation.
\end{remark}

\subsubsection{Bounded nonlinearity} \label{subsec:setting1}

We first study \eqref{eq:main} with nonlinearities $f$ that are sufficiently regular and bounded. More precisely, we work under the following assumption. Here and in the following we use the convention $\partial^0 f=f$.
\begin{assumption}\label{ass1}
Assume that $f\in C^{2}_{b}(\R)$, i.e. there exists $K>0$ such that $|\partial^i f(x)|\leq K$ for all $i=0,1,2$ and $x\in\R$.
\end{assumption}

We start by defining a discrete semigroup $\tilde{P}^N$ acting on functions on $\Pi_N$, henceforth denoted by $C(\Pi_N)$ (although continuity is of course not an interesting notion on $\Pi_N$).
Recall that the functions $(e_j)_{j\in\Z}$ are orthonormal eigenfunctions of $\Delta$ with eigenvalues $\lambda_j=-4\pi^2j^2$. 
Denote for $j\in\Z$, $e_j^N=\Theta_N e_j$.
The collection $(e_j^N)_j$ are of course not orthonormal since for any $k\in\Z$, the functions $e_j^N$ and $e_{j+kN}^{N}$ are actually equal.
However, it holds that
\begin{equ}
\langle e_j^N,e_\ell^N\rangle_{L^2(\Pi_N,\C)}\coloneqq N^{-1}\sum_{x\in\Pi_N}e_j^N(x)\overline{e_\ell^N}(x)=\mathbf{1}_{j-\ell=0\,\,\mathrm{mod}\,N}.
\end{equ}
Therefore, we restrict the index set to $J_N\coloneqq\{-\lfloor N/2\rfloor,\ldots, \lceil N/2\rceil-1\}$, where for $r\in\R$, $\lfloor r\rfloor$ denotes the largest integer $n\leq r$, respectively $\lceil r\rceil$ denotes the smallest integer $n\geq r$. Then the collection $(e_j^N)_{j\in J_N}$ is an orthonormal basis of $L^2(\Pi_N,\C)$. 
We then define, for $t\geq 0$, the linear operator $\tilde P^N_t$ by its action on the base functions given by
\begin{equ}\label{eq:discrete-semigroup}
\tilde P^N_t e_j^N=e^{-4j^2\pi^2 t}e_j^N,\qquad j\in J_N.
\end{equ}
 Since $C(\Pi_N)$ is finite dimensional, it is straightforward that the family $(\tilde P^N)_{t\geq 0}$ forms an analytic semigroup on it, and its infinitesimal generator $\tilde\Delta_N$  is given by
\begin{equ}\label{eq:tilde-Laplace}
    \tilde\Delta_N e_j^N=-4j^2\pi^2 e_j^N,\qquad j\in J_N.
\end{equ}

We start with a spatial discretisation:
let $v^N$ be the solution of
\begin{align}\label{eq:tildevN}
    \partial_t v^N=\tilde{\Delta}_N v^N + f(v^N+\Theta_N O), \quad v^{N}(0,x)=\Theta_N u_0,
\end{align}
which is given in the mild form by
\begin{align}
    v^N_t=\tilde{P}_t^N \Theta_N u_0 + \int_0^t \tilde{P}^N_{t-s}f(v^N_s+\Theta_N O_s) \,ds, \quad t\in[0,T].
\end{align}
For the well-posedness of \eqref{eq:tildevN} we refer to \cref{lem:wpvNzN} below.
Next, we introduce a space-time discrete process that plays only an auxiliary role, as it is not implementable.
This is the exponential Euler scheme for the equation for $v^N$: that is, we define $V^{M,N}$ inductively for $k=1,\ldots,M$ by
\begin{equ}\label{eq:V-recursion}
    V^{M,N}_{t_{k}}=\tilde P^N_hV^{M,N}_{t_{k-1}}+(\tilde \Delta_N)^{-1}(\tilde P^N_h-\mathrm{Id})\Big(f(V^{M,N}_{t_{k-1}}+\Theta_N O_{t_{k-1}}\big)\Big),\quad V^{M,N}_0=\Theta_N u_0.
\end{equ}
Above and below $\operatorname{Id}$ denotes the identity operator.
The scheme that we study simply replaces $O$ by $\tilde O$ in the above recursion: we define $\tilde{V}^{M,N}$ inductively for $k=1,\ldots, M$ by
\begin{equ}\label{eq:Vtilde-recursion}
    \tilde{V}^{M,N}_{t_{k}}=\tilde P^N_h\tilde{V}^{M,N}_{t_{k-1}}+(\tilde \Delta_N)^{-1}(\tilde P^N_h-\mathrm{Id})\Big(f(\tilde{V}^{M,N}_{t_{k-1}}+\Theta_N \tilde O_{t_{k-1}}\big)\Big),\quad \tilde{V}^{M,N}_0=\Theta_N u_0.
\end{equ}
Note that in both \eqref{eq:V-recursion} and \eqref{eq:Vtilde-recursion},  $\tilde \Delta_N$ is invertible only on $\mathrm{span}\{e_j^N:\,j\in J_N\setminus\{0\}\}$, therefore we set by convention  $(\tilde \Delta_N)^{-1}(\tilde P^N_h-\mathrm{Id})e_0^N\coloneqq h e_0^N$.
With this convention, one can write $V^{M,N}$ and $\tilde{V}^{M,N}$ in the mild form as 
\begin{align}
    V^{M,N}_t&=\tilde P^N_{t}(\Theta_N u_0)+\int_0^t\tilde P^N_{t-s}\Big(f\big(V^{M,N}_{\kappa_M(s)}+\Theta_N O_{\kappa_M(s)}\big)\Big)\,ds,\label{eq:V-mild}\\
   \tilde{V}^{M,N}_t&=\tilde P^N_{t}(\Theta_N u_0)+\int_0^t\tilde P^N_{t-s}\Big(f\big(\tilde{V}^{M,N}_{\kappa_M(s)}+\Theta_N \tilde{O}_{\kappa_M(s)}\big)\Big)\,ds,\label{eq:tildeV-mild}
\end{align}
where ${\kappa_M(s)}=\lfloor sh^{-1}\rfloor h=\sup\{t\in I_M:\,t\leq s\} $.  
It is easy to check that \eqref{eq:V-recursion} and \eqref{eq:V-mild} (respectively, \eqref{eq:Vtilde-recursion} and \eqref{eq:tildeV-mild}) indeed coincide for $t\in I_M$, and for all other $t$ we take \eqref{eq:V-mild} -\eqref{eq:tildeV-mild} as the definition of  $V^{M,N}_t$ and $\tilde{V}^{M,N}_{t}$.
The fully discretised numerical scheme for the solution $u=v+O$ is then defined by $$\tilde{V}^{M,N}+\Theta_N \tilde{O}.$$
In comparison to the accelerated exponential Euler scheme from \cite{Jentzen-Kloden}, we will refer to our method as the \textit{doubly accelerated exponential Euler scheme}, indicating that the principle of exact sampling leading to higher order rates is employed to speed up both temporal and spatial rates. 
Note  that $$
\Theta_N u-(\tilde{V}^{M,N}+\Theta_N \tilde{O})=\Theta_N v - \tilde{V}^{M,N}+\Theta_N(O-\tilde O).
$$
The last term is estimated by Assumption~\ref{ass:tildeO} \ref{en:b}, and so the main goal is to bound the error
 $\Theta_N v-\tilde{V}^{M,N}$ coming from the remainder equation.
Note also that the scheme only inputs $(M+1)N$ many random variables given by $\tilde{O}_{t_k}(x_j)$, $k=0,\ldots,M$, $j=0,\ldots,N-1$. 
We can now formulate our first main theorem, whose proof is given in \cref{bigsec:bounded}.
\begin{remark}
    To ease notation, we use the following convention: whenever $\theta$ is some collection of parameters, expressions of the form $C=C(\theta,\cM)$ mean that there exist $p^*,\eps^*$ depending on $\theta$ such that the constant $C$ depends only on $\theta$ and the quantity $\cM(p^*,\eps^*)$.
\end{remark}
\begin{theorem}\label{thm:main1}
  Let $p\geq 1$, $\eps\in (0, 1/2)$. Let \cref{ass:tildeO}, \cref{asn:u0}, and \cref{ass1} hold. Then there exists a constant $C=C(p,\epsilon, T, K, \mathcal{M})$, that does not depend on $M,N$, such that
   \begin{align}
       \big(\E \sup_{t\in[0,T]}\norm{ \Theta_N u_{t}- (\tilde{V}^{M,N}_t + \Theta_N \tilde{O}_t)}_{L^2(\Pi_N)}^p\big)^{1/p}\leq C (M^{-1+\epsilon}+ N^{-\frac{3}{2} + \epsilon}).
   \end{align}
\end{theorem}

\subsubsection{Nonlinearity with superlinear growth }\label{subsec:setting2}

To handle nonlinearities such as   the Allen-Cahn nonlinearity $f(x)=x-x^3$, a different scheme is required. Indeed, it is known that the explicit (exponential) Euler scheme diverges in $L^p(\Omega)$ for superlinearly growing nonlinearities (\cite{HutzJentz,beccari2019strong}).
We therefore combine the doubly accelerated exponential Euler scheme with a splitting scheme to avoid such blow-ups, which allows to
cover a large class of growing nonlinearities. The scope of the allowed nonlinearities is as follows.
\begin{assumption}\label{ass2}
    Assume that $f\in C^3(\R)$ such that there exist $K\geq 1$ and $m\geq 0$, such that for all $i=0,\dots, 3$ and $x\in\R$,
    \begin{align}
       \abs{\partial^{i} f(x)} &\leq K(1+\abs{x}^{(2m+1-i)\vee 0})\\
        \partial f(x)&\leq K.
    \end{align}
\end{assumption}
\begin{remark}
       \cref{ass2} implies a local Lipschitz bound with polynomial growth and a global one-sided Lipschitz bound. That is, for all $x,y\in\R$ one has 
    \begin{align}
        |f(x)-f(y)|&\leq K(1+|x|^{2m}+|y|^{2m})|x-y|,
    \\
    (x-y)(f(x)-f(y))&\leq K |x-y|^2.\label{eq:onesidedLip}
    \end{align}
\end{remark}

Compared to the case of bounded nonlinearity, here we consider a different spatial discretisation of the Laplacian:
instead of the operator $\tilde{\Delta}_N$ from \eqref{eq:tilde-Laplace}, we consider the classical discrete Laplacian $\Delta_N$ defined by 
\begin{align}
    \Delta_N v(x_n)= \partial^{-}_N \partial^{+}_{N}v (x_n),\quad n\in \{0,\dots, N-1\},\quad v\in C(\Pi_N)
\end{align}
with 
\begin{equ}\label{eq:discrete-derivatives}
    \partial^{-}v (x_n) = N(v(x_{n})-v(x_{n-1})), \quad \partial^{+}v (x_n) = -N(v(x_{n})-v(x_{n+1})).
\end{equ}
Indeed, it turns out that the classical discrete Laplacian is more handy in terms of the variational energy estimates, which motivates this choice.
The functions $e_j^N$ are also eigenfunctions of $\Delta_N$, with eigenvalues $\lambda_j^N\coloneqq-4N^2 \sin^2(\frac{j\pi}{N})$.
We then set $P_t^N \coloneqq e^{t\Delta_N}$, $t\geq 0$, which is the semigroup generated by $\Delta_N$.
Its action on the base functions are given by
\begin{equ}
    P^N_t e_j^N=e^{\lambda_j^N t}e_j^N,\qquad j\in  J_N.
\end{equ}
The spatial discretisation is then given by $w^N$, defined as the solution of
\begin{align}\label{eq:bfvN}
    \partial_t w^N = \Delta_N w^N + f(w^N+\Theta_N O), \quad w^N(0,x)=\Theta_N u_0.
\end{align}
One can write \eqref{eq:bfvN} in the mild form as
\begin{align}
    w^N_t=P_t^N \Theta_N u_0+ \int_0^t P^N_{t-s}f(w^N_s+\Theta_N O_s)\, ds, \quad t\in[0,T].
\end{align}
For the well-posedness of \eqref{eq:bfvN} we refer to \cref{lem:apriori-WMN} below.
For the splitting scheme used in the temporal discretisation we introduce the function $\big(\Phi_t(x)\big)_{t\geq 0,x\in\R}$ as the solution flow of the ODE with nonlinearity $f$. It is defined as the solution to the ODE
\begin{equ}\label{eq:Phi}
    \partial_t \Phi_t(x)= f(\Phi_t(x)),\quad \Phi_0(x)=x.
\end{equ}
\begin{remark}
Often the ODE \eqref{eq:Phi} admits an explicit solution. For example in the Allen-Cahn case $f(x)=x-x^3$ one has
\begin{align}
    \Phi_t(x)=\mathrm{sgn}(x)\frac{e^t}{\sqrt{x^{-2}-1+e^{2t}}}.
\end{align}
\end{remark}
Similarly to the case of bounded nonlinearity, we first introduce auxiliary processes $X^{M,N}$ and $Y^{M,N}$.
They are defined inductively through the following splitting scheme for \eqref{eq:bfvN}:
To start the iteration we set $X^{M,N}_{0}=\Theta_N u_0$. For $k=1,\dots, M$, we set
\begin{align}   Y^{M,N}_{t_k}&=\Phi_h\big(X^{M,N}_{t_{k-1}} + \Theta_N O_{t_{k-1}}\big)- \Theta_N O_{t_{k-1}}, \label{eq:Y}\\
    X^{M,N}_{t_{k}}&= P^{N}_h Y^{M,N}_{t_k}. \label{eq:XY}
\end{align}

The scheme we study simply replaces $O$ by $\tilde O$ in the above recursion. That is, again we define inductively two processes $\tilde{X}^{M,N}$ and $\tilde{Y}^{M,N}$. To start the induction we set $\tilde{X}^{M,N}_{0}=\Theta_N u_0$. For $k=1,\dots, M$, we set
\begin{align}   \tilde{Y}^{M,N}_{t_k}&=\Phi_h\big(\tilde{X}^{M,N}_{t_{k-1}} + \Theta_N \tilde{O}_{t_{k-1}}\big)- \Theta_N \tilde{O}_{t_{k-1}}, \label{eq:tY}\\
    \tilde{X}^{M,N}_{t_{k}}&= P^{N}_h \tilde{Y}^{M,N}_{t_k}. \label{eq:tXY}
\end{align}

These recursions can also be written in a mild form. Indeed, letting first
\begin{equ}\label{eq:gh}
    g_t(x)=\frac{\Phi_t(x) - x}{t},\quad t>0,\quad x\in\R,
\end{equ}
together with the convention $g_0(x)=f(x)$,
we first rewrite the recursions as
\begin{equs}
    X^{M,N}_{t_{k}}&= P^N_h X^{M,N}_{t_{k-1}}+h P_{h}^N \Big(g_h\big(X^{M,N}_{t_{k-1}}+\Theta_N O_{t_{k-1}}\big)\Big),\\
    \tilde{X}^{M,N}_{t_{k}}&= P^N_h\tilde{X}^{M,N}_{t_{k-1}}+h P_{h}^N \Big(g_h\big(\tilde{X}^{M,N}_{t_{k-1}}+\Theta_N \tilde{O}_{t_{k-1}}\big)\Big).\label{eq:XMN}
\end{equs}
Once again, it is easy to check that these inductive forms agree with the mild forms
\begin{equs}
    X^{M,N}_t &=  P^{N}_t \Theta_N u_{0} + \int_{0}^{t}  P_{t-k_{M}(s)}^{N} \Big( g_h\big(X^{M,N}_{k_{M}(s)}+\Theta_N O_{k_{M}(s)}\big)\Big)\,ds,\label{eq:auxscheme}
\\
    \tilde{X}^{M,N}_t &=  P^{N}_t \Theta_N u_{0}+ \int_{0}^{t}  P_{t-k_{M}(s)}^{N} \Big( g_h\big(\tilde{X}^{M,N}_{k_{M}(s)}+\Theta_N \tilde{O}_{k_{M}(s)}\big)\Big)\,ds\label{eq:scheme1-mild}
\end{equs}
on the time gridpoints $t=t_0,t_1,\ldots,t_M$, and for general $t\in[0,T]$ we take \eqref{eq:auxscheme}-\eqref{eq:scheme1-mild} as the definitions of
$X^{M,N}_t$ and $\tilde{X}^{M,N}_t$.

Similarly to before, the numerical scheme for $u=v+O$ is then defined by
$$\tilde{X}^{M,N}+\Theta_N \tilde O.$$ 
Our second main theorem then reads as follows, whose proof will be given in \cref{bigsec:superlinear}.
\begin{theorem}\label{thm:main2}
   Let $p\geq 1$, $\eps\in (0, 1/2)$. Let \cref{ass:tildeO}, \cref{asn:u0}, and \cref{ass2} hold. Then there exists a constant $C=C(p,\epsilon, T, K, \mathcal{M})$, that does not depend on $M,N\in\N$, such that 
   \begin{equ}
        \Big(\E\sup_{t\in[0,T]}\big\|\Theta_N u_{t}-(\tilde{X}^{M,N}_{t}+\Theta_N \tilde{O}_{t})\big\|_{L^{2}(\Pi_N)}^p\Big)^{1/p}\leq C(M^{-1+\eps}+N^{-\frac{3}{2}+\eps}).
    \end{equ}
\end{theorem}

\section{Preliminaries}\label{sec:prelim}

The proofs of \cref{thm:main1} and \cref{thm:main2} rely on a number of function space estimates, in this section we collect all of them.
These preliminaries are subdivided in several subsections. In \cref{subsec:Besov-space} we introduce Sobolev and Besov spaces on the continuum and in the discrete and state embedding, product, and norm equivalence statements.
Estimates on $O$, the stochastic sewing lemma and a version of Kolmogorov's continuity theorem are stated or proven in \cref{subsec:sewing}.
The properties of the restriction operator $\Theta_N$ and its (in an appropriate sense) inverse are collected in \cref{subsec:disc-op}.
Semigroup estimates for the different semigroups appearing in the equation and its discretisations  are found in \cref{subsec:Semigroup}.  In \cref{subsec:wp} we collect regularity estimates for $u$ and the remainder $v=u-O$. In \cref{subsec:aux} we collect some auxiliary lemmata that do not fit in the other subsections, such as estimates on $g_h$ from \eqref{eq:gh} and on the composition of regular functions with Besov or Sobolev functions.

Although many of the results in this section will be familiar to most readers, let us highlight two estimates that are perhaps less standard and relate to the choice of spatial discretisation employed by our scheme.
\cref{lem:operator} provides an error estimate on the discretisation viewed as an operator on functions on $\T$. This estimate gives the same error order as a standard Fourier truncation (employed as the spatial discretisation in, e.g. \cite{Jentzen-Kloden, Jentzen12,BGJK, Wang,becker2023strong, DjGK}), but with more restrictions on the exponents. \cref{lem:Otreg} is an estimate of $\Theta_NO$ in Sobolev spaces of negative regularity. This is crucial to retain the temporal rate $1$ from \cite{DjGK} (see Remark \ref{rem:rates}), but it does not follow from bounds on $O$ as $\Theta_N$ does not retain negative regularity norms. In fact, our estimates, if sharp, would suggest that $\Theta_N O_t$ does not satisfy all the estimates $O$ does, but rather that there is a tradeoff between the time increments considered and the spatial mesh $N^{-1}.$

In this section and throughout the paper we often use the notation $A\lesssim B$, which means that there exists a constant $C$ with $A\leq C B$. The dependence of the constant $C$ will be clear from the context.

\subsection{Function spaces and their basic properties}\label{subsec:Besov-space}
We define an extension operator $\Psi_N:C(\Pi_N)\to C(\T)$ by setting
\begin{equ}
 \Psi_N e_j^N=e_j,\quad j\in J_N,   
\end{equ}
and extending it by linearity.
Denote by $C_N(\T)$ the image of $\Psi_N$, that is
 $C_N(\T)=\mathrm{span}\{e_j:\,j\in J_N\}$.
Denote by $L^2_N(\T)$ the space $C_N(\T)$ equipped with the $L^2(\T)$ norm.
Define the continuous and discrete Sobolev norms for $\beta\in\R$ by
\begin{equ}
    \norm{v}_{H^{\beta}(\T)}^2=\sum_{j\in \Z}\abs{\langle v, e_{j}\rangle_{L^{2}(\T)}}^2 (1+\abs{j})^{2\beta},
\end{equ}
\begin{equ}
    \norm{v}_{H^{\beta}(\Pi_N)}^2=\sum_{j\in J_N}\abs{\langle v, e_{j}^{N}\rangle_{L^{2}(\Pi_N)}}^2 (1+\abs{j})^{2\beta},
\end{equ}
and denote by $H^\beta_N(\T)$ the space $C_N(\T)$ equipped with the $H^\beta(\T)$ norm. 
Note that the $L^2$ product extends to a duality between $H^\beta$ and $H^{-\beta}$, both in the case of $\T$ and $\Pi_N$.

Next, we define Besov spaces on the torus $\T$ and discrete torus $\Pi_N$.
For the definitions, we stick to the ones from \cite{GS}.
First we choose a partition of unity as follows.
Fix $\eps_0\in(0,1/10)$ and take a smooth even function $\phi^0:\mathbb{R}\to [0,1]$, such that $\phi^0|_{B_{1-\eps_0}}\equiv 1$ and $\supp \phi^0\subset B_{1}$, where $B_r=\{x\in\R:\,|x|\leq r\}$.
Set $\rho_0=\frac{3-2\eps_0}{2}$.
For $\rho\in[1,2]$ define $\phi_\rho^0(x)=\phi^0\big((\rho_0/\rho)x\big)$.
Further, set for $k\in\N$,
\begin{equ}
\phi^{k}_\rho(x)= \phi^{0}_\rho(2^{-k}x)-\phi^{0}_\rho(2^{-k+1}x)=\phi_\rho^1(2^{-k+1}x)\ .
\end{equ}
For any $\rho\in[1,2]$, the functions $\{\phi^{k}_\rho\}_{k\in \mathbb{N}}$ form a partition of unity.
Furthermore, for any Schwartz distribution $f\in \mathcal{S}'(\T)$, we define the Littlewood-Paley blocks
\begin{equ}
f^{[k]}=\sum_{j\in \Z}\phi^k_1(j)\scal{f, e_j}e_j.
\end{equ}
We then define the (inhomogeneous) Besov spaces via the norms
\begin{equ}
\|f\|_{B^{\alpha}_{p,q}(\T)}=\big\|k\mapsto 2^{\alpha k}\|f^{[k]}\|_{L^p(\T)}\big\|_{\ell^q}.
\end{equ}
Replacing $\phi^k_1$ with $\phi^k_\rho$ results in an equivalent norm for any $\rho\in[1,2]$. For the definition of discrete Besov spaces, we will choose a convenient $\rho$ following \cite{GS}.

Indeed, for the discrete analogues of Besov norms, for any integer $N\geq 2$, we now define $K_N=\lfloor \log_2 (N/2)\rfloor$ and $\rho_N=N2^{-K_N-1}\in[1,2]$.
For $f\in C(\Pi_N)$, we then define the discrete Littlewood-Paley blocks 
\begin{equ}
f^{[k],N}=\sum_{j\in J_N} \phi^k_{\rho_N}(j)\scal{f, e_j^N}e_j^N
\end{equ}
and the discrete Besov norms
\begin{equ}
\|f\|_{B^{\alpha}_{p,q}(\Pi_N)}=\big\|k\mapsto 2^{\alpha k}\|f^{[k],N}\|_{L^p(\Pi_N)}\big\|_{\ell^q}.
\end{equ}
We refer to \cite[Remark 1.5]{GS} in order to motivate the specific choice of $\rho_N$ in the partition of unity.
Two important special cases of parameters are $p=q=2$ and $p=q=\infty$. In the latter case, we use the shorthand $\cC^\alpha\coloneqq B^\alpha_{\infty,\infty}$ both in the discrete and continuous cases. 
We summarize some of their properties in the following lemma.
\begin{lemma}\label{lem:spaces-equivalence}
    \begin{enumerate}
        \item[a)] For any $\alpha\in\R$, the spaces $B^\alpha_{2,2}(\T)$ and $H^\alpha(\T)$ coincide and the norms are equivalent. Similarly, $B^\alpha_{2,2}(\Pi_N)$ and $H^\alpha(\Pi_N)$  norms are equivalent and the equivalence constant does not depend on $N$.
        \item[b)] For $\alpha\in(0,1)$, the space $\cC^\alpha(\T)=B^\alpha_{\infty,\infty}(\T)$ coincides with the space of $\alpha$-H\"older continuous functions and the norms are equivalent. Similarly, the $\cC^\alpha(\Pi_N)=B^\alpha_{\infty,\infty}(\Pi_N)$-norm is equivalent to the norm
        \begin{equ}
        C(\Pi_N)\ni u\,\mapsto\,    \sup_{x\in\Pi_N}|u(x)|+\sup_{x\neq y\in \Pi_N}\frac{|u(x)-u(y)|}{|x-y|^\alpha}
        \end{equ}
       and the equivalence constant does not depend on $N$.
        \item[c)] For $\alpha\in (0,1)$, the Sobolev norm in $H^{\alpha}(\T)$ is equivalent to the Slobodeckij-norm 
        \begin{align*}
            \norm{u}_{L^2(\T)} + \paren[\bigg]{\int_{\T}\int_{\T} \frac{\abs{u(x+y)-u(x)}^2}{\abs{y}^{1+2\alpha}} dx dy}^{1/2}.
        \end{align*}
        Similarly, the norm in $H^{\alpha}(\Pi_N)$ is equivalent to 
        \begin{align*}
            \norm{u}_{L^2(\Pi_N)} + \paren[\bigg]{N^{-2}\sum_{x\in\Pi_N}\sum_{y\in\Pi_N} \frac{\abs{u(x+y)-u(x)}^2}{\abs{y}^{1+2\alpha}}}^{1/2}
        \end{align*}
        and the equivalence constant does not depend on $N$.
    \item[d)] Let $1 \le p_1 \le p_2 \le \infty$, $1 \le q_1 \le q_2 \le \infty$, and $\alpha \in \mathbb{R}$. Then there exists a constant $C = C(p_1,p_2)$ such that for $f \in C(\Pi_N)$,
$$\|f\|_{B^{\alpha - (1/p_1 - 1/p_2)}_{p_2,q_2}(\Pi_N)}
\le
C\|f\|_{B^{\alpha}_{p_1,q_1}(\Pi_N)},
$$
that is $B^{\alpha}_{p_1,q_1}(\Pi_N)\hookrightarrow B^{\alpha - (1/p_1 - 1/p_2)}_{p_2,q_2}(\Pi_N)$.
    \item[e)] Moreover for $\alpha\in\R$ and $\epsilon>0$, $\cC^{\alpha+\eps}(\T)\hookrightarrow H^\alpha(\T)$ and $\cC^{\alpha+\eps}(\Pi_N)\hookrightarrow H^\alpha(\Pi_N)$.
\end{enumerate}
\end{lemma}
\begin{proof}
    Claim  a) follows immediately from the definitions. Claim b) is contained in \cite[Section~2.7]{Bahouri2011} in the case of $\T$ and can be found in \cite[Lemma~2.2]{GS} in the case of $\Pi_N$. Claim c) follows from a) and \cite[section 2.5.1, Theorem and Remark 4]{triebel78} in the case of $\T$. 
    In the case of $\Pi_N$, c) follows from a similar argument as in the continuous case, for the sake of completeness we provide a proof. Recall that $u=\sum_{j\in J_N}\langle u, e_j^N\rangle_{L^2(\Pi_N)}e_j^N$, and for $y\in \Pi_N$, $\langle u(\cdot + y), e_j^N\rangle_{L^2(\Pi_N)}= \langle u, e_j^N\rangle_{L^2(\Pi_N)} e^{2\pi i jy}$. Thus, we can write
    \begin{align}\label{eq:repre}
        N^{-2}\sum_{x\in\Pi_N}\sum_{y\in\Pi_N} \frac{\abs{u(x+y)-u(x)}^2}{\abs{y}^{1+2\alpha}}&= N^{-1} \sum_{y\in\Pi_N} \frac{\norm{u(\cdot+y)-u(\cdot)}_{L^2(\Pi_N)}^2}{\abs{y}^{1+2\alpha}}\nonumber
        \\&= \sum_{j\in J_N\setminus\{0\}}  \abs{\langle u, e_j^N\rangle_{L^2(\Pi_N)}}^2 N^{-1} \sum_{y\in\Pi_N\setminus\{0\}} \frac{\abs{1-e^{2\pi ijy}}^2}{\abs{y}^{1+2\alpha}}.
    \end{align}
    To prove the claim it suffices to show that for $j\in J_N\setminus\{0\}$,
    \begin{align}\label{eq:double-bound}
        \abs{j}^{2\alpha}\lesssim N^{-1}\sum_{y\in\Pi_N\setminus\{0\}} \frac{\abs{1-e^{2\pi ijy}}^2}{\abs{y}^{1+2\alpha}}
        \lesssim \abs{j}^{2\alpha}.
    \end{align}   
    The upper bound in \eqref{eq:double-bound} follows from 
    \begin{align*}
        N^{-1}\sum_{y\in\Pi_N\setminus\{0\}} \frac{\abs{1-e^{2\pi ijy}}^2}{\abs{y}^{1+2\alpha}} \lesssim \int_{\mathbb{T}} \frac{\abs{1-e^{2\pi ij x}}^2}{\abs{x}^{1+2\alpha}} dx \leq \abs{j}^{2\alpha}\int_{\R} \frac{\abs{1-e^{2\pi iz}}^2}{\abs{z}^{1+2\alpha}} dz\lesssim \abs{j}^{2\alpha},
    \end{align*}
    where the integral is finite since for $\abs{z}\leq 1$, $\abs{1-e^{2\pi ij z}}^2= 4 \sin^2 (\pi z)\leq 4 z^2$ and $\alpha<1$ and for $\abs{z}\geq 1$, we may use that $\abs{1-e^{2\pi ij z}}^2\lesssim 1$ and $\alpha>0$.   
The lower bound in \eqref{eq:double-bound} can be inferred as follows.
We have that for $\abs{y}\leq 1/2$,
$|1-e^{2\pi iy}|^2=4\sin^2(\pi y)\geq 4 y^2$,
which gives, restricting the sum to $y=\frac{k}{N}$ with $k\leq \frac{N}{2|j|}$,
\begin{align*}
N^{-1}\sum_{y\in\Pi_N\setminus\{0\}} \frac{\abs{1-e^{2\pi ijy}}^2}{\abs{y}^{1+2\alpha}}
&\gtrsim
\frac{j^2}{N^3}\sum_{1\leq k\leq \frac{N}{2|j|}}
\left(\frac{k}{N}\right)^{-1-2\alpha}k^2\\
&=
j^2N^{2\alpha-2}
\sum_{1\leq k\leq \frac{N}{2|j|}}k^{1-2\alpha}
\\&\gtrsim j^2N^{2\alpha-2}
\left(\frac{N}{|j|}\right)^{2-2\alpha}
=
|j|^{2\alpha},
\end{align*}
using that $\sum_{1\leq k\le M}k^{1-2\alpha}\geq M^{2-2\alpha}$, which holds true since for $\alpha\leq\frac{1}{2}$,
$\sum_{1\leq k\le M}k^{1-2\alpha}=1+\sum_{2\leq k\le M}k^{1-2\alpha} \geq 1+\int_{1}^{M}x^{1-2\alpha} dx\geq M^{2-2\alpha}$ and for $\alpha>\frac{1}{2}$, $\sum_{1\leq k\leq M} k^{1-2\alpha}\geq M M^{1-2\alpha}=M^{2-2\alpha}$.
Combining \eqref{eq:repre} with \eqref{eq:double-bound}, we see that c) is true in the discrete case. Claim d) can be found in \cite[Lemma~2.8]{GS}. Claim e) follows from summability of $(2^{-2\epsilon k})_k$ for $\epsilon>0$ and $\norm{f^{[k],N}}_{L^2}\leqslant \norm{f^{[k],N}}_{L^{\infty}}$, both in the case of $\T$ and $\Pi_N$. 
\end{proof}
Next, we collect some useful product estimates.
\begin{lemma}\label{lem:product-est}
Let $\vartheta > 0$.  Then there exists a constant $C=C(\vartheta)$, such that for all $u, v\in L^{\infty}(\T)\cap H^{\vartheta}(\T) $ one has
\begin{align}\label{eq:product1}
    \norm{uv}_{H^{\vartheta}(\T)}\leq C \big(\norm{u}_{L^{\infty}(\T)}\norm{v}_{H^{\vartheta}(\T)}+ \norm{v}_{L^{\infty}(\T)}\norm{u}_{H^{\vartheta}(\T)}\big).
\end{align}
If furthermore $\eps>0$, then there exists a constant $C=C(\epsilon,\vartheta)$ such that
 for all $v\in H^{\vartheta}(\T)$ and $u\in\calC^{\vartheta+\epsilon}(\T)$ one has 
\begin{align}\label{eq:product}
    \norm{uv}_{H^{\vartheta}(\T)}\leq C \norm{u}_{\calC^{\vartheta+\epsilon}(\T)}\norm{v}_{H^{\vartheta}(\T)}.
\end{align}
\end{lemma}
\begin{proof}
    The first bound \eqref{eq:product1} follows from \cite[Corollary 2.86]{Bahouri2011}. If $\vartheta>1/2$, \eqref{eq:product} follows from \eqref{eq:product1} and the embeddings $\calC^{\vartheta+\epsilon}(\T)\hookrightarrow H^{\vartheta}(\T)$ and $H^{\vartheta}(\T)\hookrightarrow L^{\infty}(\T)$. For $\vartheta\in (0,1)$, we can infer \eqref{eq:product} from \cref{lem:spaces-equivalence} b) and c) as follows 
    \begin{align*}
        \norm{uv}_{H^{\vartheta}(\T)}&\lesssim \norm{uv}_{L^{2}(\T)}+\paren[\bigg]{\int_{\T}\int_{\T} \frac{\abs{(u(x+y)-u(x))v(x) + (v(x+y)-v(x))u(x+y)}^2}{\abs{y}^{1+2\vartheta}}\, dx\,dy}^{1/2}
        \\&\lesssim \norm{u}_{L^{\infty}(\T)}\norm{v}_{L^{2}(\T)} + \norm{u}_{\calC^{\vartheta+\epsilon}(\T)} \norm{v}_{L^2(\T)} + \norm{v}_{H^{\vartheta}(\T)}\norm{u}_{L^{\infty}(\T)}
        \\&\lesssim \norm{u}_{\calC^{\vartheta+\epsilon}(\T)}\norm{v}_{H^{\vartheta}(\T)}.\qedhere
    \end{align*}
\end{proof}

\begin{lemma}\label{lem:productestnegative}
   Let $\theta\in (0,1)$ $\epsilon>0$. Then there exists a constant $C=C(\epsilon,\theta)$, such 
    that for all $u\in H^{-\theta}(\Pi_N)$ and $v\in\calC^{\theta+\epsilon}(\Pi_N)$ 
   one has
   \begin{align}\label{eq:neg-product}
       \norm{uv}_{H^{-\theta}(\Pi_N)}\leq C\norm{u}_{H^{-\theta}(\Pi_N)}\norm{v}_{\calC^{\theta+\epsilon}(\Pi_N)}.
   \end{align}
\end{lemma}
\begin{proof}
    First we claim that if $\varphi\in H^{\theta}(\Pi_N)$,then
    $$\norm{v\varphi}_{H^{\theta}(\Pi_N)}\lesssim \norm{v}_{\calC^{\theta+\epsilon}(\Pi_N)}\norm{\varphi}_{H^{\theta}(\Pi_N)}.$$ 
    This follows similar as for the proof of \eqref{eq:product} above. Indeed from \cref{lem:spaces-equivalence} b) and c) we find
    \begin{align*}
        \norm{v\varphi}_{H^{\theta}(\Pi_N)}&\lesssim \norm{v\varphi}_{L^{2}(\Pi_N)}+\paren[\bigg]{N^{-2}\sum_{x\in\Pi_N}\sum_{y\in\Pi_N} \frac{\abs{[v(x+y)-v(x)]\varphi(x) + [\varphi(x+y)-\varphi(x)]v(x+y)}^2}{\abs{y}^{1+2\theta}}}^{1/2}
        \\&\lesssim \norm{v}_{L^{\infty}(\Pi_N)}\norm{\varphi}_{L^{2}(\Pi_N)} + \norm{v}_{\calC^{\theta+\epsilon}(\Pi_N)} \norm{\varphi}_{L^2(\Pi_N)} + \norm{\varphi}_{H^{\theta}(\Pi_N)}\norm{v}_{L^{\infty}(\Pi_N)}
        \\&\lesssim \norm{v}_{\calC^{\theta+\epsilon}(\Pi_N)}\norm{\varphi}_{H^{\theta}(\Pi_N)}.
    \end{align*}
    Now \eqref{eq:neg-product} follows by duality, since
    \begin{align*}
       \norm{uv}_{H^{-\theta}(\Pi_N)}&= 
       \sup\{\abs{\langle uv,\varphi\rangle_{L^2(\Pi_N)}}:\,\norm{\varphi}_{H^{\theta}(\Pi_N)}\leq 1\}
       \\&=
       \sup\{\abs{\langle u,v\varphi\rangle_{L^2(\Pi_N)}}:\,\norm{\varphi}_{H^{\theta}(\Pi_N)}\leq 1\}
       \\&\leq
       \sup\{\abs{\langle u,\tilde{\varphi}\rangle_{L^2(\Pi_N)}}:\norm{\tilde\varphi}_{H^{\theta}(\Pi_N)}\leq C\norm{v}_{\calC^{\theta+\epsilon}(\Pi_N)}\}
       \\&=C\norm{u}_{H^{-\theta}(\Pi_N)}\norm{v}_{\calC^{\theta+\epsilon}(\Pi_N)}.\qedhere
    \end{align*}
\end{proof}

\subsection{Stochastic estimates and stochastic sewing}\label{subsec:sewing}
In this section we recall some probabilistic ingredients used in the article. First, we have the following regularity estimates for the Ornstein-Uhlenbeck process.
\begin{lemma}\label{lem:OUreg}
    Let $p\geq 1$ and $\lambda\in (0,1)$, $\epsilon\in (0,\frac{1}{2})$. Then there exists a constant $C=C(T,p,\eps,\lambda)$ such that one has
    \begin{align}\label{eq:Otimespace1}
        \E\norm{O}_{C_T^{\frac{\lambda}{2}}\calC^{\frac{1}{2}-\lambda-\epsilon}(\T)}^p\leq C.
    \end{align}
    Moreover, for any $\theta\in (0,1)$, $\alpha\in (0,\frac{1}{2})$, $\epsilon\in (0,\frac{1}{4})$, there exists $C=C(T,p,\eps,\theta,\alpha)$
    such that for all $s,t\in[0,T]$ one has
    \begin{align}\label{eq:trade-off-est}
        \big(\E\sup_{r\in[0,T]} \norm{(P_{t+s}-P_t)O_r}_{\calC^{-\alpha}(\T)}^p\big)^{1/p}\leq C  t^{-\frac{\theta}{2}}\, s^{\frac{1}{4}+\frac{\alpha}{2}+\frac{\theta}{2}-\varepsilon}.
    \end{align}
\end{lemma}
\begin{proof}
    The bound \eqref{eq:Otimespace1} is classical, in this exact form it follows from \cite[Proposition 3.1]{DjGK}.
    The bound \eqref{eq:trade-off-est} follows from \eqref{eq:Otimespace1} as by the semigroup estimates in \cref{lem:semigroup1}, we have
    \begin{align*}
        \|(P_{t+s}-P_t)O_r\|_{\cC^{-\alpha}(\T)}&=\|P_t(P_s-\operatorname{Id})O_r\|_{\cC^{-\alpha}(\T)}\\&\lesssim t^{-\frac{\theta}{2}}\|(P_s-\operatorname{Id})O_r\|_{\cC^{-\alpha-\theta}(\T)}\lesssim t^{-\frac{\theta}{2}}s^{\frac{\theta}{2}+\frac{1}{4}+\frac{\alpha}{2}-\epsilon}\|O_r\|_{\cC^{\frac{1}{2}-2\varepsilon}(\T)}.\qedhere
    \end{align*}
\end{proof}
\begin{remark}\label{rem:OUreg-Sobolev}
The corresponding bounds \eqref{eq:Otimespace1} and \eqref{eq:trade-off-est} for $H^\beta(\T)$ in place of $\calC^{\beta}(\T)$ for $\beta=1/2-\lambda-\epsilon$ or $\beta=-\alpha$ follow due to the embedding $\mathcal{C}^{\beta+\epsilon/2}(\T) \hookrightarrow H^{\beta}(\T)$ for $\beta \in \mathbb{R}$, $\epsilon>0$ together with \eqref{eq:Otimespace1} and \eqref{eq:trade-off-est}.
\end{remark}

For $t\geq 0$ define the function
\begin{equ}\label{eq:Q}
    Q(t)=\int_0^t\norm{p_{u}}_{L^{2}(\T)}^2du
\end{equ}

\begin{lemma} \label{lem:QQN}
    For, $s\leqslant r$ and any $x \in \T$, it holds that
\begin{align}
    \mathrm{Var}( O_r(x)- P_{r-s} O_s(x))=Q(r-s).\label{eq:Q-is-var}
\end{align}
Moreover there exists a $C=C(T)$ such that for all $0\leq s<l<r\leq T$ and $\eps\in [0,1/2]$,
\begin{align}\label{eq:Qtwostep}
Q(r-s)-Q(l-s)\leq C (r-l)^{1-\eps}(l-s)^{-\frac{1}{2}+\eps}.
\end{align}
\end{lemma}
\begin{proof}
The identity \eqref{eq:Q-is-var} follows from It\^o's isometry and
\begin{equ}
    O_r(x)- P_{r-s} O_s(x)=\int_s^r \int_{\T} p_{r-u}(x-y)\xi(du,dy).
\end{equ}
The bound \eqref{eq:Qtwostep} follows from the bound $\|p_u\|_{L^2(\T)}^2\lesssim u^{-1/2}$ uniformly for $u\in[0,T]$.
\end{proof}
Throughout the paper, we often use stochastic sewing arguments. Hence, we recall the following special case of the stochastic sewing lemma in Banach spaces from \cite{LeBanach}. In \cite{LeBanach}, the statement allows for Banach spaces with martingale type $\mathfrak{p}\in[1,2]$ and general control functions instead of increments. For us usual increments are sufficient and any Hilbert space has martingale type $\mathfrak{p}=2$.
We only give the statement in this simplified setting, but include weight as in e.g. \cite{ABLM}.
Let for $S<T$, $$[S,T]^{2}_{\leq}:=\{(s,t)\in[S,T]^2\mid s\leq t\}, \quad [S,T]^3_{\leq}:=\{(u,w,v)\in[S,T]^3\mid u\leq w\leq v\}.$$
We call a two-parameter process $A:[S,T]_\leq^2\to H$ adapted if $A_{s,t}$ is $\F_t$-measurable for all $(s,t) \in [0,T]_\leq^2 $. Moreover, we use the notation $\delta A_{u,w,v}\coloneqq A_{u,v}-A_{u,w}-A_{w,v}$ for $(u,w,v) \in [0,T]_{\leq}^3$.
Finally, we use the shorthand $\E_s X=\E(X|\mathcal{F}_s)$ for $s\in[0,T]$ and integrable random variables $X$.

\begin{lemma}\label{lem:SSLHilbert}
Let $T_0>0$, $(S,T)\in[0,T_0]_\leq^2$, $p\geqslant 2$, $(H,\|\cdot\|_H)$ be a separable Hilbert space and $A:\Omega \times [S,T]_\leq^2 \rightarrow H$ be an adapted two-parameter process. Assume there exists $\eps_1,\eps_2>0$, $\delta_2<1/2+\eps_2$, $\delta_1<1+\eps_1$, and $\Gamma_1,\Gamma_2>0$ such that, for all $(u,v) \in [S,T]^2_\leq$ with $v<T$ and for all $w\in(u,v)$ one has the bounds
\begin{align}
    \|\|\E_u \delta A_{u,w,v}\|_H\|_{L^p(\Omega)}\leqslant \Gamma_1 (T-v)^{-\delta_1}(v-u)^{1+\eps_1},\label{eq:sewingassump1}\\
    \|\| A_{u,v}\|_{H}\|_{L^p(\Omega)}\leqslant \Gamma_2 (T-v)^{-\delta_2}(v-u)^{1/2+\eps_2}.\label{eq:sewingassump2}
\end{align}
Then, there exists an $H$-valued stochastic process $(\mathcal{A}_v)_{v \in [S,T)}$ such that, for any $v \in [S,T)$ and any sequence of partitions $\Pi_k=\{t_n^k\}_{i=0}^{N_k}$ of $[S,v]$ with mesh size going to $0$,
\begin{align}\label{eq:cA-convergence}
    \mathcal{A}_v=\lim_{k \rightarrow \infty} \sum_{i=0}^{N_k} A_{t_i^k,t_{i+1}^k} \text{ in }L^p(\Omega; H).
\end{align}
Moreover, $\mathcal{A}$ is the (up to modifications) unique adapted $H$-valued process such that $\mathcal{A}_S=0$, $(\mathcal{A}_v)_{v\in[S,T)}$ is continuous as a function of $v$ with values in $L^p(\Omega;H)$, and that there exist constants $C_1,C_2, C_3$, such that for all $(u,v)\in[S,T]_\leq^2$ with $v<T$,
\begin{align}
    \squeeze[1]{\|\|\mathcal{A}_v-\mathcal{A}_u-A_{u,v}\|_H\|_{L^p(\Omega)}}&\!\leqslant\! \squeeze[1]{C_1 (T-v)^{-\delta_1}(v-u)^{1+\eps_1} + C_2 (T-v)^{-\delta_2}(v-u)^{1/2+\eps_2}}\label{eq:SSL-conc1}\\
    \|\|\E_u(\mathcal{A}_v-\mathcal{A}_u-A_{u,v})\|_H\|_{L^p(\Omega)}&\leqslant C_3 (T-v)^{-\delta_1}(v-u)^{1+\eps_1}.\label{eq:SSL-conc2}
\end{align}
Furthermore, there exists a constant $C'=C'(T_0,\eps_1,\eps_2,p, \delta_1,\delta_2)$, such that the bounds \eqref{eq:SSL-conc1} and \eqref{eq:SSL-conc2} hold for $C_1=C' \Gamma_1, C_2=C'\Gamma_2$, $C_3=C'\Gamma_1$.
Finally, $\mathcal{A}$ extends continuously to $[S,T]$ and there exists a constant $C''=C''(T_0,\eps_1,\eps_2,p, \delta_1,\delta_2)$, such that $\mathcal{A}$ satisfies for all $(u,v)\in [S,T]^{2}_{\leq}$,
\begin{align} \label{eq:resultSewing}
    \norm{\norm{\mathcal{A}_{v}-\mathcal{A}_u}_H}_{L^{p}(\Omega)}\leq C'' (\Gamma_1 (v-u)^{1+\eps_1-\delta_1} + \Gamma_2 (v-u)^{1/2+\eps_2-\delta_2}).
\end{align}
\end{lemma}
\begin{proof}
    The statement can be easily deduced from \cite{LeBanach}. First, for any $\bar v\in[S,T)$, we apply the stochastic sewing lemma from \cite[Theorem 3.1]{LeBanach} on the interval $[S,\bar v]$ with constants $\Gamma_1(T-\bar v)^{-\delta_1}$ and $\Gamma_2(T-\bar v)^{-\delta_2}$. This implies the convergence in \eqref{eq:cA-convergence} for $v\in[S,\bar v]$ as well as the bounds \eqref{eq:SSL-conc1}-\eqref{eq:SSL-conc2} with $(T-v)$ replaced by $(T-\bar v)$ for  $(u,v)\in[S,\bar v]$. Moreover it follows that the bounds also hold with $C_i=C'(T_0,\eps_1,\eps_2,p)\Gamma_i$ for $i=1,2$. Choosing $v=\bar v$, and recalling that $\bar v$ is arbitrary, we get \eqref{eq:SSL-conc1}-\eqref{eq:SSL-conc2} as claimed.
    The bounds \eqref{eq:sewingassump2} and \eqref{eq:SSL-conc1} for $C_i=C'(T_0,\eps_1,\eps_2,p)\Gamma_i$ for $i=1,2$ together imply
    \begin{equ}
           \norm{\norm{\mathcal{A}_{v}-\mathcal{A}_u}_H}_{L^{p}(\Omega)}\leq  (C'+1)\big(\Gamma_1 (T-v)^{-\delta_1}(v-u)^{1+\eps_1} +  \Gamma_2 (T-v)^{-\delta_2}(v-u)^{1/2+\eps_2}\big)
    \end{equ}
    for $(u,v)\in[S,T]_\leq^2$ with $v<T$. Using the assumptions $\delta_1<1+\eps_1$ and $\delta_2<1/2+\eps_2$, this implies \eqref{eq:resultSewing} by \cite[Lemma 2.3]{BFG}.
\end{proof}
Next, we recall a version of Kolmogorov's continuity theorem in a mild form, cf. \cite[Proposition 4.6]{DjGK}, where this version was also employed. 
\begin{proposition}\label{prop:vKolmogorov}
Let $(X_t)_{t\in[0,T]}$ be a continuous stochastic process with $X_0=0$ and values in a Banach space $V$ and let $(S_t)_{t\geq 0}$ be strongly continuous semigroup. Suppose that for some $p>0$, $\alpha>0$, $C'<\infty$ it holds for all $0\leq s\leq t\leq T$ that
\begin{align}
    \E\|X_t-S_{t-s}X_s\|^p\leq C'|t-s|^{1+\alpha}.
\end{align}
Then for all $\gamma\in(0,\alpha/p)$
\begin{align}\label{eq:Kolmogorov-conclusion}
    \E\sup_{s< t\in[0,T]}|t-s|^{-\gamma p }\|X_t-S_{t-s}X_s\|^p\leq C''C',
\end{align}
where $C''$ depends only on $p,\gamma,\alpha,T$, and the semigroup $(S_t)_{t\geq0}$.
\end{proposition}

\subsection{Properties of the extension and restriction operators}\label{subsec:disc-op}

We first collect a few simple properties of the restriction and extension operators, all of which follow immediately from the definitions, with the only other input is \cref{lem:spaces-equivalence} b) for the point f) below.
\begin{lemma}\label{lem:Theta-Psi}
    \begin{enumerate}
    \item[a)] $\Psi_N$ is a right inverse of $\Theta_N$, that is, $\Theta_N\Psi_N=\mathrm{Id}_{C(\Pi_N)}$.
    \item[b)] $\Psi_N$ is a left inverse of $\Theta_N$ restricted to $C_N(\T)$, that is, $\Psi_N\Theta_N|_{C_N(\T)}=\mathrm{Id}_{C_N(\T)}$.
    \item[c)] For any $\beta\in\R$, $\Psi_N$ and $\Theta_N$ are isometries from $H^\beta(\Pi_N)$ to $H^\beta_N(\T)$ and from $H^\beta_N(\T)$ to $H^\beta(\Pi_N)$, respectively.
    \item[d)] For all $f\in C(\Pi_N)$ one has  $P_t\Psi_N f=\Psi_N \tilde P^N_t f$.
    \item[e)] For any continuous function $g:\C\to\C$ and any $w\in C(\T)$ one has $g(\Theta_N w)=\Theta_Ng(w)$.
    \item[f)] For all $\alpha \in (0,1)$ there exists $C=C(\alpha)$ such that for $f \in \calC^{\alpha}(\T)$ it holds that  
    $\|\Theta_N f\|_{\calC^\alpha(\Pi_N)}\leqslant C\|f\|_{\calC^\alpha(\T)}$. Similarly, $\|\Theta_N f\|_{L^\infty(\Pi_N)}\leqslant \|f\|_{L^\infty(\T)}$ for $f \in C(\T)$.
\end{enumerate}

\end{lemma}

The operator $\Psi_N\Theta_N$ is quite similar to a projection, but instead of sending certain Fourier modes to $0$, the operator sends them to smaller ones: for all $j\in J_N$ and all $k\in \Z$
\begin{equ}\label{eq:operator-action}
    \Psi_N\Theta_N e_{j+k N}=e_{j}.
\end{equ}
Recall that for the Fourier projection $\cQ_N$ with $\cQ_N e_j = e_j$ for all $\abs{j}\leq N$ and $\cQ_N e_j =0$ for $\abs{j}>N$, one has the bound
\begin{equ}\label{eq:projection-bound}
   \norm{(\operatorname{Id}-\cQ_N) v}_{H^{\beta}(\T)}= \Big\|\sum_{|k|> N}\hat v(k)e_k\Big\|_{H^{\beta}(\T)}\leq C(\alpha,\beta) N^{\beta-\alpha}\|v\|_{H^\alpha(\T)}
\end{equ}
for any $\alpha,\beta\in\R$ with $\beta\leq\alpha$. In the case of the operator $\Psi_N\Theta_N$, the corresponding estimate holds only in  a more restricted range of exponents.
This is the content of the following lemma.
\begin{lemma}\label{lem:operator}
For any $\alpha>1$ there exists $C=C(\alpha)$ such that for all $N\in\N$ and $v\in H^\alpha(\T)$ one has the bounds
\begin{align}
    \big\|(\Id-\Psi_N\Theta_N)v\big\|_{L^2(\T)}\leq C N^{-\alpha} \|v\|_{H^\alpha(\T)}\label{eq:pos-est}\\
    \big\|\Psi_N \Theta_N v\big\|_{H^\alpha(\T)}\leqslant C \|v\|_{H^\alpha}. \label{eq:PsiThetabounded}
\end{align}
\end{lemma}
\begin{proof}
Using \eqref{eq:operator-action} one can write
\begin{align*}
    (\Id-\Psi_N \Theta_N)v&=\sum_{k\in\Z} \hat{v}(k)(e_k-\Psi_N \Theta_N e_k)
    =\sum_{m\in \mathbb{Z}\setminus\{0\}} \sum_{k \in J_N+mN}\hat{v}(k)(e_k-e_{k-mN}).
\end{align*}
The latter difference is bounded by the triangle inequality and using \eqref{eq:projection-bound} (with $\beta=0$)
\begin{align*}
   \Big\|\sum_{m \in \mathbb{Z}\setminus \{0\}}\sum_{k \in J_N+mN}\hat{v}(k)(e_k-e_{k-mN})\Big\|_{L^2(\T)}&\leqslant \Big\|\sum_{k \notin J_N}\hat{v}(k)e_k\Big\|_{L^2(\T)}+\sum_{m \in \mathbb{Z}\setminus \{0\}}\Big\|\sum_{k \in J_N+mN}\hat{v}(k)e_{k-mN}\Big\|_{L^2(\T)}
   \\ 
   &\leq \Big(\sum_{|k|\geq N/2}|\hat v(k)|^2\Big)^{1/2}+\sum_{m\in\Z\setminus\{0\}}\Big(\sum_{\abs{k}\geq(mN)/2}|\hat v(k)|^2\Big)^{1/2}
   \\
   &\lesssim N^{-\alpha} \|v\|_{H^{\alpha}(\T)}+\sum_{m \in \mathbb{Z}\setminus \{0\}} (mN)^{-\alpha}\|v\|_{H^{\alpha}(\T)}.
\end{align*}
Since $\alpha>1$, the proof of \eqref{eq:pos-est} is finished. 
Similarly we see that
\begin{align*}
   \|\Psi_N \Theta_N v\|_{H^\alpha(\T)}&\leqslant \sum_{m\in \Z} \|\sum_{k \in J_N+mN} \hat{v}(k) e_{k-mN}\|_{H^\alpha(\T)}\\
    &= \sum_{m\in \Z} \Big(\sum_{j \in \Z}|\langle \sum_{k \in J_N+mN}\hat{v}(k)e_{k-mN},e_j\rangle|^2(1+|j|)^{2\alpha}\Big)^{1/2}\\
    &= \sum_{m\in \Z} \Big(\sum_{j \in J_N}|\langle \sum_{k \in J_N+mN}\hat{v}(k)e_{k-mN},e_j\rangle|^2(1+|j|)^{2\alpha}\Big)^{1/2}\\
    &\leq  \|v\|_{H^{\alpha}(\T)}+N^{\alpha}\sum_{m \in \Z\setminus\{0\}}\Big( \sum_{\abs{k}\geqslant (mN)/2}|\hat{v}(k)|^2\Big)^{1/2}\\
    &\leqslant \|v\|_{H^{\alpha}(\T)}+ N^{\alpha}\sum_{m \in \mathbb{Z}\setminus\{0\}} (mN)^{-\alpha}\|v\|_{H^{\alpha}(\T)}
    \leq \|v\|_{H^{\alpha}(\T)}(1+\sum_{m \in \mathbb{Z}\setminus\{0\}} m^{-\alpha}).\qedhere
\end{align*}
\end{proof}

We would like to prove similar bounds for $\Theta_N O$ as for $O$ in \cref{lem:OUreg}. This is fairly trivial in spaces with positive regularity, given  \cref{lem:Theta-Psi}, but not at all obvious in spaces of negative regularity, since $\Theta_N$ is of course not bounded there. This is the content of the following lemma.

\begin{lemma} \label{lem:Otreg}
Let $\alpha\in [0,\frac{1}{2})$,  $\eps\in (0,\frac{1}{2}-\alpha)$, and $p\geq2$. Then, there exist constants $C=C(T,p,\alpha, \epsilon)$ and  such that for $(s,t) \in [0,T]_{\leq}^2$,  one has
    \begin{align}
    \big\| \| \Theta_N O_t - \Theta_N O_{s}\|_{H^{-\alpha-\eps}(\Pi_N)}\big\|_{L^p(\Omega)}&\leqslant C\max\big\{N^{-\frac{1}{2}-\alpha-\epsilon}, (t-s)^{\frac{1}{4}+ \frac{\alpha}{2}}\big\}.\label{eq:Ot-Os}
    \end{align}
Furthermore, if $\beta\in[0,1]$, then there exists a constant $C=C(T,p,\alpha,\beta, \epsilon)$    such that for all $0\leq s<u\leq t\leq T$, $N\in\N$, one has that 
    \begin{align}
    \MoveEqLeft
    \big\| \| \Theta_N P_{t-s} O_s - \Theta_N P_{u-s} O_{s}\|_{H^{-\alpha-\epsilon}(\Pi_N)}\big\|_{L^p(\Omega)}\nonumber\\&\leqslant C (u-s)^{-\beta/2}\max\big\{N^{-\frac{1}{2}-\beta-\alpha-\epsilon}, (t-u)^{\frac{1}{4}+\frac{\beta}{2}+\frac{\alpha}{2}}\big\}.\label{eq:Ot-Os2}
    \end{align}
\end{lemma}
\begin{proof}
Since the two statements share some similarities, we begin with some remarks that will simplify the proof in both cases. 
First, note that since $O$ is Gaussian, it suffices to consider the case $p=2$. Fixing the time points, both statements are then of the form
\begin{equ}
    \E \|\Theta_N A\|_{H^{\alpha-\eps}(\Pi_N)}^2\leq B,
\end{equ}
where $A$ is a random function with complex orthogonal Fourier modes, that furthermore satisfy 
\begin{equ}
    \E \|A\|_{H^{\alpha-\eps}(\T)}^2\leq B.
\end{equ}
Indeed, this follows from \cref{lem:OUreg} and \cref{rem:OUreg-Sobolev}. Denote by $\cQ^-_N$ the projection (analogous but not equal to $\cQ_N$) defined by $\cQ_N^- e_j = e_j$ for all $j\in J_N$ and $\cQ_N^-e_j =0$ for $j\notin J_N$. It is trivial that $\cQ_N^-$ has norm $1$ in every Sobolev space $H^\gamma(\T)$.
Denote also $\cQ^\perp_N=\mathrm{Id}-\cQ_N^-$. Using \cref{lem:Theta-Psi}, we write
\begin{equs}
     \E \|\Theta_N A\|_{H^{\alpha-\eps}(\Pi_N)}^2=
     \E \|\Psi_N\Theta_N A\|_{H^{\alpha-\eps}(\T)}^2&\leq
     2\E \|\cQ^-_N A\|_{H^{\alpha-\eps}(\T)}^2+
     2\E \|\Psi_N\Theta_N\cQ_N^\perp A\|_{H^{\alpha-\eps}(\T)}^2
     \\
     &\leq 2B+2\E \|\Theta_N\cQ_N^\perp A\|_{H^{\alpha-\eps}(\Pi_N)}^2,
\end{equs}
and we can conclude that it suffices to prove the bound for $\cQ_N^\perp A$ instead of $A$.
We write
\begin{align*}
    \Theta_N \cQ_N^\perp A = \sum_{j\notin J_N\in \Z} \langle A, e_j\rangle_{L^2(\T)} e_j^N= \sum_{k\in J_N}e_k^N\sum_{m\in\Z\setminus\{0\}} \langle A, e_{k+mN}\rangle_{L^2(\T)}.
\end{align*}
By complex orthogonality, we get
\begin{equ}\label{eq:prepare}
I(A):= \E\|\Theta_N \cQ_N^\perp A\|_{H^{\alpha-\eps}(\Pi_N)}^2=\sum_{k\in J_N}(1+|k|)^{-2\alpha-2\eps}\sum_{m\in\Z\setminus\{0\}}\E\big|\langle A, e_{k+mN}\rangle_{L^2(\T)}\big|^2.
\end{equ}

    We first prove \eqref{eq:Ot-Os} and bound the right-hand side of \eqref{eq:prepare} for $A=O_t-O_s$.    Since
    \begin{align*}
        \langle (O_t-O_s),e_{k+mN}\rangle_{L^2(\T)} &= \int_{s}^{t}\int_\T e^{-4\pi^2 (k+mN)^2 (t-r)} e_{k+mN}(y) \xi(dr,dy)
        \\&\quad + \int_{0}^{s}\int_\T e^{-4\pi^2 (k+mN)^2 (s-r)}(e^{-4\pi^2 (k+mN)^2 (t-s)}-1) e_{k+mN}(y) \xi(dr,dy),
    \end{align*}
    we obtain by It\^o isometry that
    \begin{align*}
    \MoveEqLeft
        \E\abs{\langle (O_t-O_s),e_{k+mN}\rangle_{L^2(\T)}}^2
        \\&=\int_{s}^{t} e^{-8\pi^2 (k+mN)^2 (t-r)} \, dr+(1-e^{-4\pi^2 (k+mN)^2 (t-s)})^2\int_{0}^{s}e^{-8\pi^2 (k+mN)^2 (s-r)} \,dr.
    \end{align*}
    We first claim that the second summand can be bounded by the first. To see this, simply use the fact that for $a,b\geq 0$ with $b\leq a$, $(a-b)^2\leq a^2-b^2$, and integrate the exponentials to get
    \begin{align}
        (1-e^{-4\pi^2 (k+mN)^2 (t-s)})^2&\int_{0}^{s}e^{-8\pi^2 (k+mN)^2 (s-r)} \,dr\nonumber
        \\&\leqslant (1-e^{-8\pi^2 (k+mN)^2 (t-s)})\int_{0}^{\infty}e^{-8\pi^2 (k+mN)^2 r} \,dr\nonumber\\
    &=\int_s^t e^{-8 \pi^2(k+mN)^2(t-r)}\, dr. \label{eq:rewritingintegral}
    \end{align}
    Hence, we obtain that
    \begin{align}\label{eq:IOtOs}
        I (O_t-O_s)\leqslant 2\sum_{k\in J_N}(1+\abs{k})^{-2\alpha-2\epsilon}\int_s^t\sum_{m\in\Z\setminus\{0\}} e^{-8 \pi^2(k+mN)^2(t-r)}\, dr.
    \end{align}
   We proceed by bounding the sum of the exponentials.
    To that aim, we claim that if $\tau>0$, $k\in[-N/2,N/2]$, $\hat{\alpha}\in [0,1/2]$, then
    \begin{align}\label{eq:expsum}
        \sum_{m\in\Z\setminus\{0\}} e^{-(k+mN)^2\tau}\lesssim \begin{cases}
            \frac{1}{N\sqrt{\tau}}\quad&\,
            \\
            (1+|k|)^{-1+2\hat{\alpha}}\tau^{-1/2+\hat{\alpha}}\quad&\text{if }N\sqrt{\tau}> 1.
        \end{cases}
    \end{align}
Indeed, with the notation $u=k\sqrt\tau$, 
 \begin{equ}
     \sum_{m\in\Z\setminus\{0\}} e^{-(k+mN)^2\tau}= \frac{1}{N\sqrt{\tau}}N\sqrt{\tau}\sum_{m\in\Z\setminus\{0\}} e^{-(u+mN\sqrt{\tau})^2}\lesssim \frac{2}{N\sqrt{\tau}}\int_\R e^{-(u+x)^2}\,dx\lesssim\frac{1}{N\sqrt{\tau}},
\end{equ}
where the penultimate step follows by noting that both $m>0$ and $m<0$ halves of the sum (together with the prefactor $N\sqrt{\tau}$) give a lower Riemann sum approximation of the integral.
When $N\sqrt{\tau}> 1$, we use that $|u|\leq N\sqrt{\tau}/2$ implies $(u+mN\sqrt{\tau})^2\geq u^2+(m^2-m)N\sqrt{\tau}$, and so
\begin{align}
    \sum_{m\in\Z\setminus\{0\}} e^{-(u+mN\sqrt{\tau})^2}
    &\leq e^{-u^2}  \sum_{m\in\Z\setminus\{0\}} e^{-(m^2-m)N^2\tau}\lesssim u^{-1/2+\hat\alpha}\lesssim (1+|k|)^{-1+2\hat{\alpha}}\tau^{-\frac{1}{2}+\hat{\alpha}}.\nonumber
\end{align}
We split the estimate of \eqref{eq:IOtOs} into two cases. First assume that $|t-s|\leqslant N^{-2}$. 
    Hence, by the first inequality in \eqref{eq:expsum} for $\tau=t-r$, $r\in(s,t)$, using that $2\alpha+2\epsilon <1$, we have
\begin{align*}
I(O_t-O_s)\lesssim \int_{s}^{t}\sum_{k \in J_N} (1+|k|)^{-2\alpha-2\epsilon}\frac{1}{N(t-r)^{1/2}}\, dr
&\lesssim N^{-1} (t-s)^{1/2} \int_0^{N/2} (1+x)^{-2\alpha-2\epsilon}\,dx
\\&\lesssim N^{-1} (t-s)^{1/2} N^{1-2\alpha-2\epsilon} 
\lesssim N^{-1-2\alpha-2\epsilon}.
\end{align*}
For the second case $|t-s|>N^{-2}$ we split the integral into two regions: $[s,t-N^{-2}]$ and $[t-N^{-2},t]$. The former can be treated precisely as in the first case and using the assumption that $N^{-2}<|t-s|$. Hence, we focus on the latter for which we have by \eqref{eq:expsum} for $\tau=t-r\geq N^{-2}$ for $r\in (s,t-N^{-2})$  and $\hat{\alpha}=\alpha$ that
\begin{align*}
    \int_{s}^{t-N^{-2}}\sum_{m \in \Z} e^{-8\pi^2 (k+mN)^2 (t-r)} \, dr &\lesssim \int_s^{t-N^{-2}} (t-r)^{-\frac{1}{2}+\alpha}  (1+|k|)^{-1+2\alpha} \\
    &\lesssim  (t-s)^{1/2+\alpha}(1+|k|)^{-1+2\alpha}.
\end{align*}
This implies that since $\epsilon>0$,
\begin{align*}
I(O_t-O_s)\lesssim \sum_{k \in J_N} (1+|k|)^{-1-2\eps}(t-s)^{1/2+\alpha}\lesssim (t-s)^{1/2+\alpha}.
\end{align*}
Combining both cases, we find that
\begin{align*}
   I(O_t-O_s)\lesssim \max \{N^{-1-2\alpha-2\epsilon}, (t-s)^{1/2+\alpha}\},
\end{align*}
Together with the initial remarks at the beginning of the proof, this concludes the proof of \eqref{eq:Ot-Os}.

We proceed by showing \eqref{eq:Ot-Os2} via a similar line of proof as for \eqref{eq:Ot-Os}. Now we want to bound  $A=P_{t-s}O_s-P_{u-s}O_s$ in \eqref{eq:prepare}. We have that
    \begin{align*}
        \langle (P_{t-s}O_s&-P_{u-s}O_s),e_{k+mN}\rangle_{L^2(\T)} \\
        &= (e^{-4\pi^2 (k+mN)^2 (t-u)}-1)e^{-4\pi^2 (k+mN)^2 (u-s)}\int_{0}^{s}\int_\T e^{-4\pi^2 (k+mN)^2 (s-r)} e_{k+mN}(y) \xi(dr,dy),
    \end{align*}
    and thus by the isometry property of the white noise integral we obtain
    \begin{align}\label{eq:sk-pr}
        \E\big|&\langle (P_{t-s}O_s-P_{u-s}O_s),e_{k+mN}\rangle_{L^2(\T)}\big|^2\nonumber
        \\&=e^{-8\pi^2 (k+mN)^2 (u-s)}(e^{-4\pi^2 (k+mN)^2 (t-u)}-1)^2\int_{0}^{s} e^{-8\pi^2 (k+mN)^2 (s-r)} \, dr.
    \end{align}
    Similar to \eqref{eq:rewritingintegral}, we have
    \begin{align}\label{eq:rewritingintegral2}
        (1-e^{-4\pi^2 (k+mN)^2 (t-u)})^2\int_{0}^{s}e^{-8\pi^2 (k+mN)^2 (s-r)} \,dr&\leqslant \int_u^t e^{-8 \pi^2(k+mN)^2(t-r)}\, dr
    \end{align}
    and thus obtain together with \eqref{eq:sk-pr} that 
    \begin{align}\label{eq:IOtOs2}
        I(P_{t-s}O_s-P_{u-s}O_s)
        &\leqslant  \sum_{k\in J_N}(1+\abs{k})^{-2\alpha-2\epsilon}\int_u^t \sum_{m\in\Z\setminus\{0\}}e^{-8\pi^2 (k+mN)^2 (u-s)}e^{-8 \pi^2(k+mN)^2(t-r)}\, dr.
    \end{align}
To bound the inner sum, we combine \eqref{eq:expsum} together with the bound
$$e^{-(k+mN)^2\eta}\lesssim(1+|k+mN|)^{-2\beta}\eta^{-\beta}\lesssim N^{-2\beta}\eta^{-\beta}$$ for $\abs{m}\geq 1$, $\eta>0$ and $\beta\in[0,1]$, in order to obtain that for $\eta,\tau>0$, $k\in[-N/2,N/2]$ and $\hat{\alpha}\in [0,\frac{1}{2}]$, $\beta\in [0,1]$,
    \begin{align}\label{eq:expsum2}
        \sum_{m\in\Z\setminus\{0\}} e^{-(k+mN)^2\tau}e^{-(k+mN)^2\eta}\lesssim \begin{cases}
             \frac{1}{N^{1+2\beta}\sqrt{\tau}}\eta^{-\beta}\quad&\,
             \\
            N^{-2\beta}(1+|k|)^{-1+2\hat{\alpha}}\tau^{-1/2+\hat{\alpha}}\eta^{-\beta}\quad&\text{if }N\sqrt{\tau}> 1.
        \end{cases}
    \end{align}
We split the estimate of \eqref{eq:IOtOs2} into two cases. First assume that $t-u\leqslant N^{-2}$. 
Combining \eqref{eq:expsum2} 
with $\eta=u-s$ and $\tau=t-r\leq N^{-2}$ for $r\in (u,t)$ and \eqref{eq:rewritingintegral2}, we obtain, since $2\alpha+2\epsilon<1$ 
\begin{align*}
I(P_{t-s}O_s-P_{u-s}O_s)
&\lesssim (u-s)^{-\beta}N^{-1-2\beta}\sum_{k \in J_N} (1+|k|)^{-2\alpha-2\epsilon} \int_{u}^{t}(t-r)^{-1/2}\, dr
\\&\lesssim N^{-2\alpha-2\beta-2\epsilon}(u-s)^{-\beta}(t-u)^{1/2}\lesssim N^{-1-2\alpha-2\epsilon-2\beta}(u-s)^{-\beta}.
\end{align*}
If $t-u> N^{-2}$, we split the integral from $u$ to $t$ into two regions: $[u,t-N^{-2}]$ and $[t-N^{-2},t]$. The former integral can be treated precisely as in the first case. 
For the latter, by \eqref{eq:expsum2} applied for $\tau=t-r\geq N^{-2}$ with $r\in (u,t-N^{-2})$, $\eta=u-s$, $\hat{\alpha}=\alpha$, we find
\begin{align*}
\MoveEqLeft
    \int_{u}^{t-N^{-2}}\sum_{k\in J_N}(1+\abs{k})^{-2\alpha-2\epsilon}\sum_{m \in \Z\setminus\{0\}}e^{-8\pi^2 (k+mN)^2 (u-s)} e^{-8\pi^2 (k+mN)^2 (t-r)} \, dr \\&\lesssim N^{-2\beta}(u-s)^{-\beta}\sum_{k\in J_N}(1+|k|)^{-1-2\epsilon}\int_u^{t-N^{-2}} (t-r)^{-1/2+\alpha}\, dr \\
    &\lesssim (u-s)^{-\beta}(t-u)^{1/2+\alpha} N^{-2\beta}\lesssim (u-s)^{-\beta}(t-u)^{1/2+\alpha+\beta} 
\end{align*}
using that $N^{-2}< t-u$.
Overall we obtain
\begin{align*}
I(P_{t-s}O_s-P_{u-s}O_s)\lesssim (u-s)^{-\beta}\max(N^{-1-2\beta-2\alpha-2\epsilon},(t-u)^{1/2+\alpha+\beta}).
\end{align*}
Together with the initial remarks at the beginning of the proof, this concludes the proof of \eqref{eq:Ot-Os2}.
\end{proof}

\subsection{Semigroup estimates}\label{subsec:Semigroup}

For function spaces on $\R$, we define $C^{k}_{b}(\R)$ for $k=0,1,2,\dots$ as the space of bounded measurable functions whose distributional derivatives up to order $k$ are essentially bounded, equipped with the canonical norm (note that $C^{0}_{b}(\R)$ functions are not assumed to be continuous). We denote $C_b^k(\R)$ shortly as $C_b^k$. Let for $v\in \mathcal{S}'(\R)$, $$P_{t}^{\R}v\coloneqq p_{t}^{\R}\ast v, \quad p_{t}^{\R}(x)=\frac{1}{\sqrt{2\pi t}}e^{-x^2/2t}.$$ 
We have the immediate semigroup estimate for $0\leq s\leq t\leq 1$:
\begin{align}\label{eq:R-semigroup}
    \norm{(P^{\R}_t-P^{\R}_{s})v}_{C^{0}_b}\leq \abs{t-s}\norm{v}_{C^2_b}.
\end{align}
Furthermore, for a polynomial weight $\omega(x)=(1+\abs{x}^2)^{-\beta/2}$, $\beta\geq 1$, we define the weighted space 
$$C^{k}_{\omega}=\{f\in \mathcal{S}'(\R)\mid \omega f\in C^{k}_{b}\}, \quad\norm{f}_{C^{k}_{\omega}}=\norm{\omega f}_{C^{k}_b}.$$
In the following we recall basic heat kernel estimates. 
For a proof of \cref{lem:semigroup-w}, \cref{lem:semigroup1} and \cref{lem:semigroup2} see \cref{appendix}.

\begin{lemma}\label{lem:semigroup-w}
Let $(P^\R_t)$ be the heat semigroup and $\omega$ a polynomial weight for $\beta\geq 1$ as above. Then there exits a constant $C=C(\beta,T)$, such that
for $ t\in [0,T]$:
\begin{align}\label{eq:R-semigroup-weighted}
    \norm{(P^{\R}_t-\operatorname{Id})v}_{C^{0}_\omega}\leq C t\norm{v}_{C^2_\omega}, \quad \norm{P^{\R}_t v}_{C^{0}_\omega}\leq C \norm{v}_{C^{0}_{\omega}}.
\end{align}
\end{lemma}

\begin{lemma}\label{lem:semigroup1}
Let $(P_t), (P_t^{N}), (\tilde{P}^N_t)$ be the semigroups defined in \cref{subsec:setting}. Then we have the continuity bounds, for any $v\in L^{2}(\Pi_N)$, $$\norm{P_{t}^{N}v}_{L^{2}(\Pi_N)}\leq\norm{v}_{L^{2}(\Pi_N)},\quad \norm{\tilde{P}_{t}^{N}v}_{L^{2}(\Pi_N)}\leq\norm{v}_{L^{2}(\Pi_N)},
$$ and for any $v\in L^{2}(\T)$, $$\norm{P_{t}v}_{L^{2}(\T)}\leq\norm{v}_{L^{2}(\T)}.$$ 
Let $\alpha\in\R$, $\delta\geqslant 0$, $p\in [1,\infty]$. Then there exists a constant $C=C(\delta,T)$ such that for all $t\in(0,T]$ one has 
\begin{align}
\|P_t f\|_{B_{p,\infty}^{\alpha+\delta}(\T)}&\leq Ct^{-\delta/2}\|f\|_{B_{p,\infty}^{\alpha}(\T)}\label{eq:heatkernelH}
\\
\|P_t f\|_{\calC^{\delta}(\T)}&\leqslant C t^{-\delta/2}\|f\|_{L^\infty(\T)}\label{eq:heatLinfty}\\
\|P^N_t f\|_{B_{p,\infty}^{\alpha+\delta}(\Pi_N)}&\leq Ct^{-\delta/2}\|f\|_{B_{p,\infty}^{\alpha}(\Pi_N)}\label{eq:Calpha}\\
\|P_t^N f\|_{\cC^{\delta}(\Pi_N)}&\leqslant C t^{-\delta/2}\|f\|_{L^\infty(\Pi_N)}\label{eq:Linfty}
\end{align}
If $\delta\in[0,2]$, then there exists a constant $C=C(\delta, T)$ such that for all $t\in[0,T]$ one has
\begin{align}
\|P_t f-f\|_{\calC^{\alpha}(\T)}&\leq Ct^{\delta/2}\|f\|_{\calC^{\alpha+\delta}(\T)} \label{eq:heatalpha2}
\\
\|P_t f-f\|_{L^{\infty}(\T)}&\leq Ct^{\delta/2}\|f\|_{\calC^{\delta}(\T)}\label{eq:heatLinfty2}
\\
\|P^N_t f-f\|_{\calC^{\alpha}(\Pi_N)}&\leq Ct^{\delta/2}\|f\|_{\calC^{\alpha+\delta}(\Pi_N)}\label{eq:Calpha2}\\
\|P_t^N f-f\|_{L^{\infty}(\Pi_N)}&\leq Ct^{\delta/2}\|f\|_{\calC^{\delta}(\Pi_N)}\label{eq:Linfty2}
\end{align}
\end{lemma}

\begin{lemma}\label{lem:semigroup2}
    Let $(\tilde{P}_{t}^{N}), (P_t)$ be defined as in \cref{subsec:setting}.  For $\alpha\in\R$ and $\delta\in [0,2]$, there exists a constant $C=C(\delta,T)$, such that for all $t\in (0,T]$,
    \begin{align}
        \norm{P_{t}f}_{H^{\alpha+\delta}(\T)}&\leq C t^{-\delta/2}\norm{f}_{H^{\alpha}(\T)}
        \\\norm{\tilde{P}_{t}^{N}f}_{H^{\alpha+\delta}(\Pi_N)}&\leq C t^{-\delta/2}\norm{f}_{H^{\alpha}(\Pi_N)}\label{eq:tildeP1}\\
        \norm{P_{t}f-f}_{H^{\alpha}(\T)}&\leq C t^{\delta/2}\norm{f}_{H^{\alpha+\delta}(\T)}
        \\\norm{\tilde{P}_{t}^{N}f-f}_{H^{\alpha}(\Pi_N)}&\leq C t^{\delta/2}\norm{f}_{H^{\alpha+\delta}(\Pi_N)}.
    \end{align}
\end{lemma}

Next we collect a few properties of the discrete Laplacian $\Delta_N$, its eigenvalues, and its semigroup. The proofs of these are elementary and can be found in e.g. \cite{Gy, BDG-SPDE, GS}. 
\begin{lemma}\label{lem:eigenvalues}
\begin{enumerate} 
    \item[a)] For all $j\in J_N$ one has $\frac{4}{\pi^2}|\lambda_j|\leq|\lambda_j^N|\leq |\lambda_j|$.
    \item[b)] For all $j\in J_N\setminus\{0\}$ one has $\big|1-\frac{\lambda_j^N}{\lambda_j}\big|\leq\frac{1}{3}(\frac{j\pi}{N})^2$.
    \item[c)] For all $p\geq 1$, $f\in C(\Pi_N)$, one has $\langle |f|^{p-1} f,\Delta_N f\rangle_{L^2(\Pi_N)}\leq 0$.
    \item[d)] The operator $\Delta_N$ is the generator of the continuous time random walk on $\Pi_N$ that jumps to a neighbor with equal probability after an exponential time with parameter $2N^2$. As a consequence, $P^N_t:L^\infty(\Pi_N)\to L^\infty(\Pi_N)$ has operator norm $1$.
    \end{enumerate}    
\end{lemma}

Next, we are concerned with a commutator-type estimate of the discrete semigroup $(P_{t}^{N})$ and continuous heat semigroups $(P_t)$ with $\Psi_N$.
Recall that for $(\tilde{P}_t^{N})$ the commutator vanishes in the sense that $\Psi_N\tilde{P}_t^{N}f=P_{t}\Psi_N f$ by Lemma \ref{lem:Theta-Psi} (d). 
\begin{lemma}\label{lem:commutator}
   Let $\alpha,\beta,\gamma\in\R$ be such that $\gamma\in[0,2]$ and $\beta-\alpha-\gamma\leq 2$. Then there exists a constant $C=C(\alpha,\beta,\gamma)$ such that for all $f\in C(\Pi_N)$, $t> 0$ one has
    \begin{equ}
        \big\|(\Psi_N P_t^N-P_t \Psi_N) f\big\|_{H^\alpha(\T)}\leq C N^{-\gamma} t^{(\beta-\alpha-\gamma)/2}\|f\|_{H^\beta(\Pi_N)}.
    \end{equ}
\end{lemma}
\begin{proof}

Let us use the shorthands 
$\hat f_N(j)=\langle f, e_{j}^{N}\rangle_{L^{2}(\Pi_N)}$.
    By definition of the Sobolev norms and the action of the semigroups
    \begin{equs}
        \big\|(\Psi_N P_t^N-P_t \Psi_N) f\big\|_{H^\alpha(\T)}^2\lesssim\sum_{j\in J_N} |j|^{2\alpha} \big|\big(e^{\lambda^N_j t}-e^{\lambda_j t}\big)\hat f_N(j)\big|^2.
    \end{equs}
    Note that here and below in the proof $j=0$ can be excluded from the summation.
    Using that for $x\geq 0$, $e^x-1\leq xe^x$, we can write for $\gamma\in [0,2]$ and $\abs{j}\leq N$
    \begin{equ}
        e^{\lambda^N_j t}-e^{\lambda_j t}
        =e^{\lambda_j t}(e^{(\lambda_j^N-\lambda_j)t}-1)
        \leq t(\lambda_j^N-\lambda_j)e^{\lambda_j^N t}
        \lesssim t j^4 N^{-2}e^{-c j^2 t}
        \leq 
        t|j|^{2+\gamma}N^{-\gamma}e^{-cj^2t},
    \end{equ}
    where we furthermore used that $\lambda_j^{N}\geq \lambda_j$ and Lemma \ref{lem:eigenvalues} and where the constant $c>0$ is chosen such that $\sin^2(x)\geq c x^2$ for $x\in [0,1]$.
    Combining this with the estimate $e^{-x}\leq C(\theta) x^{\theta}$ for any $\theta\leq 0$ uniformly in $x$, applied with $x=2cj^2 t$ and $\theta=\beta-\alpha-\gamma-2$, we get 
    \begin{equs}
         \big\|(\Psi_N P_t^N-P_t \Psi_N) f\big\|_{H^\alpha(\T)}^2&\lesssim t^2N^{-2\gamma}\sum_{j\in J_N}|j|^{2\alpha+4+2\gamma}e^{-2cj^2t}|\hat f_N(j)|^2
         \\
         &\lesssim t^{\beta-\alpha-\gamma}N^{-2\gamma}\sum_{j\in J_N}|j|^{2\beta}|\hat f_N(j)|^2,
    \end{equs}
    which proves the claim.
\end{proof}

\subsection{Well-posedness}\label{subsec:wp}

In the following we recall well-posedness of \eqref{eq:main} and regularity estimates of the solution. For notational simplicity we denote $\|\cdot\|_{\cC^s(\T)}$ by $\|\cdot\|_{\cC^s}$.
\begin{proposition}\label{prop:wp}
    Let \cref{ass1} or \cref{ass2} hold, where we let $m=0$ if $f$ satisfies \cref{ass1} and $m$ is given by \cref{ass2} otherwise. Let \cref{asn:u0} hold. Then there exists a unique mild solution $u$ to \eqref{eq:main}. Moreover, for any $\lambda\in (0,1),\, \epsilon\in (0,1/2)$, $p\geq 1$, there exists a constant $C=C(T,p,\lambda,\epsilon, m, K)$ such that
    \begin{align}
        \E\norm{u}^p_{C_T^{\lambda/2}\calC^{1/2-\lambda-\epsilon}(\T)}\leq C (1+ \E\norm{u_0}_{\calC^{1/2-\epsilon}(\T)}^{(m+1)p}).
    \end{align}
    Furthermore, with the notation $v=u-O$, for any $\lambda\in (0,2],\,\epsilon\in (0,1/2)$, $p\geq 1$, there exists a constant $C=C(T,p,\lambda,\epsilon, m, K)$ such that
    \begin{align}\label{eq:apriori-v}
        \E\norm{v}^p_{C_T^{\lambda/2}\calC^{5/2-\lambda-\epsilon}(\T)}\leq C (1+ \E\norm{u_0}_{\calC^{5/2-\epsilon}(\T)}^{(m+1)^2 p}).
    \end{align}

\end{proposition}
\begin{proof}
    The well-posedness of \eqref{eq:main} is classical (see e.g. \cite[Proposition 6.2.2]{Cerrai}) and 
    regularity bounds for $u$ in the stated form follow from e.g. \cite[Proposition 3.3]{DjGK}. 
    Given the regularity bounds for $u$, we infer the bound for $v$. 
    First we note that by \eqref{eq:comp-B2} below, with $\theta=1/2-\epsilon>0$, $\kappa=1$ and $u\in\calC^{1/2-\epsilon}(\T)$,
    \begin{align}
        \norm{f(u)}_{\calC^{1/2-\epsilon}(\T)}\leqslant C K (1+\norm{u}_{L^{\infty}(\T)}^m)(1+\norm{u}_{\calC^{1/2-\epsilon}(\T)})\lesssim (1+\norm{u}_{\calC^{1/2-\epsilon}(\T)}^{m+1}).
    \end{align}
    Using the latter bound and the semigroup estimates \eqref{eq:heatkernelH}, \eqref{eq:heatalpha2} from \cref{lem:semigroup1}, we see that with $u_r=v_r+O_r$ and $\lambda\in (0,2]$,
    \begin{align*}
        \norm{v_t-v_s}_{\calC^{5/2-\lambda-\epsilon}(\T)}&\leq \norm{(P_{t} -P_s)u_0}_{\calC^{5/2-\lambda-\epsilon}(\T)} + \int_{s}^{t} \norm{P_{t-r} f(v_r+ O_r)}_{\calC^{5/2-\lambda-\epsilon}(\T)}\, dr \\&\qquad+ \int_{0}^{s}\norm{P_{s-r}(P_{t-s}-\operatorname{Id})f(v_r+ O_r)}_{\calC^{5/2-\lambda-\epsilon}(\T)} \, dr
        \\&\lesssim (t-s)^{\lambda/2}\norm{u_0}_{\calC^{5/2-\epsilon}(\T)} + (1+\norm{u}_{C_T\calC^{1/2-\epsilon}(\T)}^{m+1})\int_{s}^{t} (t-r)^{-1+\lambda/2}\, dr \\&\qquad + (t-s)^{\lambda/2} (1+\norm{u}_{C_T\calC^{1/2-\epsilon}(\T)}^{m+1})\int_{0}^{s} (s-r)^{-1+\epsilon/4}\, dr
        \\&\lesssim (t-s)^{\lambda/2}(1+\norm{u_0}_{\calC^{5/2-\epsilon}(\T)} + \norm{u}_{C_T\calC^{1/2-\epsilon}(\T)}^{m+1})
    \end{align*}
    and the claim follows after taking the $p$-th moment and using the bound for the solution $u$.
    \end{proof}

\subsection{Auxiliary lemmata}\label{subsec:aux}
First, we recall the properties of the flow $(\Phi_t)_{t\geq 0}$ and the approximate nonlinearities $(g_t)_{t\geq0}$. Such properties are commonplace in the literature on splitting schemes, the exact form stated below can be found in \cite[Lemma~5.1]{DjGK}.
    \begin{lemma}\label{lem:g_h-bounds}
Let $f$ satisfy Assumption \ref{ass2}. Let $\big(\Phi_t(x)\big)_{t\geq 0,x\in\R}$ and $\big(g_t(x)\big)_{t\geq 0,x\in\R}$ be given as in \eqref{eq:Phi} and \eqref{eq:gh}.
Then there exist constants $C, \tilde K>0$ and $\tilde{m}\geq m$, depending only on $K$ and $m$, such that for all $i=0,1,2,3$,  $t\in[0,1]$, $x,y\in\R$ the following bounds hold
\begin{align}
    \abs{\Phi_{t}(x)-\Phi_{t}(y)}&\leq e^{Kt/2}\abs{x-y},\label{eq:PhiLip}
    \\
    \abs{\partial^{i}g_{t}(x)}&\leq \tilde K(1+\abs{x}^{2\tilde{m}+1-i}),
    \label{eq:ghgrowth}\\
    \partial g_{t}(x)&\leq K,
    \\
    \abs{g_{t}(x)-g_{0}(x)}&\leq C t(1+\abs{x}^{4m+2}).
\end{align}   
In particular, one has 
\begin{align}
    (g_{t}(x)-g_{t}(y))(x-y)&\leq K (x-y)^2,
    \label{eq:g-bound1}\\
    \abs{g_{t}(x)-g_{t}(y)}&\leq \tilde K(1+\abs{x}^{2\tilde m}+\abs{y}^{2\tilde m})\abs{x-y}.\label{eq:g-bound2}
\end{align}
\end{lemma}

The next lemma deals with estimates on the composition of a Hölder function with a regular function $g$, that satisfies certain growth bounds. The proof of \cref{lem:composition} is straightforward and can be found in the Appendix \ref{appendix}. 
\begin{lemma}\label{lem:composition}
Let $K>0$, $m\geq 0$ and $g:\R\to\R$ with $g\in C^{\kappa}$ for some $\kappa\in \{1,2\}$.
Assume that $g$ admits polynomial growth in the sense that there exist $K> 0$, $m\geq 0$, such that for all $i=0,\dots,\kappa$, $x\in\R$,
\begin{align*}
    \abs{(\partial^{i}g) (x)}\leq K (1+\abs{x}^{m}).
\end{align*}
Then there exist constants $C_1=C_1(m,\kappa), C_2=C_2(m,\kappa), C_3=C_3(m,\kappa)$ such that the composition of $g$ with Besov, respectively Sobolev, functions respects the following bounds. If $\kappa=2$, $\theta\in (1,2)$ and $v\in \calC^{\theta}(\T)$, we have 
\begin{align}\label{eq:comp-B}
\norm{g(v)}_{\calC^{\theta}(\T)}&\leq C_1 K (1+\norm{v}_{L^{\infty}(\T)}^{m})(1+\|v\|_{C^1_b(\T)}\|v\|_{\cC^{\theta-1}(\T)}+\norm{v}_{\calC^{\theta}(\T)})\\
    &\leqslant 2 C_1 K (1+\norm{v}_{L^{\infty}(\T)}^{m})(1+\|v\|_{\cC^\theta(\T)}^2).\nonumber
\end{align}
For $\kappa=1$, $\theta \in (0,1)$ and $v\in \calC^{\theta}(\T)$, we have 
\begin{align}\label{eq:comp-B2}
    \|g(v)\|_{\cC^\theta (\T)}\leqslant C_2K(1+\|v\|_{L^\infty (\T)}^{m})(1+\|v\|_{\cC^\theta (\T)}).
\end{align}
Moreover, if $\theta\in (0,1)$ and $v\in H^{\theta}(\T)\cap L^{\infty}(\T)$, we have
\begin{align}\label{eq:comp-H}
    \norm{g(v)}_{H^{\theta}(\T)}\leq C_3 K (1+\norm{v}_{L^{\infty}(\T)}^{m})(1+\norm{v}_{H^{\theta}(\T)}).
\end{align}
In particular, if all $\partial_i g$ are bounded for $i=0,\dots,\kappa$ (i.e. $m=0$), we find for $\theta\in (1,2)$ and $\kappa=2$ and $v\in \calC^{\theta}(\T)$, 
\begin{align}
    \norm{g(v)}_{\calC^{\theta}(\T)}\leq C_1 K (1+\|v\|_{C^1_b(\T)}\|v\|_{\cC^{\theta-1}(\T)}+\norm{v}_{\calC^{\theta}(\T)}),
\end{align}
as well as, for $\theta \in (0,1)$ and $\kappa=1$ and $v\in H^{\theta}(\T)$, 
\begin{align}
    \norm{g(v)}_{H^{\theta}(\T)}\leq C_3 K (1+\norm{v}_{H^{\theta}(\T)})
\end{align}
and for $\theta \in (0,1)$ and $\kappa=1$ and $v\in \calC^{\theta}(\T)$,
\begin{align}
    \norm{g(v)}_{\calC^{\theta}(\T)}\leq C_2 K (1+\norm{v}_{\calC^{\theta}(\T)}).
\end{align}
Moreover, if we let $v\in \calC^{\theta}(\Pi_N)$, respectively $v\in H^{\theta}(\Pi_N)$, the same composition estimates as above hold true when replacing  $\calC^{\theta}(\T)$ and $H^\theta (\T)$ by $\calC^{\theta}(\Pi_N)$ and $H^{\theta}(\Pi_N)$.
\end{lemma}

\begin{lemma}\label{lem:composition-applies}
    Let $h\in C^{\kappa}_{\omega}(\R)$ for a polynomial weight $\omega(x)=(1+\abs{x}^2)^{-\beta/2}$ for some $\beta\geq 0$ and $\kappa\in\{1,2\}$. Let again $(P^{\R}_t)$ be the heat semigroup acting on functions on $\R$. Then it follows that for any $t\in [0,T]$, $P_{t}^{\R}h\in C^{\kappa}_{\omega}$ satisfies the assumptions of \cref{lem:composition} for $m=\beta$ and a constant $K=K(\beta,T)$, that is uniform in $t\in[0,T]$.
\end{lemma}
\begin{proof}
Let $\omega_{-\beta}(x):=(1+x^2)^{\beta/2}$, $x\in\R$. An easy estimate shows that for $x,y\in\R$, $\frac{\omega_{-\beta}(x-y)}{\omega_{-\beta}(x)}\leq 2^{\beta/2}\omega_{-\beta}(y)$.
Since $h\in C^{0}_\omega$, we thus have that for almost all $x\in\R$, 
\begin{align*}
   \squeeze[1]{\abs{P_{t}^\R h (x)} =\abs{p_t\ast h (x)} \leq \!\!\int \!\! p_{t}(y) \abs{h(x-y)} \, dy \leq \!\!\int\!\! p_t (y) \omega_{-\beta}(x-y) dy \leq 2^{\beta/2} \paren[\bigg]{\int\!\! p_{t}(y) \omega_{-\beta}(y)\, dy}\omega_{-\beta}(x)}
\end{align*}
and the claim follows since for $t\in (0,T]$, $$\norm{p_t\omega_{-\beta}}_{L^1}=\E(1+B_t^2)^{\beta/2}\leq c(\beta) (1+t^{\beta/2})\leq c(\beta)(1+T^{\beta/2})$$ for a constant $c(\beta)>0$, where $B$ denotes a standard Brownian motion. For $\partial^i P_{t}^\R h $ for $i=1,\dots,\kappa$ the estimate is similar using the specific bound for $\partial^{i} h$.
\end{proof}

\section{Bounded nonlinearity}\label{bigsec:bounded}
Throughout the section \cref{ass:tildeO}, \cref{asn:u0}, and \cref{ass1} are imposed.

We decompose the error into a spatial and a temporal error, that we bound in the respective subsections. The proof of the main \cref{thm:main1} immediately follows from the bound for the spatial error, \cref{lem:spatial} in \cref{sec:spatialerror}, and the bound for the temporal error, \cref{lem:temporal} in \cref{sec:temporalerror}.

In the following we introduce some processes used throughout the section. Recall the spatial discretisation $v^N$ defined in \eqref{eq:tildevN}. Let $z^N$ be the extension of $v^N$ from $\Pi_N$ to $\mathbb{T}$, that is $z^N=\Psi_N v^N$, which satisfies
\begin{align*}
     z^{N}_t&= \Psi_N \tilde{P}^N_t\Theta_N u_0 + \int_{0}^{t} \Psi_N\tilde{P}^N_{t-s} \Theta_N \big(f(z^N_s+O_s)\big)\, ds
     \\&= P_t\Psi_N\Theta_N u_0 + \int_{0}^{t} P_{t-s} \Psi_N\Theta_N \big(f(z^N_s+O_s)\big)\, ds,
\end{align*}
using Lemma~\ref{lem:Theta-Psi} a), d) and e). 
\begin{lemma} \label{lem:wpvNzN}
The equation \eqref{eq:tildevN} is well posed and there exists $C=C(T,K)$ such that almost surely,
\begin{align}\label{eq:L2tildev}
    \sup_{t \in [0,T]} \|v^N_t\|_{L^2(\Pi_N)}=\sup_{t \in [0,T]} \|z^N_t\|_{L^2(\T)}\leqslant C(1+\|u_0\|_{L^\infty(\T)}).
\end{align}
\end{lemma}
\begin{proof}
    Note that both $x\mapsto u_0(x)$ as well as $x\mapsto O_t (x)$ are continuous, so that $\Theta_N$ is well-defined when applied to $u_0$ and $O_t$. Note that \eqref{eq:tildevN} is a finite dimensional ODE with globally Lipschitz coefficients (by Assumption~\ref{ass1} and boundedness of $\tilde{\Delta}^N$ on $L^2(\Pi_N)$)   well posedness almost surely is immediate. The mild form is simply the finite dimensional variation of constants formula.
    Moreover, using the contraction property of $\tilde{P}^N$ from Lemma~\ref{lem:semigroup1} and boundedness of $f$,
    \begin{align*}
     \|v^N_t\|_{L^2(\Pi_N)} &\lesssim  \|\Theta_N u_0\|_{L^2(\Pi_N)}+K
     \lesssim \|u_0\|_{L^\infty(\T)}+1,
    \end{align*}  
    where we used Lemma~\ref{lem:Theta-Psi} f) in the second inequality. The isometry property of $\Psi_N$ (see Lemma 2.1 c)), gives the equality in \eqref{eq:L2tildev}.
\end{proof}
One may further upgrade temporal and spatial regularity for the solutions $v^N, z^N$ to obtain similar bounds as for $v=u-O$, however we do not bother to do so, since it won't be relevant in our error analysis below.

Recall the processes $V^{M,N}$, $\tilde{V}^{M,N}$ given in mild form in \eqref{eq:V-mild} and \eqref{eq:tildeV-mild}.
We decompose the error of $\Theta_N v-\tilde{V}^{M,N}$ into an error of the spatial discretisation, an error of a temporal discretisation (which, however, will still depend on $N$ due to \cref{lem:Otreg}), and an error coming from the approximation of the noise $O$, as follows. More precisely, we write
\begin{align} \label{eq:errordecomposed}
\norm{\Theta_N v_t - \tilde{V}^{M,N}_t}_{L^{2}(\Pi_N)}\leq \norm{\Theta_Nv_t-v^N_t}_{L^2(\Pi_N)}+ \norm{v_t^N - V_t^{M,N}}_{L^{2}(\Pi_N)}+\norm{V^{M,N}_t - \tilde{V}^{M,N}_t}_{L^{2}(\Pi_N)}.
\end{align}

We treat the first two of the terms separately in the subsequent sections, the third term can be easily dealt with via the following lemma:

\begin{lemma}\label{lem:approximation}
Let $p\geq 1$ and $\eps\in (0, 1/2)$. Then there exists $C=C(T,\eps, p,K, \mathcal{M})$, that does not depend on $M,N$, such that 
\begin{align}
\MoveEqLeft
    (\E\sup_{t\in[0,T]}\norm{V^{M,N}_t - \tilde{V}^{M,N}_t}_{L^{2}(\Pi_N)}^p)^{1/p}
    \leq C (M^{-1+\eps}+N^{-\frac{3}{2}+\eps}).
\end{align}    
\end{lemma}

\begin{proof}
    Using that $f$ is globally Lipschitz and the contraction property of $\tilde{P}^N$ on $L^2(\Pi_N)$, we have that for $0\leq r\leq T$, since $k_{M}(s)\leq s$
\begin{align*}
\MoveEqLeft
    \E\sup_{t\in[0,r]}\|V^{M,N}_t-\tilde{V}^{M,N}_t\|_{L^2(\Pi_N)}^p\\&\lesssim \int_0^r \E\|V^{M,N}_{k_M(s)}-\tilde{V}^{M,N}_{k_M(s)}\|_{L^2(\Pi_N)}^p\, ds+\int_0^r \E\|\Theta_N(O_{k_M(s)}-\tilde{O}_{k_M(s)})\|_{L^2(\Pi_N)}^p\, ds
    \\&\lesssim \int_0^r \E\sup_{t\in [0,s]}\|V^{M,N}_{t}-\tilde{V}^{M,N}_{t}\|_{L^2(\Pi_N)}^p\, ds+\max_{k=0,\dots,M} \E\|\Theta_N(O_{t_k}-\tilde{O}_{t_k})\|_{L^2(\Pi_N)}^p.
\end{align*}
The result then follows after applying  Gr\"onwall's inequality and using Assumption~\ref{ass:tildeO} \ref{en:b}. 
\end{proof}

\subsection{Spatial error}\label{sec:spatialerror}
We start by giving an outline of how we treat the spatial error term in \eqref{eq:errordecomposed} in order to motivate the intermediate Lemma~\ref{lem:space-ss}. Together with its Corollary~\ref{cor:girsanov}; this allows to state and prove the main lemma -- Lemma~\ref{lem:spatial}.
We bound the spatial error on the continuum $\mathbb{T}$. Using the isometry property of the extension operator $\Psi_N$ and triangle inequality, we arrive at
\begin{align*}
   \norm{\Theta_N v_t - v^N_t}_{L^{2}(\Pi_N)}&\leq  \norm{v_t - z^N_t}_{L^{2}(\mathbb{T})} +  \norm{(\operatorname{Id}-\Psi_N\Theta_N)v_t}_{L^{2}(\mathbb{T})}.
\end{align*}
The crucial part is the first summand above. We have that 
\begin{align}\label{eq:error-decomp}
    \|v_t-z^N_t\|_{L^2(\mathbb{T})}&\leqslant \|P_t(\operatorname{Id}-\Psi_N \Theta_N)u_0\|_{L^2(\mathbb{T})}+ \Big\|\int_0^t P_{t-s} f(v_s+O_s)- P_{t-s} \Psi_N\Theta_N f(z^N_s+O_s) \,ds\Big\|_{L^2(\mathbb{T})}\nonumber\\
    &\leqslant \|P_t(\operatorname{Id}-\Psi_N \Theta_N)u_0\|_{L^2(\mathbb{T})}+ \Big\|\int_0^t P_{t-s} f(v_s+O_s)-P_{t-s}f(z_s^N+O_s)\, ds\Big\|_{L^2(\mathbb{T})}\nonumber\\
    &\quad + \Big\|\int_0^t P_{t-s} (\operatorname{Id}-\Psi_N \Theta_N) f(z_s^N+O_s) \, ds\Big\|_{L^2(\mathbb{T})}\\
    &\eqqcolon E_0(t)+E_1(t)+E_2(t).
\end{align}
The terms $E_0(t)$ and $E_1(t)$ are treated directly in the proof of \cref{lem:spatial}. After an application of Girsanov's theorem (see \cref{cor:girsanov} below), bounding $E_2(t)$ reduces to bounding
\begin{align}
    \Big\|\int_0^t P_{t-s} (\operatorname{Id}-\Psi_N \Theta_N) f(O_s+P_s\Psi_N\Theta_N u_0)  \, ds\Big\|_{L^2(\mathbb{T})}.
\end{align}
Bounding the above is the content of the next lemma.

\begin{lemma}\label{lem:space-ss}
Let $p\geq 1$ and $\alpha\in (1,3/2)$. Then there exists $C=C(T,\alpha, p,K, \mathcal{M})$, that does not depend on $N$, such that for all $0\leq s\leq t\leq T$,
\begin{align}
   \paren[\bigg]{\E \norm[\bigg]{\int_{s}^{t} P_{t-r} (\operatorname{Id}- \Psi_N \Theta_N)f(O_r+P_r\Psi_N\Theta_N u_0) \,dr}_{L^2(\T)}^p}^{1/p}\leq C (t-s)^{1/2} N^{-\alpha}.
\end{align}
In particular, with a possibly different constant $C$
\begin{align}
   \paren[\bigg]{\E \sup_{t\in[0,T]}\norm[\bigg]{\int_{0}^{t} P_{t-s} (\operatorname{Id}- \Psi_N \Theta_N)f(O_s+P_s\Psi_N\Theta_N u_0) \,ds}_{L^2(\T)}^p}^{1/p}\leq C  N^{-\alpha}.
\end{align}
\end{lemma}

\begin{proof}
 Throughout the proof we follow the convention of omitting $\T$ in the norms appearing. Fix arbitrary $(s,t) \in [0,T]_\leq$. We aim to apply the stochastic sewing lemma for Hilbert space valued processes, see Lemma~\ref{lem:SSLHilbert} with $H=L^2(\T)$, on $[s,t]$. 
 Let the germ be defined by 
 \begin{align}\label{eq:first-SS-A-def}
     A_{u,v}&\coloneqq \E_u\int_u^v P_{t-r}(\operatorname{Id}-\Psi_N\Theta_N)f(O_r+P_r\Psi_N\Theta_N u_0)\, dr
 \end{align}
 for $(u,v) \in [s,t]_\leq^2$.
 In order compute the conditional expectation, we use the following rule: for $F:\R\to\R$ measurable and random variables $X,Y$, where $X$ is $\F_u$-measurable and $Y\sim N(0,\sigma)$ (that is, $Y$ has a normal distribution with mean $0$ and variance $\sigma$) with $Y$ independent of $\F_u$, we have almost surely
 \begin{align}\label{eq:rule}
     \E_{u}[F(X+Y)]=(P^{\R}_{\sigma}F) (X)
 \end{align}
 for the heat semigroup $P^{\R}$ acting on functions on $\R$.
 Let $u\leq r$, $x\in\T$, and let $w$ be a $\mathcal{F}_0$-measurable function-valued random variable with arbitrary high finite moments in $\mathcal{C}^\theta(\T)$ for some $\theta>0$. 
 Using\footnote{We remark that at this point the fact that $f$ is Nemytskii type is used.}  \eqref{eq:rule} with $X=P_{r-u}O_u(x)+w(x)$ and $Y=O_r(x)-P_{r-u}O_u(x)$, we get
 \begin{align*}
     \E_u f(O_r+w) (x)= \E_u f(O_r(x)+w (x)) = (P^{\mathbb{R}}_{Q(r-u)}f)(P_{r-u}O_u (x)+w (x)),
 \end{align*}
 where we recall $Q$ from \cref{lem:QQN} and in particular that it does not depend on $x\in\T$. 
 Since both sides have a continuous modification, the equality holds when viewing both sides as an element of $H$.
  We apply this with $w=P_r\Psi_N\Theta_N u_0$  
  and thus we see that, almost surely,
\begin{align*}
    A_{u,v} =\int_{u}^{v} P_{t-r}(\operatorname{Id}-\Psi_N\Theta_N)(P^{\mathbb{R}}_{Q(r-u)}f)(P_{r-u}O_u+P_r\Psi_N\Theta_N u_0)\, dr.
\end{align*}
From \eqref{eq:first-SS-A-def} and the tower property of the conditional expectation we immediately see  $\E_u \delta A_{u,\xi,v} =0$ for $\xi \in [u,v]$, and thus
\eqref{eq:sewingassump1} is fulfilled with $\Gamma_1=0$.

It remains to check \eqref{eq:sewingassump2}.  
Let $\epsilon\in (0, \min(\frac{3}{4}-\frac{\alpha}{2},\alpha-1))$, which is possible since $\alpha\in (1,3/2)$ and which in particular implies that $\alpha+\eps-1\leq \frac{1}{2}-\eps$. Then, using the semigroup estimate $\norm{P_{t-r}v}_{L^{2}}\leq \norm{v}_{L^{2}}$, as well as \cref{lem:operator} and the embedding $\calC^{\alpha+\epsilon}\hookrightarrow H^{\alpha}$, we arrive at
\begin{align*}
    \|A_{u,v}\|_{L^2}&\leqslant \int_{u}^{v}\|(\operatorname{Id}-\Psi_N \Theta_N)(P^{\mathbb{R}}_{Q(r-u)} f)(P_{r-u} O_u+P_{r}\Psi_N\Theta_N u_0)\|_{L^2}\, dr
     \\&\lesssim N^{-\alpha} \int_{u}^{v} \norm{(P^{\mathbb{R}}_{Q(r-u)} f)(P_{r-u} O_u+P_{r}\Psi_N\Theta_N u_0)}_{\cC^{\alpha+\varepsilon}}\, dr.
\end{align*}
To simplify the calculation, we call $\mathfrak{a}_{u,r}:=P_{r-u}O_u$ for $u\leq r$ and $\mathfrak{b}:=P_{r}\Psi_N\Theta_N u_0$. Due to the semigroup estimate $\norm{P_tv}_{\calC^{\theta}}\lesssim \norm{v}_{\calC^{\theta}}$ for $\theta\in\R$ from \cref{lem:semigroup1}, the embedding $H^{s}(\T)\hookrightarrow \cC^{s-1/2}(\T)$ together with $1+\epsilon<\alpha<2$, as well as $\calC^{\theta}\hookrightarrow C^{1}_{b}$ for $\theta>1$ and $\calC^{\theta+\hat\epsilon}\hookrightarrow H^{\theta}$ for $\hat\epsilon>0$, and \cref{lem:operator} in the prenultimate inequality below, we find that
\begin{align}\label{eq:b-bound}
    \norm{\mathfrak{b}}_{C^{1}_{b}}\lesssim \norm{\mathfrak{b}}_{\calC^{\alpha-\epsilon}}\lesssim \norm{\Psi_N\Theta_N u_0}_{\calC^{\alpha-\epsilon}}\lesssim \norm{\Psi_N\Theta_N u_0}_{H^{\frac{5}{2}-\epsilon}}\lesssim \norm{u_0}_{H^{\frac{5}{2}-\epsilon}}\lesssim \norm{u_0}_{\calC^{\frac{5}{2}-\frac{\epsilon}{2}}},
\end{align}
where we have bounds on the $p$-th moment of the norm of the initial condition by \cref{asn:u0}.
Furthermore, due to \cref{lem:semigroup1} we have that 
\begin{align}\label{eq:a-bound}
    \norm{\mathfrak{a}_{u,r}}_{C^{1}_b}\lesssim\norm{\mathfrak{a}_{u,r}}_{\calC^{1+\epsilon}}\lesssim (r-u)^{-\frac{1}{4}-\epsilon}\norm{O_u}_{\calC^{\frac{1}{2}-\epsilon}}, \quad \norm{\mathfrak{a}_{u,r}}_{\calC^{\alpha+\epsilon}}\lesssim (r-u)^{-\frac{\alpha}{2}+\frac{1}{4}-\epsilon}\norm{O_u}_{\calC^{\frac{1}{2}-\epsilon}},
\end{align}
where we have bounds on the norm of the OU process on the right-hand side by \cref{lem:OUreg}.
With the above bounds on $\mathfrak{a},\mathfrak{b}$, we can further bound $A_{u,v}$ as follows. By the composition estimate from \cref{lem:composition} together with \cref{lem:composition-applies}, triangle inequality and using that $\alpha+\epsilon-1\leq\frac{1}{2}-\epsilon$ and $\alpha>1+\epsilon>1/2$, we obtain
\begin{align*}
\MoveEqLeft
\|A_{u,v}\|_{L^2}
   \\& \lesssim  N^{-\alpha} \int_{u}^{v} \norm{P^{\mathbb{R}}_{Q(r-u)} f}_{C^{2}_{b}}\Big(1+\norm{\mathfrak{a}_{u,r}+\mathfrak{b}}_{\cC^{\alpha+\varepsilon-1}}\norm{\mathfrak{a}_{u,r}+\mathfrak{b}}_{C^{1}_b}  +\norm{\mathfrak{a}_{u,r}}_{\cC^{\alpha+\varepsilon}} +\norm{\mathfrak{b}}_{\cC^{\alpha+\varepsilon}}\Big) \,dr
   \\ &\lesssim  N^{-\alpha} \int_{u}^{v} \norm{P^{\mathbb{R}}_{Q(r-u)} f}_{C^{2}_{b}}\Big(1+\norm{\mathfrak{a}_{u,r}}_{\calC^{\frac{1}{2}-\epsilon}}\norm{\mathfrak{a}_{u,r}}_{C^{1}_b}
 + \norm{\mathfrak{b}}_{\cC^{\frac{1}{2}-\epsilon}}\norm{\mathfrak{a}_{u,r}}_{C^{1}_b}
     + \norm{\mathfrak{a}_{u,r}}_{\cC^{\frac{1}{2}-\epsilon}}\norm{\mathfrak{b}}_{C^1_b}
    \\
    &\qquad\qquad  +\norm{\mathfrak{a}_{u,r}}_{\cC^{\alpha+\varepsilon}} +\norm{\mathfrak{b}}_{\cC^{\alpha+\varepsilon}}+  \norm{\mathfrak{b}}_{\cC^{\frac{1}{2}-\epsilon}}\norm{\mathfrak{b}}_{C^{1}_b} \Big) \,dr
    \\ &\lesssim  N^{-\alpha} \int_{u}^{v} \norm{P^{\mathbb{R}}_{Q(r-u)} f}_{C^{2}_{b}}\Big(1+\norm{\mathfrak{a}_{u,r}}_{\calC^{\frac{1}{2}-\epsilon}}\norm{\mathfrak{a}_{u,r}}_{C^{1}_b}
 + \norm{\mathfrak{b}}_{\cC^{\alpha-\epsilon}}\norm{\mathfrak{a}_{u,r}}_{C^{1}_b}
     + \norm{\mathfrak{a}_{u,r}}_{\cC^{\frac{1}{2}-\epsilon}}\norm{\mathfrak{b}}_{C^1_b}
    \\
    &\qquad\qquad  +\norm{\mathfrak{a}_{u,r}}_{\cC^{\alpha+\varepsilon}} +\norm{\mathfrak{b}}_{\cC^{\alpha+\varepsilon}}^2 \Big) \,dr.
\end{align*}
Using that  $\norm{P^{\R}_{Q(r-s)}f}_{C^{2}_{b}}\lesssim \norm{f}_{C^{2}_{b}}$ together with the bounds \eqref{eq:b-bound} and \eqref{eq:a-bound} and that $\alpha>1$, we find
\begin{align*}
\MoveEqLeft
  \|A_{u,v}\|_{L^2} \\&\lesssim  \norm{f}_{C^{2}_{b}} N^{-\alpha} \int_{u}^{v} \Big[1+(r-u)^{-\frac{1}{4}-\varepsilon}\|O_u\|_{\cC^{\frac{1}{2}-\varepsilon}}^2
 + (r-u)^{-\frac{1}{4}-\varepsilon}\|O_u\|_{\cC^{\frac{1}{2}-\varepsilon}}\norm{u_0}_{H^{\frac{5}{2}-\varepsilon}}
    \\&\qquad\qquad\qquad + \|O_u\|_{\cC^{\frac{1}{2}-\varepsilon}}\norm{u_0}_{H^{\frac{5}{2}-\varepsilon}}
  +(r-u)^{-\frac{\alpha}{2}+\frac{1}{4}-\epsilon}\norm{O_u}_{\cC^{\frac{1}{2}-\epsilon}}+\norm{u_0}_{H^{\frac{5}{2}-\varepsilon}}^2\Big] \, dr
    \\&\lesssim \norm{f}_{C^{2}_{b}} N^{-\alpha} (v-u)^{\frac{5}{4}-\frac{\alpha}{2}-\epsilon}(1+\norm{O}^2_{C_T\calC^{\frac{1}{2}-\epsilon}}+\norm{u_0}_{H^{\frac{5}{2}-\epsilon}}^2).
\end{align*}
By our choice of $\eps$,  $0<\frac{5}{4}-\frac{\alpha}{2}-\epsilon- \frac{1}{2}=:\eps_2$. After taking the $p$-th moment, using the  estimates on $O$ from \cref{lem:OUreg}, and the assumed regularity of $u_0$ from \cref{asn:u0}, we see that \eqref{eq:sewingassump2} is fulfilled with $\delta_2=0$ and $\Gamma_2=CN^{-\alpha}$.
Then the stochastic sewing lemma applies and  \eqref{eq:resultSewing} yields the desired bound for the integral $\mathcal{A}$ corresponding to the germ $A$ above. 

It remains to check that $(\mathcal{A}_u)_{u \in [s,t]}$, is given by $\tilde{\mathcal{A}}_{u}=\int_{s}^{u}P_{t-r}(\operatorname{Id}-\Psi_N\Theta_N)f(O_r+P_r\Psi_N\Theta_N u_0)\,dr$.
This follows immediately from the characterisation through \eqref{eq:SSL-conc1}-\eqref{eq:SSL-conc2}. Indeed, $\tilde{\mathcal{A}}$ satisfies \eqref{eq:SSL-conc2} with $C_3=0$ and satisfies \eqref{eq:SSL-conc1} with $C_1=0$ and $C_2=K$, using the boundedness of $f$.
Thus, $\mathcal{A}=\tilde{\mathcal{A}}$ and this finishes the proof of the first inequality in the lemma. The second inequality in the lemma follows from an application of a version of Kolmogorov's continuity, see \cref{prop:vKolmogorov}.
\end{proof}
The content of the next corollary is to infer a bound for $E_2$ in \eqref{eq:error-decomp}.
\begin{corollary}\label{cor:girsanov}
Let $p\geq 1$ and $\alpha\in (1,3/2)$. Then there exists $C=C(T,\alpha, p)$, such that 
\begin{align}
   \paren[\bigg]{\E \sup_{t\in[0,T]}\norm[\bigg]{\int_{0}^{t} P_{t-s} (\operatorname{Id}- \Psi_N \Theta_N)f(z^N_s+O_s) \,ds}_{L^2(\T)}^p}^{1/p}\leq C N^{-\alpha}.
\end{align}
\end{corollary}
\begin{proof}
The argument is almost identical to \cite[Corollary~3.3.2]{BDG-SPDE} or  \cite[Corollary~4.8]{DjGK}.
Denote $D_t=\Psi_N\Theta_N f(z^N_t+O_t)$. Since deterministically
$$
\int_{0}^{T}\|D_t\|_{L^2(\T)}\, dt\lesssim \int_{0}^{T}\| f(z^N_t+O_t)\|_{L^{\infty}(\T)}\, dt\lesssim 1,
$$
by an application of Girsanov's theorem \cite[Theorem 10.14]{DPZ},
$\xi(dy,ds)+ D_s(y)\,dyds$ defines a space-time white noise under a measure $\mathbb{Q}$ whose Radon-Nikodym derivative is explicit and has finite moments of any positive or negative order. As a consequence, for any nonnegative Borel function $g$ defined on the space of continuous functions on $[0,T]$ one has
\begin{equ}\label{eq:Girsanov-ineq}
   \E\big|g(O+z^N)|^p= \E\Big|g\Big(O+\int_0^\cdot P_{\cdot-s}D_s\,ds+P_{\cdot}\Psi_N\Theta_N u_0\Big)\Big|^p\lesssim  \big(\E\abs{g(O+P_{\cdot}\Psi_N\Theta_N u_0)}^{2p}\big)^{1/2}.
\end{equ}
Applying this with the functional
\begin{align*}
        g(Z)\coloneqq \sup_{t\in[0,T]}\norm[\bigg]{\int_{0}^{t} P_{t-s} (\operatorname{Id}- \Psi_N \Theta_N)f(Z_s) \, ds}_{L^2(\T)}, 
    \end{align*}
    and using \cref{lem:space-ss}, we get the claim.
\end{proof}

We conclude on the spatial error bound in the following lemma.
\begin{lemma}\label{lem:spatial}
Let $p\geq 1$ and $\alpha\in (1,3/2)$. Then there exists $C=C(T,\alpha, p,K, \mathcal{M})$, that does not depend on $N$, such that 
\begin{align}
    (\E\norm{\Theta_N v_t - v^N_t}_{L^{2}(\Pi_N)}^p)^{1/p}\leq C N^{-\alpha}.
\end{align}    
\end{lemma}
\begin{proof}

Recall that by \cref{prop:wp} it holds that for any $\epsilon>0$, $$(\E\sup_{t\in[0,T]}\norm{v_t}_{\calC^{5/2-\epsilon}(\mathbb{T})}^p)^{1/p}<\infty.$$
Using the isometry property of the extension operator $\Psi_N$, triangle inequality and \cref{lem:operator} for $\alpha=5/2-\epsilon$, we arrive at
\begin{align}\label{eq:aprioriv}
   (\E\norm{\Theta_N v_t - v^N_t}_{L^{2}(\Pi_N)}^p)^{1/p}&\leq  (\E\norm{v_t - z^N_t}_{L^{2}(\mathbb{T})}^p)^{1/p} +  (\E\norm{(\operatorname{Id}-\Psi_N\Theta_N)v_t}_{L^{2}(\mathbb{T})}^p)^{1/p}
   \\&\lesssim (\E\norm{v_t - z^N_t}_{L^{2}(\mathbb{T})}^p)^{1/p} +  N^{-5/2+\epsilon}(\E\norm{v_t}_{H^{5/2-\epsilon}(\mathbb{T})}^p)^{1/p},
\end{align}
where the second summand is finite by the recalled regularity bound of $v$ above. Further recall the error decomposition \eqref{eq:error-decomp} into $E_0, E_1, E_2$.
We are left to bound the terms $E_0$ and $E_1$, as $E_2$ is bounded by \cref{cor:girsanov}.
To bound $E_0$, we use again \cref{lem:operator} for $\alpha=5/2-\eps$ and the assumed regularity on the initial condition, as well as the bound $\norm{P_t v}_{L^{2}}\leq\norm{v}_{L^{2}}$, which yields, for any $T'\in[0,T]$
\begin{align*}
    (\E\sup_{t\in[0,T']}\abs{E_0(t)}^p)^{1/p}\leq (\E\norm{(\operatorname{Id}-\Psi_N\Theta_N)u_0}_{L^2(\T)}^p)^{1/p}\lesssim N^{-5/2+\epsilon}(\E\norm{u_0}_{H^{5/2-\epsilon}(\T)}^p)^{1/p}.
\end{align*}
The term $E_1$ is bounded by
\begin{align*}
    (\E\sup_{t\in[0,T']}\abs{E_1(t)}^p)^{1/p}\leq K\int_{0}^{T'}(\E\sup_{r\in[0,s]}\norm{v_r-z^N_r}_{L^2(\T)}^p)^{1/p}ds
\end{align*}
using the global Lipschitz bound on $f$. 
Together, we arrive at
\begin{align}
   (\E\sup_{t\in[0,T']}\|v_t&-z^N_t\|_{L^2(\mathbb{T})}^p)^{1/p}\leqslant C N^{-\alpha}+    K\int_{0}^{T'}(\E\sup_{r\in[0,s]}\norm{v_r-z^N_r}_{L^2(\T)}^p)^{1/p}\, ds
\end{align}
and an application of Grönwall's inequality yields the claim.
\end{proof}

\subsection{Temporal error} \label{sec:temporalerror}

Now we estimate the temporal error in \eqref{eq:errordecomposed}. Note that
\begin{align*}
    \|v^N_t-V^{M,N}_t\|_{L^2(\Pi_N)}&= \Big\|\int_0^t {P}^N_{t-s}\big[ f(v^N_s +\Theta_N O_s) - f(V^{M,N}_{k_M(s)} + \Theta_N O_{k_M(s)})\big]\, ds\Big\|_{L^2(\Pi_N)}\\
    &\leqslant \Big\| \int_0^t \tilde{P}_{t-s}^N \big[f(v_s^N +\Theta_N O_s) - f (v^N_{k_M(s)}+ \Theta_N O_{k_M(s)})\big]\, ds \Big\|_{L^2(\Pi_N)}\\
    &\quad + \Big\|\int_0^t \tilde{P}_{t-s}^N \big[ f(v^N_{k_M(s)} + \Theta_N O_{k_M(s)})- f(V^{M,N}_{k_M(s)} + \Theta_N O_{k_M(s)})\big]\, ds\Big\|_{L^2(\Pi_N)}
\end{align*}
 Our strategy in bounding the temporal error is as follows. First, by an application of Girsanov's theorem, in \cref{cor:Girsanovtemporal}, bounding the first summand in the above can be reduced to bounding
\begin{align}
    \norm[\bigg]{\int_0^t \tilde{P}_{t-r}^N \big[f(\Theta_N O_r+\tilde{P}_r^N\Theta_N u_0) - f (\Theta_N O_{k_M(r)}+\tilde{P}_{k_{M}(r)}^N\Theta_N u_0)\big]\, dr}_{L^2(\Pi_N)},
\end{align}
which is done in \cref{lem:temp-ss}. Then, for the overall temporal error, we will conclude with a version of Gr\"onwall's lemma in \cref{lem:temporal} below.
\begin{lemma}\label{lem:temp-ss}
Let $p\geqslant 1$, $\eps\in (0,\frac{1}{2})$. Then there exists a constant $C=C(\epsilon, p, T, K, \mathcal{M})$, such that,
\begin{align*}
    \paren[\bigg]{\E\sup_{t\in[0,T]}\norm[\bigg]{\int_0^t \tilde{P}_{t-r}^N \big[f&(\Theta_N O_r+\tilde{P}_r^N\Theta_N u_0) - f (\Theta_N O_{k_M(r)}+\tilde{P}^N_{k_{M}(r)}\Theta_N u_0)\big]\, dr}_{L^2(\Pi_N)}^p}^{1/p}\\
    &\leqslant C (N^{-2+\eps}+M^{-1+\eps}).
\end{align*}
\end{lemma}

\begin{proof}
We follow similar arguments as in the proof of \cite[Proposition 4.5]{DjGK}. Note that in contrast to the proof of \cite[Proposition 4.5]{DjGK}, we are now estimating in discrete Sobolev spaces $H^{\alpha}(\Pi_N)$ and the relevant a priori bounds, which are being used, are the ones for the OU noise $\Theta_N O$ restricted to the spatial grid, which follow from the estimates for $O$ from \cref{lem:OUreg} in spaces of positive regularity, while in spaces of negative regularity, the estimates follow from \cref{lem:Otreg}.

Let $\hat{O}_r\coloneqq \Theta_N O_r+\tilde{P}_r^N\Theta_N u_0$ and fix arbitrary $(s,t) \in [0,T]^2_\leq$. We aim to apply Lemma~\ref{lem:SSLHilbert}. To that aim, we define the germ for $(u,v) \in [s,t]_\leq$,
    \begin{align*}
        A_{u,v}\coloneqq \E_u \int_u^v \tilde{P}_{t-r}^N\big[ f(\hat{O}_r)-f(\hat{O}_{k_M(r)})]\, dr.
    \end{align*}
We immediately see that $\E_u[\delta A_{u,\xi,v}]=0$ for $\xi \in [u,v]$ and thus \eqref{eq:sewingassump1} is satisfied with $\Gamma_1=0$.
In order to verify \eqref{eq:sewingassump2}, we first assume that $|v-u|\leqslant 3 M^{-1}$. Using the semigroup estimate \eqref{eq:tildeP1} from \cref{lem:semigroup2} for $\tilde{P}^N_{t-r}$, the product estimate from Lemma~\ref{lem:productestnegative} and Cauchy-Schwarz, we get 
\begin{align}\label{eq:A-bound}
    \|\|&A_{u,v}\|_{L^2(\Pi_N)}\|_{L^p(\Omega)}\leqslant (t-v)^{-1/4+\eps/2}\int_u^v \|\|f(\hat{O}_r)-f(\hat{O}_{k_M(r)})\|_{H^{-1/2+\eps}(\Pi_N)}\|_{L^p(\Omega)}\, dr\nonumber\\
    &\lesssim  (t-v)^{-1/4+\eps/2}(v-u) \sup_{t \in [0,T]}\Big\|\Big\|\int_0^1 f^\prime(\lambda \hat{O}_t+ (1-\lambda) \hat{O}_{k_M(t)})\, d\lambda \Big\|_{\cC^{1/2-\eps/2}(\Pi_N)}\Big\|_{L^{2p}(\Omega)}\nonumber \\
    &\quad \quad \quad \times \sup_{t \in [0,T]}\| \|\hat{O}_t-\hat{O}_{k_M(t)}\|_{H^{-1/2+\eps}(\Pi_N)}\|_{L^{2p}(\Omega)}.
\end{align}
The first of the two factors to bound in the above is bounded by the composition estimate Lemma~\ref{lem:composition} with $f'\in C^{1}_{b}$ and \cref{lem:OUreg}, as well as \cref{lem:Theta-Psi} f), as follows
\begin{align}\label{eq:f-bound}
\MoveEqLeft
    \sup_{t \in [0,T]}\Big\|\Big\|\int_0^1 f^\prime(\lambda \hat{O}_t+ (1-\lambda) \hat{O}_{k_M(t)})\, d\lambda \Big\|_{\cC^{1/2-\eps/2}(\Pi_N)}\Big\|_{L^{2p}(\Omega)} \nonumber
    \\&\lesssim \norm{f'}_{C^{1}_{b}}(1+ \norm{\norm{\hat{O}}_{\calC^{1/2-\epsilon/2}(\Pi_N)}}_{L^{2p}(\Omega)})\nonumber 
    \\&\lesssim \norm{f'}_{C^{1}_{b}}(1+ \norm{\norm{O}_{\calC^{1/2-\epsilon/2}(\T)}}_{L^{2p}(\Omega)}+\norm{\norm{\tilde{P}_t^{N}\Theta_Nu_0}_{\calC^{1/2-\epsilon/2}(\Pi_N)}}_{L^{2p}(\Omega)}) 
\end{align}
and we further bound, using the embedding $H^{1}(\Pi_N)\hookrightarrow\calC^{\frac{1}{2}-\frac{\eps}{2}}(\Pi_N)$, the semigroup estimates for $\tilde{P}^{N}$, the embedding $\calC^{1+\eps}(\Pi_N)\hookrightarrow H^{1}(\Pi_N)$ and \cref{lem:Theta-Psi} f),
\begin{align*}
    \norm{\tilde{P}_t^{N}\Theta_Nu_0}_{\calC^{1/2-\epsilon/2}(\Pi_N)}\lesssim \norm{\tilde{P}_t^{N}\Theta_Nu_0}_{H^1(\Pi_N)}\lesssim \norm{\Theta_Nu_0}_{H^1(\Pi_N)}\lesssim \norm{\Theta_N u_0}_{\calC^{1+\eps}(\T)}\lesssim \norm{u_0}_{\calC^{\frac{5}{2}}(\T)}.
\end{align*}
Additionally using Lemma~\ref{lem:Otreg} for $\alpha=1/2-2\epsilon$ for the first summand in the below, the isometry property of $\Psi_N$ and that $P_t \Psi_N=\Psi_N \tilde{P}_t^N$ (see \cref{lem:Theta-Psi}) with semigroup estimates for the second summand in the below, we have
\begin{align}\label{eq:tildeO}
    &\sup_{r \in [0,T]}\| \|\hat{O}_r-\hat{O}_{k_M(r)}\|_{H^{-1/2+\eps}(\Pi_N)}\|_{L^{2p}(\Omega)}\nonumber\\
    &\leqslant \sup_{r \in [0,T]}\| \|\Theta_N O_r-\Theta_N O_{k_M(r)}\|_{H^{-1/2+\eps}(\Pi_N)}\|_{L^{2p}(\Omega)}+\sup_{r \in [0,T]}\| \|\tilde{P}^N_r\Theta_Nu_0-\tilde{P}^N_{k_M(r)}\Theta_Nu_0\|_{L^2(\T)}\|_{L^{2p}(\Omega)}\nonumber\\
    &\leqslant \squeeze[1]{\sup_{r \in [0,T]}\| \|\Theta_N O_r-\Theta_N O_{k_M(r)}\|_{H^{-1/2+\eps}(\Pi_N)}\|_{L^{2p}(\Omega)}+\sup_{r \in [0,T]}\| \|P_r \Psi_N\Theta_Nu_0-P_{k_M(r)}\Psi_N\Theta_Nu_0\|_{L^2(\T)}\|_{L^{2p}(\Omega)}}\nonumber\\
    &\lesssim \max\{N^{-1+\eps},M^{-\frac{1}{2}+\eps}\}+\sup_{r \in [0,T]}(r-k_M(r))\|\|\Psi_N \Theta_N u_0\|_{H^{2}(\T)}\|_{L^{2p}(\Omega)}\nonumber\\
    &\lesssim \max\{N^{-1+\eps},M^{-\frac{1}{2}+\eps}\}+M^{-1}\|\| u_0\|_{\cC^{5/2-\varepsilon}(\T)}\|_{L^{2p}(\Omega)},
\end{align}
where we used \eqref{eq:PsiThetabounded} together with a Besov embedding in the last line.
Hence, inserting the bounds \eqref{eq:f-bound} and \eqref{eq:tildeO} into the estimate \eqref{eq:A-bound}, we see that together with $|v-u|\lesssim M^{-1}$,
\begin{align*} \|\|A_{u,v}\|_{L^2(\Pi_N)}\|_{L^p(\Omega)}&\lesssim (t-v)^{-1/4+\eps/2} (v-u) \max\{N^{-1+\eps},M^{-\frac{1}{2}+\eps}\}\\
&\lesssim (t-v)^{-1/4+\eps/2} (v-u)^{1/2+\eps}M^{-\frac{1}{2}+\eps}\Big(N^{-1+\eps}+M^{-\frac{1}{2}+\eps}\Big)\\
&\lesssim (t-v)^{-1/4+\eps/2} (v-u)^{1/2+\eps}\Big(N^{-2+2\eps}+M^{-1+2\eps}\Big),
\end{align*}
where we used that $ab\leq a^2 +b^2$ for $a, b \in \R$ in the last line. Thus \eqref{eq:sewingassump2} is satisfied in the case of $\abs{v-u}\leq 3 M^{-1}$.
It is left to consider the case $|v-u|\geqslant 3 M^{-1}$. Let $t^\prime=k_{M}(u)+3M^{-1}$. We decompose
\begin{align*}
\int_u^v \tilde{P}_{t-r}^N[f(\hat{O}_r)-f(\hat{O}_{k_M(r)})]\, dr= \int_u^{t^\prime} \tilde{P}_{t-r}^N[f(\hat{O}_r)-f(\hat{O}_{k_M(r)})]\, dr + \int_{t^\prime}^{v} \tilde{P}_{t-r}^N[f(\hat{O}_r)-f(\hat{O}_{k_M(r)})]\, dr.
\end{align*}
The first summand in the above can be bounded like in the first case as by construction $|t^\prime - u|\leqslant 3 M^{-1}$. Hence, it remains to bound the second summand.

By Lemma~\ref{lem:QQN}, we have that $\mathrm{Var}(\Theta_N O_r-\Theta_N P_{r-u} O_u)=Q(r-u)$, since $Q$ does not depend on $x$. Further, we can compute the conditional expectation as follows, using again the rule \eqref{eq:rule} for each fixed $x\in\Pi_N$, together with the assumption that $f$ is of Nemytskii type, to deduce that almost surely
\begin{align*}
    \E_u&\int_{t^\prime}^v \tilde{P}_{t-r}^N\big[f(\hat{O}_r)-f(\hat{O}_{k_M(r)})\big]\, dr\\
    &=\int_{t^\prime}^v \tilde{P}_{t-r}^N \big[(P^\mathbb{R}_{Q(r-u)} f)(Y^N_{u,r}+\tilde{P}_r^N\Theta_N u_0)-(P^\mathbb{R}_{Q(r-u)} f)(Y^N_{u,k_{M}(r)}+\tilde{P}_{k_{M}(r)}^N\Theta_N u_0)\big]\, dr\\
    &\quad +\int_{t^\prime}^v \tilde{P}_{t-r}^N \paren[\Big]{[P^\mathbb{R}_{Q(r-u)} f- P^\mathbb{R}_{Q(k_M(r)-u)} f](Y^N_{u,k_{M}(r)}+\tilde{P}_{k_{M}(r)}^N\Theta_N u_0)}\, dr 
    \\&\eqqcolon E_1+E_2,
\end{align*}
where we employed the simplified notation $Y^N_{u,r}=\Theta_N P_{r-u}O_u$ for $u\leq r$.
We first bound $E_1$. Proceeding similarly as in the first case of the proof, we have
\begin{align*}
    &\|\|E_1\|_{L^2(\Pi_N)}\|_{L^p(\Omega)} \lesssim (t-v)^{-1/4+\eps/2}\times \\
    &\squeeze[1]{\int_{t^\prime}^v\!\!\Big\|\Big\|\int_0^1 \!\!(P^{\mathbb{R}}_{Q(r-u)} f)^\prime (\lambda (Y^N_{u,r}+\tilde{P}_r^N\Theta_N u_0) + (1-\lambda) (Y^N_{u,k_{M}(r)}+\tilde{P}_{k_{M}(r)}^N\Theta_N u_0))\, d\lambda\Big\|_{\cC^{1/2-\eps/2}(\Pi_N)}\Big\|_{L^{2p}(\Omega)}}\\
    &\times \|\|Y^N_{u,r}+\tilde{P}_r^N\Theta_N u_0-Y^N_{u,k_{M}(r)}-\tilde{P}_{k_{M}(r)}^N\Theta_N u_0) \, \|_{H^{-1/2+\eps}(\Pi_N)}\|_{L^{2p}(\Omega)}\,dr.
\end{align*}
 Again, the first factor is bounded as $f \in C^2_b$, the composition estimate from Lemma~\ref{lem:composition} and \eqref{eq:Otimespace} in Assumption~\ref{ass:tildeO}. Recall that for $v \in C(\T)$, $\|\Theta_N v\|_{L^\infty(\Pi_N)}\leqslant \|v\|_{L^\infty(\T)}$ (see Lemma~\ref{lem:Theta-Psi}). Combining this with \eqref{eq:Ot-Os2} for $\beta=1-2\eps$, $\alpha=\frac{1}{2}-2\epsilon$, we find that
\begin{align}\label{eq:semigroupSSL}
  \norm{\norm{Y^N_{u,r}-Y^N_{u,k_M(r)}}_{H^{-1/2+\eps}(\Pi_N)}}_{L^{2p}(\Omega)}
  \lesssim (k_M(r)-u)^{-1/2+\eps}(N^{-2+3\eps}+M^{-1+2\eps})
\end{align}
Furthermore, by the same arguments as in \eqref{eq:tildeO}, we obtain that
\begin{align*}
    \norm{\tilde{P}_r^N\Theta_N u_0-\tilde{P}_{k_{M}(r)}^N\Theta_N u_0}_{L^2(\Pi_N)}
    &\lesssim (r-k_{M}(r))\norm{u_0}_{\calC^{5/2-\epsilon}(\T)}.
\end{align*}
After integration in $r$ and using that $k_{M}(r)-u\geq (r-u)/2$, we thus find
\begin{align*}
    \|\|E_1\|_{L^2(\Pi_N)}\|_{L^p(\Omega)} 
    \lesssim  (t-v)^{-1/4+\eps/2}(v-u)^{1/2+\epsilon}(N^{-2+3\eps}+M^{-1+2\eps}).
\end{align*}
We estimate $E_2$ as follows
\begin{align*}
   \norm{\norm{E_2}_{L^2(\Pi_N)}}_{L^{p}(\Omega)}&\leq \int_{t^\prime}^v \|\norm{[P^\mathbb{R}_{Q(r-u)} f-P^\mathbb{R}_{Q(k_M(r)-u)} f](Y^N_{u,k_{M}(r)}+\tilde{P}_{k_{M}(r)}^N\Theta_N u_0)}_{L^{2}(\Pi_N)}\|_{L^p(\Omega)} \, dr
   \\&\lesssim \int_{t^\prime}^v \norm{P^\mathbb{R}_{Q(r-u)} f-P^\mathbb{R}_{Q(k_M(r)-u)} f}_{C^{0}_b(\R)} \, dr
\end{align*}
and we use the semigroup estimates for $P^{\R}$ from \eqref{eq:R-semigroup}, as well as the bounds for $Q$ from \cref{lem:QQN} to bound 
\begin{align*}
    \norm{P^\mathbb{R}_{Q(r-u)} f-P^\mathbb{R}_{Q(k_M(r)-u)} f}_{C^{0}_b(\R)} &\lesssim (Q(r-u)-Q(k_{M}(r)-u)) \norm{f}_{C^{2}_{b}(\R)}
    \\&\lesssim (r-k_{M}(r))^{1-\epsilon}(k_{M}(r)-u)^{-1/2+\epsilon}.
\end{align*}
After integration in $r$ and using that $k_{M}(r)-u\geq (r-u)/2$, we thus find
\begin{align*}
    \|\|E_2\|_{L^2(\Pi_N)}\|_{L^p(\Omega)} 
    \lesssim  M^{-1+\epsilon}(v-u)^{1/2+\epsilon}.
\end{align*}
Together we obtain \eqref{eq:sewingassump2} for $\delta_2=\frac{1}{4}-\frac{\epsilon}{2}$, $\Gamma_2= C(T,\epsilon, p, \mathcal{M}, K) (N^{-2+3\eps}+M^{-1+2\eps}) $, $\epsilon_2=\epsilon$.
Thus, we can apply the stochastic sewing lemma to the germ $A_{u,v}$. Since $f$ is bounded, by arguments as in \cref{lem:space-ss}, the corresponding process $\mathcal{A}$ is given by $\mathcal{A}_u=\int_{s}^{u}\tilde{P}^{N}_{t-r} \big[f(\hat{O}_r) - f (\hat{O}_{k_M(r)})\big]\, dr$. Thus we obtain from \eqref{eq:resultSewing},
\begin{align*}
    \paren[\bigg]{\E\norm[\bigg]{\int_s^t \tilde{P}_{t-r}^N \big[f(\hat{O}_r) - f (\hat{O}_{k_M(r)})\big]\, dr}_{L^2(\Pi_N)}^p}^{1/p}\leqslant C (t-s)^{\frac{1}{4}+\frac{\eps}{2}}(N^{-2+3\eps}+M^{-1+2\eps}).
\end{align*}
An application of a version of Kolmogorov's continuity theorem, Proposition~\ref{prop:vKolmogorov}, yields the result.
\end{proof}

\begin{corollary}\label{cor:Girsanovtemporal}
Let $p\geqslant 1$, $\epsilon\in (0,1/2)$. Then there exists a constant $C=C(\epsilon, p, T, K, \mathcal{M})$, such that 
\begin{align*}
    \paren[\bigg]{\E\sup_{t\in[0,T]}\norm[\bigg]{\int_0^t \tilde{P}_{t-r}^N \big[f(v_r^N+\Theta_N O_r) - f (v_r^N+\Theta_N O_{k_M(r)})\big]\, dr}_{L^2(\Pi_N)}^p}^{1/p}\leqslant C (N^{-2+\eps}+M^{-1+\eps}).
\end{align*}
\end{corollary}
\begin{proof}
The proof follows that of \cref{cor:girsanov}, we only change the choice of $g$ to:
\begin{align}
    g(Z)\coloneqq \sup_{t\in[0,T]}\norm[\bigg]{\int_0^t \tilde{P}_{t-r}^N \big[f(\Theta_N Z_r) - f (\Theta_N Z_{k_M(r)})\big]\, dr}_{L^2(\Pi_N)}^p. 
\end{align}
By \eqref{eq:Girsanov-ineq} and \cref{lem:temp-ss} we get the claim.
\end{proof}

\begin{lemma}\label{lem:temporal}
Let $p\geqslant 1$, $\epsilon\in (0,1/2)$. Then there exists a constant $C=C(\epsilon, p, T, K, \mathcal{M})$, such that 
\begin{align*}
    (\E \sup_{t\in[0,T]}\norm{v^N_t-V^{M,N}_t}_{L^2(\Pi_N)}^p)^{1/p}\leq C (N^{-2+\eps}+M^{-1+ \eps}).
\end{align*}
\end{lemma}

\begin{proof}
Recall that
\begin{align*}
 v^N_t-&V^{M,N}_t=\int_0^t \tilde{P}_{t-s}^N \big[f(v_s^N +\Theta_N O_s) - f (v^N_{k_M(s)}+ \Theta_N O_{k_M(s)})\big]\, ds\\
    &\quad + \int_0^t \tilde{P}_{t-s}^N \big[ f(v^N_{k_M(s)} + \Theta_N O_{k_M(s)})- f(V^{M,N}_{k_M(s)} + \Theta_N O_{k_M(s)})\big]\, ds.
\end{align*}
After applying a version of Gr\"onwall's lemma \cite[Proposition 3.4]{DjGK}, which is applicable since $f$ is globally Lipschitz and \cref{cor:Girsanovtemporal}, we get
\begin{align*}
    \|&\sup_{t \in [0,T]}\|v_t^N-V_t^{M,N}\|_{L^2(\Pi_N)}\|_{L^p(\Omega)}\\
    &\leqslant C \Big\|\Big\|\sup_{t \in [0,T]}\int_0^t \tilde{P}_{t-s}^N \big[f(v_s^N +\Theta_N O_s) - f (v^N_{k_M(s)}+ \Theta_N O_{k_M(s)})\big]\, ds\Big\|_{L^2(\Pi_N)}\Big\|_{L^p(\Omega)}\\
    &\leqslant C (N^{-2+\eps}+M^{-1+\eps}).\qedhere
\end{align*}
\end{proof}

\begin{proof}[Proof of \cref{thm:main1}]
    The result follows by Lemma~\ref{lem:approximation}, Lemma~\ref{lem:spatial},  Lemma~\ref{lem:temporal} and Assumption~\ref{ass:tildeO} \ref{en:b} due to the error decomposition in \eqref{eq:errordecomposed}.
\end{proof}

\begin{remark}\label{rem:rates}
    Comparing the two stochastic sewing arguments in the proofs of \cref{lem:space-ss} and \cref{lem:temp-ss}, respectively, one can notice the subtle difference leading to the error rates not showing a parabolic scaling.
    Indeed, the integrand in the spatial error estimate in the proof of \cref{lem:space-ss} is estimated in $L^2(\T)$ and not in a weaker space, which is an effect of the restrictions on the exponents in \cref{lem:operator} (as opposed to, for example \eqref{eq:projection-bound}).
    In contrast, the integrand after the semigroup in the proof of \cref{lem:temp-ss} is estimated in $H^{-1/2+\eps}(\Pi_N)$, which is possible thanks to the regularity bounds on $\Theta_N O$ from \cref{lem:Otreg}. This extra room of regularity is leveraged via stochastic sewing to a temporal rate that is better than half of the spatial one.
\end{remark}

\section{Nonlinearity with superlinear growth}\label{bigsec:superlinear}
Throughout the section \cref{ass:tildeO}, \cref{asn:u0}, and \cref{ass2} are imposed.

Although the proof of Theorem \ref{thm:main2} shares some features with that of Theorem \ref{thm:main1}, the implementation of the strategy is quite a bit more involved. In addition to $w^N$ (see \eqref{eq:bfvN}) and $X^{M,N}$ (see \eqref{eq:auxscheme}), several further intermediate processes will be introduced, which we first list and briefly describe for the reader's convenience. Next, we will define those processes in detail and provide the main error decomposition.
\begin{itemize}
    \item The process $v^M$ is the solution of the continuum equation with modified nonlinearity $g_h=g_{T/M}$ from \eqref{eq:gh} in place of $f$.
    \item The process $w^{M,N}$ is the spatially discrete approximation of $v^M$.
    \item The process $z^{M,N}$ is the continuum extension of $w^{M,N}$. 
    \item To define $\mathbf{z}^{M,N}$, we replace in the mild form of $z^{M,N}$ the operator $\Psi_N P^N$ by $P\Psi_N$. Unlike in the case of the semigroup $\tilde P^N$, here these two operators are different.
    \item To define $\hat{\mathbf{z}}^{M,N}$, we remove from the mild form of $\mathbf{z}^{M,N}$ an operator $\Psi_N\Theta_N$.
\end{itemize}

To make the above more precise, define $v^M$ as the solution of
\begin{equ}
    \label{eq:step1}
     \partial_t v^M=\Delta  v^M+ g_h( v^M+O), \qquad  v^M(0,x)=u_0(x).
\end{equ}
In mild form, the equation reads as
\begin{equ}
     v^M_t=P_t u_0+\int_0^t P_{t-s}\big(g_h( v^M_s+O_s)\big)\,ds.
\end{equ}
The spatial discretisation of $v^M$ is then given by $w^{M,N}$, defined as the solution of
\begin{equ}\label{eq:step2}
    \partial_t w^{M,N}=\Delta_N w^{M,N}+g_h(w^{M,N}+\Theta_N O), \qquad w^{M,N}(0,x)=\Theta_N u_0(x).
\end{equ}
In mild form, the equation reads as
\begin{equ}
   w^{M,N}_t= P^N_t (\Theta_N u_0)+\int_0^t  P^N_{t-s}\Big( g_h(w^{M,N}_s+\Theta_N O_s)\Big)\,ds.
\end{equ}
Define the extension of $w^{M,N}$ by $z^{M,N}\coloneqq\Psi_N w^{M,N}$. By definition it satisfies
\begin{equ}
    z^{M,N}_t=\Psi_N P_t^N(\Theta_Nu_0)+\int_0^t\Psi_N P_{t-s}^N\Big(\Theta_N \big(g_h( z^{M,N}_s+O_s)\big)\Big)\,ds.
\end{equ}
For the well-posedness of the equations for $v^M$ and $w^{M,N}$ (and thus $z^{M,N}$) and suitable a priori estimates, we refer to \cref{lem:wpvM} and \cref{lem:apriori-WMN} below.

The expression for $z^{M,N}$ does not correspond to any continuum equation, so that we introduce two auxiliary processes $\mathbf{z}^{M,N}$ and $\hat{\mathbf{z}}^{M,N}$ that correspond to continuum equations and are therefore easier to compare to $v^M$. Set
\begin{equ}
    \mathbf{z}^{M,N}_t= P_t(\Psi_N\Theta_Nu_0)+\int_0^tP_{t-s}\Big(\Psi_N\Theta_N \big(g_h( z^{M,N}_s+O_s)\big)\Big)\,ds.
\end{equ}
and
\begin{equ}
    \hat{\mathbf{z}}^{M,N}_t= P_t(\Psi_N\Theta_Nu_0 )+\int_0^tP_{t-s}\big(g_h( z^{M,N}_s+O_s)\big)\,ds.
\end{equ}
Recall the splitting scheme $\tilde{X}^{M,N}$ from \eqref{eq:scheme1-mild} and the auxiliary scheme $X^{M,N}$ from \eqref{eq:auxscheme}.
We decompose  the error in Theorem~\ref{thm:main2} in two steps, separating the error made by switching $O$ to $\tilde O$ from the rest:
\begin{equs}
\MoveEqLeft\|\Theta_N u_t-(\tilde{X}_t^{M,N}+\Theta_N \tilde{O}_t)\|_{L^2(\Pi_N)}
    \\&\leqslant \|\Theta_N v_t - X^{M,N}_t\|_{L^2(\Pi_N)}+ \norm{X^{M,N}-\tilde{X}^{M,N}}_{L^2(\Pi_N)}+\|\Theta_N(O_t-\tilde{O}_t)\|_{L^2(\Pi_N)}\label{SMALL-decomposition}
\end{equs}
and decompose the rest of the error as
\begin{equs}
\MoveEqLeft
    \|\Theta_N v_t - X^{M,N}_t\|_{L^2(\Pi_N)}=\|\Psi_N\Theta_N v_t -\Psi_N X^{M,N}_t \|_{L^2(\T)}
    \\
    &\leq 
\|(\Psi_N\Theta_N -\Id) v_t\|_{L^2(\T)}+\|v_t-v^M_t\|_{L^2(\T)}+ \|v^M_t-\hat{\mathbf{z}}^{M,N}_t\|_{L^2(\T)}    \\
    &\quad+ \|\hat{\mathbf{z}}^{M,N}_t-\mathbf{z}^{M,N}_t\|_{L^2(\T)}+\|\mathbf{z}^{M,N}_t-z^{M,N}_t\|_{L^2(\T)}
    +\|z^{M,N}_t -\Psi_N X^{M,N}_t\|_{L^2(\T)}\label{BIG-decomposition}.
\end{equs}
We bound the second and third terms on the right-hand side of \eqref{SMALL-decomposition} them by \cref{ass:tildeO} and \cref{lem:XtildeX} in the end of \cref{subsec:apriori}.
In Section~\ref{subsec:spatial}, we bound the spatial error terms in \eqref{BIG-decomposition} and in Section~\ref{subsec:temporal} the temporal error terms. We conclude with a proof of \cref{thm:main2} in the very end of \cref{subsec:temporal}.

\subsection{A priori estimates and approximation error}\label{subsec:apriori}

We start by showing well-posedness and a priori estimates for the corresponding solutions to \eqref{eq:bfvN}, \eqref{eq:step1}, and \eqref{eq:step2}.
Below we often use equivalence of mild and weak solutions of these equations, which follows from \cite{bell}. 
\begin{lemma} \label{lem:wpvM}
    There exists a unique mild solution to \eqref{eq:step1}.
    Moreover for all $p\geq 1$ and $\eps>0$, $\lambda\in (0,2]$, there exists a constant $C=C(p,\eps,\lambda, T,K,m,\mathscr{M})$, that does not depend on $M$, such that
    \begin{align}
    \big\|\|v^M\|_{C_T^{\lambda/2}\cC^{5/2-\eps-\lambda}(\T)}\big\|_{L^p(\Omega)}&\leq C,\label{eq:aprioritildevM}
    \end{align}
\end{lemma}
\begin{proof}
The wellposedness and bound \eqref{eq:aprioritildevM} follows from \cref{prop:wp} by changing the nonlinearity $f$ to $g_h$. Due to \cref{lem:g_h-bounds}, $g_h$ respects similar one-sided Lipschitz and growth bounds as $f$ with constants that do not depend on $M=h^{-1}$. 
\end{proof}

\begin{lemma}\label{lem:apriori-WMN}
    There exist unique mild (and weak) solutions $w^N$ and $w^{M,N}$ to the equations \eqref{eq:bfvN} and \eqref{eq:step2}, respectively.
    Furthermore . For any $p, q\geq 1$ and $\eps>0$ there exists a constant $C=C(p,q,T,\eps,K,m,\mathscr{M})$, that does not depend on $M,N$, such that the following bound holds:
    \begin{align}
        &
\big\|\sup_{t\in[0,T]}\|w^{N}_{t}\|_{L^{q}(\Pi_N)}\big\|_{L^p(\Omega)}\leq C,
\\
        &
\big\|\sup_{t\in[0,T]}\|w^{M,N}_{t}\|_{L^{q}(\Pi_N)}\big\|_{L^p(\Omega)}\leq C.
    \end{align}
\end{lemma}
\begin{proof}
We only prove the result for $w^{N}$, since the one for $w^{M,N}$ is analogue replacing the nonlinearity $f$ by $g_h$, since $g_h$ and $f$ satisfy similar growth and one-sided Lipschitz bounds by \cref{lem:g_h-bounds}.

The well-posedness follows from the fact that \eqref{eq:bfvN} is in fact a system of ODEs with a continuous and one-sided Lipschitz nonlinearity $g_h$, for which the well-posedness is classical.
Next, we prove the energy bounds on $w^{N}$.
    Recall $\tilde{m}$ introduced in Lemma~\ref{lem:g_h-bounds}. Let $q$ be an even integer greater than $2(2\tilde m+1)$.
    Recall that $\Delta_N=\partial^{-}\partial^{+}$ with the discrete derivatives defined in \eqref{eq:discrete-derivatives}. Since $-\partial^+$ is the adjoint of $\partial^-$ with respect to the scalar product $\scal{\cdot,\cdot}_{L^2(\Pi_N)}$,
    by the chain rule we get
    \begin{equs}
        \partial_t\|w^{N}_t\|_{L^q(\Pi_N)}^q&=-q\scal{\partial^+\big((w^{N}_t)^{q-1})\big),\partial^+ w^{N}_t}_{L^2(\Pi_N)}
        \\
        &\qquad+q\scal{(w^{N}_t)^{q-1},g_h(w^{N}_t+\Theta_N O_t)}_{L^2(\Pi_N)}.\label{eq:semidiscrete-energy}
    \end{equs}
    Note that for $w:\Pi_N\to\R$ one has
    \begin{equ}
        \scal{\partial^+(w^{q-1}),\partial^+ w}_{L^2(\Pi_N)}=\frac{1}{N}\sum_{n=0}^{N-1} |\partial^+w(x_n)|^2\int_0^1(q-1)\big(\lambda w(x_{n+1})^{q-2}+(1-\lambda)w(x_{n})^{q-2}\big) \,d\lambda\geq 0.
    \end{equ}
    Furthermore, for $x,y\in\R$ one has due to the one-sided Lipschitz bound of $g_h$,
    \begin{equs}
        x^{q-1}g_h(x+y)&=x^{q-2}(x+y-y)(g_h(x+y)-g_h(y))+x^{q-1}g_h(y)
        \\
        &\leq K x^{q}+x^{q-1}g_h(y) \leq (K+1)x^q+g_h(y)^q.
    \end{equs}
    Using these two estimates in \eqref{eq:semidiscrete-energy}  and the growth of $g_h$ from \cref{lem:g_h-bounds}, we get
    \begin{equ}
        \partial_t\|w^{N}_t\|_{L^q(\Pi_N)}^q\lesssim\|w^{N}_t\|_{L^q(\Pi_N)}^q +\big(1+\|\Theta_N O_t\|_{L^{\infty}(\Pi_N)}^{(2\tilde m+1)q}\big)
    \end{equ}
and thus by Gronwall's lemma and taking expectations we obtain.
\begin{equ}
    \E\sup_{t\in[0,T]}\|w^{N}_t\|_{L^q(\Pi_N)}^q\lesssim 1.
\end{equ}
This proves the claim for $p=q\in2\mathbb{N}$ for $q\geq 2(2\tilde{m}+1)$, which is of course sufficient.
\end{proof}

\begin{corollary}\label{cor:wMN1}
     Let $p \geq 1$, $\eps\in (0, 1/2)$ and $\lambda\in (0,2]$. Then there exists a constant $C=C(p,\eps,\lambda,K,T,m,\mathscr{M})$, that does not depend on $M,N$, such that the following bounds hold:
    \begin{align}
       \Big\|\|w^{M,N}\|_{C_T^{\lambda/2}\calC^{\frac{5}{2}-\eps-\lambda}(\Pi_N)}\Big\|_{L^p(\Omega)}&\leq C;\label{eq:apriori-WMN}\\
        \Big\|\|z^{M,N}\|_{C_T^{\lambda/2}H^{\frac{5}{2}-\eps-\lambda}(\T)}\Big\|_{L^p(\Omega)}&\leq C;\label{eq:apriori-zMN}\\
         \Big\|\|z^{M,N}\|_{C_T\calC^{2-\eps}(\T)}\Big\|_{L^p(\Omega)}&\leq C;\label{eq:apriori-zMN-emb}
         \\\Big\|\|\hat{\mathbf{z}}^{M,N}\|_{C_T^{\lambda/2}H^{\frac{5}{2}-\eps-\lambda}(\T)}\Big\|_{L^p(\Omega)}&\leq C.\label{eq:apriorihatz}
    \end{align}
\end{corollary}

\begin{proof}   
Given \eqref{eq:apriori-WMN}, \eqref{eq:apriori-zMN} follows promptly as
\begin{align*}
    \|z^{M,N}\|_{H^{\frac{5}{2}-\eps-\lambda}(\T)}=\|\Psi_N w^{M,N}\|_{H^{\frac{5}{2}-\eps-\lambda}(\T)}=\| w^{M,N}\|_{H^{\frac{5}{2}-\eps-\lambda}(\Pi_N)}\lesssim \| w^{M,N}\|_{\cC^{\frac{5}{2}-\eps/2-\lambda}(\Pi_N)},
\end{align*}
where we used 
\cref{lem:Theta-Psi} c) and
\cref{lem:spaces-equivalence} e). 
Next, \eqref{eq:apriori-zMN-emb} follows from \eqref{eq:apriori-zMN} by \cref{lem:spaces-equivalence} d).
We now prove \eqref{eq:apriori-WMN}, which is done by upgrading the regularity in \cref{lem:apriori-WMN} in a standard way.
First for any $\eps\in(0,2)$ and $q\geq 1$, we bound the $B^{2-\epsilon}_{q,\infty}$-norm of $w^{M,N}$ by using the semigroup estimates for $P^{N}$ from \cref{lem:semigroup1}
\begin{equs}
\sup_{t\in[0,T]}\|w^{M,N}_t\|_{B^{2-\eps}_{q,\infty}(\Pi_N)}&\lesssim\norm{P_{t}^{N}\Theta_N u_{0}}_{B^{2-\eps}_{q,\infty}(\Pi_N)}+\int_{0}^{t}\norm{P_{t-s}^{N}g_h(w^{M,N}_s+\Theta_N O_s)}_{B^{2-\eps}_{q,\infty}(\Pi_N)}\, ds 
    \\&\lesssim \norm{\Theta_N u_{0}}_{B^{2-\eps}_{q,\infty}(\Pi_N)}+\int_{0}^{t}(t-s)^{-1+\frac{\eps}{2}}\norm{g_h(w^{M,N}_s+\Theta_N O_s)}_{L^{q}(\Pi_N)}\, ds 
    \\&\lesssim \|\Theta_N u_0\|_{\calC^{2-\eps}(\Pi_N)}+\sup_{t\in[0,T]}\|g_h(w^{M,N}_t+\Theta_NO_t)\|_{L^q (\Pi_N)}
 \\   
&\lesssim \| u_0\|_{\calC^{2-\epsilon}(\T)}+1+\sup_{t\in[0,T]}\|w^{M,N}_t\|_{L^{q(2\tilde m+1)}(\Pi_N)}^{2\tilde{m}+1}.\label{eq:w-intermediate}
\end{equs}
Let $\epsilon\in (0, 1/2)$ be arbitrary. Let $q \in [1,\infty)$. Then, using the embedding $ B^{\frac{1-\epsilon}{2}+\frac{1}{q}}_{q,\infty}\hookrightarrow B^{\frac{1-\epsilon}{2}}_{\infty,\infty}=\calC^{\frac{1-\epsilon}{2}}$, the composition estimate Lemma~\ref{lem:composition} as well as the semigroup estimates for $(P^N_t)$, we deduce
\begin{align*}
    \norm{w^{M,N}_t}_{\cC^{\frac{5}{2}-\epsilon}(\Pi_N)}
    &\lesssim\norm{P_{t}^{N}\Theta_N u_{0}}_{\cC^{\frac{5}{2}-\epsilon}(\Pi_N)}+\int_{0}^{t}\norm{P_{t-s}^{N}g_h(w^{M,N}_s+\Theta_N O_s)}_{\cC^{\frac{5}{2}-\epsilon}(\Pi_N)}\, ds 
    \\&\lesssim \norm{u_{0}}_{\cC^{\frac{5}{2}-\epsilon}(\T)}+\int_{0}^{t}(t-s)^{-1+\frac{\eps}{4}}\norm{g_h(w^{M,N}_s+\Theta_N O_s)}_{\cC^{\frac{1-\epsilon}{2}}(\Pi_N)}\, ds 
    \\&\lesssim \norm{u_{0}}_{\cC^{\frac{5}{2}-\epsilon}(\T)}+1+\sup_{t\in[0,T]}\norm{w^{M,N}_t+\Theta_N O_t}_{ \cC^{\frac{1-\epsilon}{2}}(\Pi_N)}^{2\tilde{m}+1}\\
    &\lesssim \norm{u_{0}}_{\cC^{\frac{5}{2}-\epsilon}(\T)}+1+\sup_{t\in[0,T]}\norm{w^{M,N}_t}_{ B^{\frac{1-\epsilon}{2}+\frac{1}{q}}_{q,\infty}(\Pi_N)}^{2\tilde{m}+1}+ \norm{\Theta_N O}_{ C_T\cC^{\frac{1-\epsilon}{2}}(\Pi_N)}^{2\tilde{m}+1}.
\end{align*}
Due to \cref{lem:apriori-WMN} and \cref{lem:OUreg} the $L^p(\Omega)$- norm of the right-hand side is bounded.
Hence, taking the $p$-th moment, by plugging in the estimate for \eqref{eq:w-intermediate}, which is possible since the regularities satisfy $\frac{1-\epsilon}{2}+\frac{1}{q}<2$, we see that
\begin{align}\label{eq:use}
\norm{\sup_{t\in[0,T]}\norm{w^{M,N}_t}_{\cC^{\frac{5}{2}-\epsilon}(\Pi_N)}}_{L^p(\Omega)}\lesssim 1.
\end{align}
To see the time regularity, we write
    \begin{align*}
        w^{M,N}_t-w^{M,N}_s &= (P_{t}^N -P_s^N)\Theta_N u_0 + \int_{s}^{t} P_{t-r}^{N} g_h(w^{M,N}_r+\Theta_N O_r)\, dr \\&\qquad + \int_{0}^{s}P_{s-r}^N(P_{t-s}^N-\operatorname{Id})g_h(w^{M,N}_r+\Theta_N O_r) \, dr.
    \end{align*}
Now we use the semigroup estimates from \cref{lem:semigroup1} and the composition estimate \cref{lem:composition} with the spatial regularity bound, from above, to obtain, for $\lambda\in (0,2]$,
\begin{align*}
\MoveEqLeft
    \norm{w^{M,N}_t-w^{M,N}_s}_{\cC^{\frac{5}{2}-\epsilon-\lambda}(\Pi_N)}\\&\lesssim \norm{(P_{t}^N-P_s^N)\Theta_N u_0}_{\cC^{\frac{5}{2}-\epsilon-\lambda}(\Pi_N)} + \Big\|\int_{s}^{t} P_{t-r}^{N} g_h(w^{M,N}_r+\Theta_N O_r)\, dr\Big\|_{\cC^{\frac{5}{2}-\epsilon-\lambda}(\Pi_N)}
    \\&\qquad + \Big\|\int_{0}^{s}P_{s-r}^N(P_{t-s}^N-\operatorname{Id})g_h(w^{M,N}_r+\Theta_N O_r) \, dr\Big\|_{\cC^{\frac{5}{2}-\epsilon-\lambda}(\Pi_N)}
    \\&\lesssim (t-s)^{\lambda/2}\norm{u_0}_{\cC^{\frac{5}{2}-\epsilon}(\T)} + \paren[\big]{1+\sup_{t\in[0,T]}\norm{w^{M,N}_t+\Theta_N O_t}_{\cC^{\frac{1}{2}-\epsilon}(\Pi_N)}^{2\tilde{m}+2}} \int_{s}^{t}(t-r)^{-1+\lambda/2} \, dr \\&\qquad + \paren[\big]{1+\sup_{t\in[0,T]}\norm{w^{M,N}_t+\Theta_N O_t}_{\cC^{\frac{1}{2}-\epsilon/2}(\Pi_N)}^{2\tilde{m}+2}} \int_{0}^{s} (s-r)^{-1+\varepsilon/4}(t-s)^{\lambda/2}\, dr.
\end{align*}
Applying the supremum in $0\leq s<t\leq T$ and the $p$-th moment we thus find, using \cref{lem:OUreg} and \eqref{eq:use},
\begin{align*}
    (\E\sup_{0\leq s<t\leq T}(t-s)^{-\lambda/2}\norm{w^{M,N}_t-w^{M,N}_s}_{\cC^{\frac{5}{2}-\epsilon-\lambda}(\Pi_N)}^p)^{1/p}\lesssim 1,
\end{align*}
which implies \eqref{eq:apriori-WMN}.
It is left to prove \eqref{eq:apriorihatz}, which follows along the same lines of proof as for \eqref{eq:apriori-WMN}. Indeed, we replace in the above $\Theta_N u_0$ by $\Psi_N\Theta_N u_0$, $P^{N}$ by the standard heat semigroup $P$ and $w^{M,N}+\Theta_N O$ inside the nonlinearity $g_h$ by $z^{M,N}+O$. Further we note that \eqref{eq:PsiThetabounded} implies that $\norm{\Psi_N\Theta_N u_0}_{H^{5/2-\eps}(\T)}\lesssim \norm{u_0}_{H^{5/2-\eps}(\T)}\lesssim  \norm{u_0}_{\calC^{5/2-\eps/2}(\T)} $. Then the apriori estimates for $z^{M,N}$ from \eqref{eq:apriori-zMN-emb}, together with \cref{lem:OUreg}, the latter estimate for the initial condition and the semigroup estimates for $P$ from \cref{lem:semigroup1} yield the claim.
\end{proof}

Next, we prove a priori estimates for the scheme $\tilde{X}^{M,N}$ and the auxiliary scheme $X^{M,N}$.
The proof of the following is similar to \cite[Proposition 5.3 and Corollary 5.4]{DjGK}. 

\begin{proposition}\label{prop:apriori-XMN}
Let $\lambda \in [0,2)$, $p\geqslant 1$, $\varepsilon\in (0,2-\lambda)$. Then, there exists a constant $C=C(p,K,m,\varepsilon,\lambda,T)$, that does not depend on $M,N$, such that
    \begin{align}
         &\E\|X^{M,N}\|^p_{C^{\frac{\lambda}{2}}_T\cC^{\frac{5}{2}-\lambda-\eps}(\Pi_N)}\leqslant C,
        \\ &\E\|\tilde{X}^{M,N}\|^p_{C^{\frac{\lambda}{2}}_T\cC^{\frac{5}{2}-\lambda-\eps}(\Pi_N)}\leqslant C.
    \end{align}
\end{proposition}
\begin{proof}
We only prove the result for $\tilde{X}^{M,N}$, as the one for $X^{M,N}$ is analogue replacing $\tilde{O}$ by $O$.
We first prove that, for $t_k=\frac{kT}{M},$
\begin{align} \label{eq:Xapriori}
   \E \sup_{k=0,\dots,M} \|\tilde{X}_{t_k}^{M,N}\|^p_{L^\infty(\Pi_N)}\lesssim 1.
\end{align}
Using the formulation in \eqref{eq:XY}, the Lipschitz property \eqref{eq:PhiLip} of $\Phi_t$ and that $\Phi_h(x)-x=hg_h(x)$, we have by definition (see \eqref{eq:gh}) that for $0\leqslant k\leqslant M-1$,
\begin{align*}
\|\tilde{X}_{t_{k+1}}^{M,N}\|_{L^\infty(\Pi_N)}&=\|P_h^N \tilde{Y}_{t_{k+1}}^{M,N}\|_{L^\infty(\Pi_N)}\\&=\|P_h^N[\Phi_h(\tilde{X}_{t_k}^{M,N}+\Theta_N \tilde{O}_{t_k})-\Theta_N \tilde{O}_{t_k}]\|_{L^\infty(\Pi_N)}\\
    &\leqslant \|P_h^N[\Phi_h(\tilde{X}_{t_k}^{M,N}+\Theta_N \tilde{O}_{t_k})-\Phi_h(\Theta_N \tilde{O}_{t_k})]+P_h^N[\Phi_h(\Theta_N \tilde{O}_{t_{k}})-\Theta_N \tilde{O}_{t_{k}}]\|_{L^\infty(\Pi_N)}\\
    &\leqslant e^{Kh/2}\|\tilde{X}_{t_k}^{M,N}\|_{L^\infty(\Pi_N)}+h\tilde{K}(1+\|\Theta_N \tilde{O}\|_{C_T L^\infty(\Pi_N)}^{2\tilde{m}+1}),
\end{align*}
where we used the contraction property of $P^N$ in the $L^\infty(\Pi_N)$ norm, see  \cref{lem:eigenvalues} d).
Iterating the above, using that $\|\Theta_N f\|_{L^\infty(\Pi_N)}\leqslant \|f\|_{L^\infty(\T)}$ for a continuous function $f$ by Lemma~\ref{lem:Theta-Psi} f) and using that $(k+1)h\leqslant T$ twice, we arrive at
\begin{align*}   \|\tilde{X}_{t_{k+1}}^{M,N}\|_{L^\infty(\Pi_N)}&\leqslant e^{(k+1)Kh/2}\|\Theta_N u_0\|_{L^\infty(\Pi_N)}+h\tilde{K}\sum_{j=0}^k e^{jKh/2}(1+\|\tilde{O}\|_{C_T L^\infty(\T)}^{2\tilde{m}+1})\\
&\leqslant e^{TK/2}\Big(\|u_0\|_{L^\infty(\T)}+T\tilde{K} (1+\|\tilde{O}\|_{C_T L^\infty(\T)}^{2\tilde{m}+1})\Big),
\end{align*}
which implies \eqref{eq:Xapriori}.
Using the semigroup estimates from Lemma~\ref{lem:semigroup1} and the composition estimate \cref{lem:composition}, we obtain for any $t\in[0,T]$, using that $k_{M}(r)\leq r$,
\begin{align*}
    \norm{\tilde{X}^{M,N}_t}_{\calC^{\frac{1}{2}-\eps}(\Pi_N)}&\leq \norm{u_0}_{\calC^{1/2-\eps}(\T)} + \int_{0}^{t}\norm{P_{t-k_{M}(r)}^Ng_{h}(\tilde{X}^{M,N}_{k_M(r)}+\Theta_N \tilde{O}_{k_M(r)})}_{\calC^{1/2-\eps}(\Pi_N)}\, dr
    \\&\leq  \norm{u_0}_{\calC^{1/2+\eps}(\T)} + \int_{0}^{t}(t-r)^{-1/4+\eps/2}\norm{g_{h}(\tilde{X}^{M,N}_{k_M(r)}+\Theta_N \tilde{O}_{k_M(r)})}_{L^{\infty}(\Pi_N)}\, dr
    \\&\leq \norm{u_0}_{\calC^{1/2+\eps}(\T)} + T^{3/4+\eps/2}\Big(1+\sup_k\|\tilde{X}_{t_k}^{M,N}\|_{L^\infty(\Pi_N)}^{2\tilde{m}+1}+\|\tilde{O}\|^{2\tilde{m}+1}_{C_TL^\infty(\T)}\Big),
\end{align*}
where the right-hand side is bounded by \eqref{eq:Xapriori} and \cref{ass:tildeO}. Furthermore, using again \cref{lem:semigroup1} and \cref{lem:composition}, we find for $\lambda\in[0,2)$ and $\eps\in (0,2-\lambda)$,
\begin{align*}
\MoveEqLeft
    \|\tilde{X}_t^{M,N}-\tilde{X}_s^{M,N}\|_{\cC^{\frac{5}{2}-\lambda-2\eps}(\Pi_N)}
    \\&\leqslant \int_s^t \|P_{t-k_M(r)}^N g_h(\tilde{X}^{M,N}_{k_M(r)}+\Theta_N \tilde{O}_{k_M(r)})\|_{\cC^{\frac{5}{2}-\lambda-2\eps}(\Pi_N)} \, dr \\
    &\qquad \qquad \quad \quad  + \int_0^s \|P_{s-k_M(r)}^N(P_{t-s}^N-\Id)g_h(\tilde{X}^{M,N}_{k_M(r)}+\Theta_N \tilde{O}_{k_M(r)})\|_{\cC^{\frac{5}{2}-\lambda-2\eps}(\Pi_N)}\, dr\\
    &\lesssim \int_s^t (t-r)^{-1+\lambda/2+\eps/2} \|g_h(\tilde{X}^{M,N}_{k_M(r)}+\Theta_N \tilde{O}_{k_M(r)})\|_{\calC^{\frac{1}{2}-\eps}(\Pi_N)}\, dr\\
    &\qquad \qquad +\int_0^s (s-r)^{-1+\eps/2}\|(P_{t-s}^N-\Id)g_h(\tilde{X}^{M,N}_{k_M(r)}+\Theta_N \tilde{O}_{k_M(r)})\|_{\cC^{\frac{1}{2}-\lambda-\eps}(\Pi_N)}\, dr\\
    &\lesssim \squeeze[1]{\Big(1+\sup_k\|\tilde{X}_{t_k}^{M,N}\|_{\calC^{\frac{1}{2}-\eps}(\Pi_N)}^{2\tilde{m}+1}+\|\tilde{O}\|^{2\tilde{m}+1}_{C_T\calC^{\frac{1}{2}-\eps}(\T)}\Big)\Big(\int_s^t (t-r)^{-1+\lambda/2+\eps/2}\, dr+\int_0^s (s-r)^{-1+\frac{\varepsilon}{2}}(t-s)^{\frac{\lambda}{2}}\, dr\Big)}\\
    &\lesssim \Big(1+\sup_k\|\tilde{X}_{t_k}^{M,N}\|_{\calC^{\frac{1}{2}-\eps}(\Pi_N)}^{2\tilde{m}+1}+\|\tilde{O}\|^{2\tilde{m}+1}_{C_T\calC^{\frac{1}{2}-\eps}(\T)}\Big)(t-s)^{\lambda/2}
\end{align*}
and the result follows by an application of Kolmogorov's continuity theorem and the a priori estimate \eqref{eq:Xapriori}, as well as the regularity bounds for $\tilde{O}$ from \cref{ass:tildeO} b).
\end{proof}

Next we bound the terms in the last line of the error decomposition \eqref{BIG-decomposition}. 
\begin{lemma}\label{lem:XtildeX}
    Let $p\geq 1$ and $\eps>0$. Then there exists $C=C(T,\eps, p,K, m)$, that does not depend on $M,N\in\N$, such that
\begin{align}
\MoveEqLeft
    (\E\sup_{t\in[0,T]}\norm{X^{M,N}_t - \tilde{X}^{M,N}_t}_{L^{2}(\Pi_N)}^p)^{1/p}\nonumber
    \leq C (M^{-1+\eps}+N^{-\frac{3}{2}+\eps}).
\end{align}
\end{lemma}
\begin{proof}
We have that, using the shorthand notations $\hat{y}_t^{M,N}=X^{M,N}_{t}-\tilde{X}^{M,N}_t$, and $X_{t}=X^{M,N}_t$, $\tilde{X}_t=\tilde{X}^{M,N}_t$,
\begin{align*}
    \hat{y}_{t}^{M,N} = \int_{0}^{t}P^{N}_{t-k_{M}(s)} [g_{h}(X_{k_{M}(s)}+\Theta_NO_{k_{M}(s)})-g_{h}(\tilde{X}_{k_{M}(s)}+\Theta_N\tilde{O}_{k_{M}(s)})] \, ds.
\end{align*}
Define an auxiliary process $y_t^{M,N}$ by
\begin{align*}
    y_{t}^{M,N} = \int_{0}^{t}P^{N}_{t-s} [g_{h}(X_{k_{M}(s)}+\Theta_NO_{k_{M}(s)})-g_{h}(\tilde{X}_{k_{M}(s)}+\Theta_N\tilde{O}_{k_{M}(s)})] \, ds,
\end{align*}
where we replaced $P_{t-k_{M}(s)}^N$ in the definition of $\hat{y}^{M,N}$ by $P_{t-s}^N$.
First we prove that $y^{M,N}$ and $\hat{y}^{M,N}$ are sufficiently close. This follows by the semigroup estimates, \cref{lem:semigroup1}, with $u^{M,N}_s=g_{h}(X_{k_{M}(s)}+\Theta_NO_{k_{M}(s)})-g_{h}(\tilde{X}_{k_{M}(s)}+\Theta_N\tilde{O}_{k_{M}(s)})$ and the following observation,
\begin{align*}
    \int_0^{t}\norm{(P_{t-s}^{N}-P_{t-k_{M}(s)}^{N}) u^{M,N}_s}_{L^2(\Pi_N)}\, ds
    &\leq \int_{0}^{t}\norm{(P_{t-s}^{N}(\operatorname{Id}-P_{s-k_{M}(s)}^{N}) u^{M,N}_s}_{L^{\infty}(\Pi_N)}\, ds
    \\&\leq \sup_{s\in[0,T]}\norm{u^{M,N}_s}_{\calC^{1/2-2\eps}(\Pi_N)}\int_{0}^{t}(t-s)^{-3/4} (s-k_{M}(s))^{1-\eps}\, ds 
    \\&\leq M^{-1+\eps} T^{1/4} \sup_{s\in[0,T]}\norm{u^{M,N}_s}_{\calC^{1/2-2\eps}(\Pi_N)}.
\end{align*}
The $p$-th moment of $\sup_{s\in[0,T]}\norm{u^{M,N}_s}_{\calC^{1/2-2\eps}(\Pi_N)}$ is bounded by a constant that does not depend on $M,N$ by the composition estimate, \cref{lem:composition}, and the apriori estimates from \cref{prop:apriori-XMN}, as well as \cref{ass:tildeO}. Thus we obtain that
\begin{align}\label{eq:close}
    (\E\norm{y^{M,N}-\hat{y}^{M,N}}_{C_TL^2(\Pi_N)}^p)^{1/p}\leq C M^{-1+\eps}.
\end{align}
In particular, we find
\begin{align*}
    \E\norm{\hat{y}^{M,N}}_{C_TL^2}^p\lesssim \E\norm{y^{M,N}}_{C_TL^2}^p + M^{(-1+\epsilon)p}.
\end{align*}
Next, we observe that for $0\leq s\leq t\leq T$, due to the semigroup estimates of $P^{N}$ from \cref{lem:semigroup1} and the growth bound for $g_h$, for any $\epsilon\in (0,1)$ with the notation from above
\begin{align*}
\MoveEqLeft
    \norm{y^{M,N}_{t}-y^{M,N}_{s}}_{L^{\infty}(\Pi_N)}\nonumber
    \\&\lesssim \int_{s}^{t}\norm{P^{N}_{t-r} u^{M,N}_{r}}_{L^{\infty}(\Pi_N)}\,dr + \int_{0}^{s} \norm{P_{s-r}^N(P_{t-s}^N-\operatorname{Id}) u^{M,N}_{r}}_{L^{\infty}(\Pi_N)}\, dr \nonumber
    \\&\lesssim [(t-s) + (t-s)^{1-\epsilon} T^{\epsilon}] \sup_{r\in[0,T]}\norm{u^{M,N}_{r}}_{L^{\infty}(\Pi_N)}\nonumber
    \\&\lesssim (t-s)^{1-\epsilon} (1+\norm{X^{M,N}}_{C_TL^{\infty}(\Pi_N)}^{2\tilde{m}+1}+\norm{\tilde{X}^{M,N}}_{C_TL^{\infty}(\Pi_N)}^{2\tilde{m}+1}+\norm{O}_{_{C_TL^{\infty}(\Pi_N)}}^{2\tilde{m}+1}+\norm{\tilde{O}}_{_{C_TL^{\infty}(\Pi_N)}}^{2\tilde{m}+1}),
\end{align*}
where the $q$-th moment of the right-hand side is bounded due to \cref{prop:apriori-XMN} and \cref{lem:OUreg} and \cref{ass:tildeO}, for any $q\geq 1$. This implies that for any $q\geq 1$,
\begin{align}\label{eq:y-time-reg}
 \norm{\norm{y^{M,N}}_{C_T^{1-\epsilon}L^{\infty}(\Pi_N)}}_{L^{q}(\Omega)}\lesssim \tilde{C},
\end{align}
where $\tilde{C}$ does not depend on $M,N$.
Furthermore, we have that $y^{M,N}$ weakly solves (cf. \cite{bell})
\begin{align*}
    \partial_{t} y_t^{M,N}=\Delta_N y_t^{M,N}+g_{h}(X_{k_{M}(t)}+\Theta_NO_{k_{M}(t)})-g_{h}(\tilde{X}_{k_{M}(t)}+\Theta_N\tilde{O}_{k_{M}(t)}), \quad y_0^{M,N}=0.
\end{align*}
Thus we have that by testing against $y^{M,N}$, where we denote $\langle u, v\rangle=\langle u, v\rangle_{L^2(\Pi_N)}$ below and where we add and subtract appropriate terms, 
\begin{align*}
\MoveEqLeft
    \norm{y_t^{M,N}}_{L^{2}(\Pi_N)}^2 
    \\&= 2\int_{0}^{t}\langle \Delta_N y^{M,N}_s, y^{M,N}_s\rangle \, ds 
    \\&\quad+ 2\int_{0}^{t}\langle g_{h}(X_{k_{M}(t)}+\Theta_NO_{k_{M}(t)})-g_{h}(\tilde{X}_{k_{M}(t)}+\Theta_N\tilde{O}_{k_{M}(t)}), y^{M,N}_{s}\rangle \, ds
    \\&\leqslant 2\int_{0}^{t}\langle g_{h}(X_{k_{M}(s)}+\Theta_NO_{k_{M}(s)})-g_{h}(\tilde{X}_{k_{M}(s)}+\Theta_N\tilde{O}_{k_{M}(s)}), y^{M,N}_{s}\rangle \, ds
    \\& = 2\int_{0}^{t}\langle g_{h}(X_{k_{M}(s)}+\Theta_NO_{k_{M}(s)})-g_{h}(\tilde{X}_{k_{M}(s)}+\Theta_N\tilde{O}_{k_{M}(s)}), \hat{y}^{M,N}_{k_{M}(s)} + \Theta_NO_{k_M(s)}-\Theta_N\tilde{O}_{k_{M}(s)}\rangle \, ds
    \\&\quad -2\int_{0}^{t}\langle g_{h}(X_{k_{M}(s)}+\Theta_NO_{k_{M}(s)})-g_{h}(\tilde{X}_{k_{M}(s)}+\Theta_N\tilde{O}_{k_{M}(s)}), \Theta_NO_{k_M(s)}-\Theta_N\tilde{O}_{k_{M}(s)}\rangle \, ds
    \\&\quad -2\int_{0}^{t}\langle g_{h}(X_{k_{M}(s)}+\Theta_NO_{k_{M}(s)})-g_{h}(\tilde{X}_{k_{M}(s)}+\Theta_N\tilde{O}_{k_{M}(s)}), y^{M,N}_{k_{M}(s)} -y^{M,N}_{s}\rangle \, ds
    \\&\quad +2\int_{0}^{t}\langle g_{h}(X_{k_{M}(s)}+\Theta_NO_{k_{M}(s)})-g_{h}(\tilde{X}_{k_{M}(s)}+\Theta_N\tilde{O}_{k_{M}(s)}), y^{M,N}_{k_{M}(s)} -\hat{y}^{M,N}_{k_{M}(s)}\rangle \, ds.
\end{align*}
Using the one-sided Lipschitzness of $g_h$ with constant $K$, as well as triangle inequality, Cauchy Schwarz and the local Lipschitz bound of $g_h$ with $\tilde{m}$ from \cref{lem:g_h-bounds} we find
\begin{align*}
\MoveEqLeft
    \norm{y_t^{M,N}}_{L^{2}(\Pi_N)}^2
    \\&\leq 2K\int_{0}^{t} \norm{\hat{y}^{M,N}_{k_{M}(s)}}^{2}_{L^2(\Pi_N)} \, ds + 2K\int_{0}^{T}\norm{\Theta_NO_{k_{M}(s)}-\Theta_N\tilde{O}_{k_{M}(s)}}_{L^{2}(\Pi_N)}^2 \, ds
    \\&\quad + 2K^{M,N}\int_{0}^{t} \norm{\hat{y}^{M,N}_{k_{M}(s)}}_{L^2(\Pi_N)}\norm{\Theta_NO_{k_{M}(s)}-\Theta_N\tilde{O}_{k_{M}(s)}}_{L^{2}(\Pi_N)}  \, ds 
     \\&\quad + 2K^{M,N} \int_{0}^{T} \norm{\Theta_NO_{k_{M}(s)}-\Theta_N\tilde{O}_{k_{M}(s)}}_{L^{2}(\Pi_N)}^2 \, ds
    \\&\quad + 2K^{M,N} M^{-1+\eps}\norm{y^{M,N}}_{C_T^{1-\eps}L^{2}(\Pi_N)}\int_{0}^{t}\norm{\hat{y}^{M,N}_{k_{M}(s)}}_{L^2(\Pi_N)}\, ds
    \\&\quad + 2K^{M,N} M^{-1+\eps}\norm{y^{M,N}}_{C_T^{1-\eps}L^{2}(\Pi_N)}\int_{0}^{T} \norm{\Theta_NO_{k_{M}(s)}-\Theta_N\tilde{O}_{k_{M}(s)}}_{L^{2}(\Pi_N)} \, ds
    \\&\quad + 2K^{M,N}\norm{y^{M,N}-\hat{y}^{M,N}}_{C_TL^2(\Pi_N)} \int_{0}^{t}\norm{\hat{y}^{M,N}_{k_{M}(s)}}_{L^2(\Pi_N)}\, ds
    \\&\quad + 2K^{M,N}\norm{y^{M,N}-\hat{y}^{M,N}}_{C_TL^2(\Pi_N)} \int_{0}^{T} \norm{\Theta_NO_{k_{M}(s)}-\Theta_N\tilde{O}_{k_{M}(s)}}_{L^{2}(\Pi_N)} \, ds
\end{align*}
and thus with Young's inequality,
\begin{align*}
\MoveEqLeft
    \norm{y_t^{M,N}}_{L^{2}(\Pi_N)}^2
    \\&\leqslant 2(K+\frac{1}{2}+\frac{1}{2}+\frac{1}{2})\int_{0}^{t} \norm{\hat{y}^{M,N}_{k_{M}(s)}}^{2}_{L^2(\Pi_N)} \, ds
    \\&\quad + 2(K+K^{M,N}+\frac{1}{2}(K^{M,N})^2) \int_{0}^{T} \norm{\Theta_NO_{k_{M}(s)}-\Theta_N\tilde{O}_{k_{M}(s)}}_{L^{2}(\Pi_N)}^2 \, ds
    \\&\quad +  (K^{M,N})^2 M^{-2+2\eps} \norm{y^{M,N}}_{C_T^{1-\eps}L^{2}(\Pi_N)}^2 
    \\&\quad +  2K^{M,N} M^{-1+\eps}\norm{y^{M,N}}_{C_T^{1-\eps}L^{2}(\Pi_N)}\int_{0}^{T} \norm{\Theta_NO_{k_{M}(s)}-\Theta_N\tilde{O}_{k_{M}(s)}}_{L^{2}(\Pi_N)} \, ds
    \\&\quad + (K^{M,N})^2 \norm{y^{M,N}-\hat{y}^{M,N}}_{C_TL^2(\Pi_N)}^2
    \\&\quad + 2K^{M,N}\norm{y^{M,N}-\hat{y}^{M,N}}_{C_TL^2(\Pi_N)} \int_{0}^{T} \norm{\Theta_NO_{k_{M}(s)}-\Theta_N\tilde{O}_{k_{M}(s)}}_{L^{2}(\Pi_N)} \, ds,
\end{align*}
where we employed the notation $$K^{M,N}:=1+\norm{X^{M,N}}_{C_TL^{\infty}(\Pi_N)}^{2\tilde{m}+1}+\norm{\tilde{X}^{M,N}}_{C_TL^{\infty}(\Pi_N)}^{2\tilde{m}+1}+\norm{O}_{C_TL^{\infty}(\T)}^{2\tilde{m}+1}+\norm{\tilde{O}}_{C_TL^{\infty}(\T)}^{2\tilde{m}+1}.$$
Taking the $p$-th moment and applying Gronwalls inequality to $\E\sup_{t\in[0,r]}\norm{\hat{y}^{M,N}_{t}}_{L^2(\Pi_N)}^{2p}$ for $r\in [0,T]$, thus yields that, since $k_{M}(s)\leq s$, 
\begin{align*}
\MoveEqLeft
    (\E\sup_{t\in[0,r]}\norm{\hat{y}_t^{M,N}}_{L^{2}(\Pi_N)}^{2p})^{1/2p}
    \\&\lesssim (\E\sup_{t\in[0,r]}\norm{y_t^{M,N}}_{L^{2}(\Pi_N)}^{2p})^{1/2p} + M^{-1+\epsilon}
    \\&\lesssim (\E(K^{M,N})^{4p})^{1/4p} 
    \\&\quad\times\paren[\Big]{\max_{t\in I_M}(\E\norm{\Theta_NO_t-\Theta_N\tilde{O}_t}_{L^{2}(\Pi_N)}^{4p})^{1/4p}+ M^{-1+\eps} (\E\norm{y^{M,N}}_{C_T^{1-\eps}L^{2}(\Pi_N)}^{4p})^{1/4p}
    \\&\quad + M^{-\frac{1}{2}+\eps/2} (\E\norm{y^{M,N}}_{C_T^{1-\eps}L^{2}(\Pi_N)}^{4p})^{1/8p}\max_{t\in I_M}(\E\norm{\Theta_N O_t-\Theta_N\tilde{O}_t}_{L^{2}(\Pi_N)}^{4p})^{1/8p}
    \\&\quad +\squeeze[1]{(\E\norm{y^{M,N}-\hat{y}^{M,N}}_{C_TL^2}^{4p})^{1/4p} + (\E\norm{y^{M,N}-\hat{y}^{M,N}}_{C_TL^2}^{4p})^{1/8p}\max_{t\in I_M}(\E\norm{\Theta_NO_t-\Theta_N\tilde{O}_t}_{L^{2}(\Pi_N)}^{4p})^{1/8p}}}
    \\&\quad + M^{-1+\epsilon}
    \\&\lesssim (1+(\E(K^{M,N})^{4p})^{1/4p}) \paren[\Big]{M^{-1+\eps}+N^{-\frac{3}{2}+\eps}+ M^{-1+\eps} (\E\norm{y^{M,N}}_{C_T^{1-\eps}L^{2}(\Pi_N)}^{4p})^{1/4p}
    \\&\qquad + M^{-\frac{1+\eps}{2}} (M^{-\frac{1+\eps}{2}}+N^{-\frac{3}{4}+\frac{\eps}{2}}) (\E\norm{y^{M,N}}_{C_T^{1-\eps}L^{2}(\Pi_N)}^{4p})^{1/8p} +(M^{-\frac{1+\eps}{2}}+N^{-\frac{3}{4}+\frac{\eps}{2}})^2}
    \\&\lesssim (M^{-1+\eps}+ N^{-3/2+\eps}),
\end{align*}
where in each line we estimate by a constant $C=C(T,K,p, \eps, m)$ that does not depend on $M,N$. Above we further used the apriori estimates from \cref{prop:apriori-XMN} which yield a bound on the moments of $K^{M,N}$, that does not depend on $M,N$, together with the bound from \eqref{eq:y-time-reg}, as well as Cauchy Schwarz and the bound on $\Theta_NO-\Theta_N\tilde{O}$ from \cref{ass:tildeO} a) together with the bound \eqref{eq:close}.
\end{proof}

\subsection{Spatial error} \label{subsec:spatial}

The following lemmata each bound one term in the decomposition \eqref{BIG-decomposition}.
\begin{lemma} \label{lem:easyterm}
   Let $\varepsilon\in (0,1/2)$ and $p\geqslant 1$. Then there exists $C=C(T,\epsilon, p, K,m, \mathcal{M})$, that does not depend on $N$, such that
   \begin{equ}\label{eq:error2/1}
         \Big\|\sup_{t\in[0,T]}\|(\Psi_N\Theta_N -\Id) v_t\|_{L^2(\T)}\Big\|_{L^p(\Omega)}\leq C N^{-5/2+\eps}.
    \end{equ}
\end{lemma}
\begin{proof}
    This follows from \cref{lem:operator} for $\alpha=5/2-\epsilon$ and the regularity bounds for $v$ from \cref{prop:wp}.
\end{proof}

\begin{lemma} \label{lem:ztildez}
 Let $p\geq 1$. Then there exists a constant $C=C(p,K,T,m,\mathscr{M})$, that does not depend on $N$, such that the following bound holds:
 \begin{align}
    \norm{\|\mathbf{z}^{M,N} -z^{M,N}\|_{C_T L^2(\T)}}_{L^{p}(\Omega)}\lesssim C N^{-2}.
 \end{align}
\end{lemma}
\begin{proof}
   We recall that
   \begin{align*}
       \mathbf{z}^{M,N}_t-z^{M,N}_t= (P_{t}\Psi_N-\Psi_N P_{t}^{N}) (\Theta_N u_0) + \int_{0}^{t} (P_{t-r}\Psi_N-\Psi_N P_{t-r}^{N}) g_{h}(w^{M,N}_r+\Theta_N O_r) \, dr.
   \end{align*}
   Using the commutator \cref{lem:commutator} for $\gamma=2$, $\alpha=0$ and $\beta=2$ and the regularity of $u_0$ from \cref{asn:u0}, we bound
   \begin{align*}
       \sup_{t\in[0,T]}\norm{(P_{t}\Psi_N-\Psi_N P_{t}^{N}) (\Theta_N u_0)}_{L^2(\T)}\lesssim N^{-2}\norm{\Theta_N u_0}_{H^{2}(\Pi_N)}\lesssim N^{-2}\norm{u_0}_{\calC^{\frac{5}{2}-\epsilon}(\T)}.
   \end{align*}
   Furthermore, by \cref{lem:commutator} for $\gamma=2$, $\alpha=0$ and $\beta=1/2-\epsilon$, the composition estimate from \cref{lem:composition} and integration,
   \begin{align*}
   \MoveEqLeft
       \norm[\bigg]{\int_{0}^{t} (P_{t-r}\Psi_N-\Psi_N P_{t-r}^{N}) (g_{h}(w^{M,N}_r+\Theta_N O_r) \, dr}_{L^2(\T)}\\&\lesssim N^{-2}\, t^{\frac{1-2\epsilon}{4}} \sup_{r\in[0,t]}\norm{g_{h}(w^{M,N}_r+\Theta_N O_r)}_{H^{1/2-\epsilon}(\Pi_N)}\\&\lesssim N^{-2}  (1+\norm{w^{M,N}+\Theta_N O}_{C_T L^{\infty}(\Pi_N)}^{2\tilde{m}+1}) (1+ \norm{w^{M,N}+\Theta_N O}_{C_T H^{1/2-\epsilon}(\Pi_N)}).
   \end{align*}
   We conclude by \cref{cor:wMN1} with a Besov embedding into $L^\infty$ and the bounds for $O$ from \cref{lem:OUreg}.
\end{proof}

\begin{lemma} \label{lem:zhatztilde}
Let $\alpha \in (1,3/2)$ and $p\geqslant 1$. Then there exists $C=C(T,p,\alpha,K,m,\mathscr{M})$, that does not depend on $N$, such that 
\begin{align*}
    (\E\|\hat{\mathbf{z}}^{M,N}-\mathbf{z}^{M,N}\|_{C_T L^2(\T)}^p)^{1/p}\leqslant C N^{-\alpha}.
\end{align*}
\end{lemma}
\begin{proof}  
Since throughout the proof there is no risk of confusion, we omit $\T$ in the norms appearing. Throughout the proof, we rely on the properties of $g_h$, collected in Lemma~\ref{lem:g_h-bounds}. We aim to bound, for arbitrary $0\leq s\leq t\leq T$,
\begin{align}\label{eq:claim1}
   \paren[\bigg]{\E\Big\|\int_s^t P_{t-r} ((\operatorname{Id}-\Psi_N \Theta_N) g_h(z^{M,N}_r + O_r))\, ds\Big\|_{L^2}^p}^{1/p}\leq C N^{-\alpha} (t-s)^{1/2},
\end{align}
from which the claim follows by a version of Kolmogorov's continuity theorem from \cref{prop:vKolmogorov}.
We use the stochastic sewing Lemma~\ref{lem:SSLHilbert} to estimate the above. To do so, let for $(u,v) \in [s,t]_\leq^2$,
\begin{align*}
    A_{u,v}&\coloneqq \E_u \int_u^v P_{t-r} (\operatorname{Id}-\Psi_N \Theta_N)g_h(\E_u z^{M,N}_r+O_r)\, dr
\end{align*}
Computing the conditional expectation using that $g_h$ is Nemytskii together with the rule \eqref{eq:rule}, similar as in the proof of \cref{lem:space-ss}, we find that almost surely,
\begin{align*}
    A_{u,v}=\int_u^v P_{t-r} (\operatorname{Id}-\Psi_N \Theta_N)(P^{\mathbb{R}}_{\Qx{u}{r}} g_h)(\E_u z^{M,N}_r+P_{r-u} O_u)\, dr.
\end{align*}
We first verify \eqref{eq:sewingassump2}. 
Again we let $\varepsilon\in (0,\min(\frac{3}{4}-\frac{\alpha}{2},\frac{1}{2}))$ be small and introduce the notation $\mathfrak{a}_{u,r}:=P_{r-u}O_u$ for $u\leq r$ and $\mathfrak{c}_u:=\E_u z_r^{M,N}=\mathfrak{c}$, where the bounds for $\mathfrak{a}$ are given by \eqref{eq:a-bound}. Further, we note that due to the contraction property of the heat semigroup on $L^{\infty}$,
\begin{align}\label{eq:a-bound-unif}
    \norm{\mathfrak{a}_{u,r}}_{L^{\infty}}\lesssim \norm{O}_{C_TL^{\infty}},
\end{align}
where the right-hand side is a uniform bound in $0\leq u\leq r\leq T$.
For $\mathfrak{c}$, we find the following uniform bound in $u\in[0,T]$,
\begin{align}\label{eq:c-bound}
  \norm{\mathfrak{c}_u}_{C^{1}_{b}}\lesssim \norm{z^{M,N}}_{C_T \calC^{\frac{1}{2}-\epsilon}}\lesssim \norm{z^{M,N}}_{C_T \calC^{\frac{3}{2}-\epsilon}},
\end{align}
where $p$-th moments of the right-hand side are bounded by \cref{cor:wMN1}. Below we only write $\mathfrak{c}$, omitting the index $u$, whenever we only employ the uniform bound \eqref{eq:c-bound}.
Notice that the composition estimate from \cref{lem:composition} applies to $P^{\mathbb{R}}_{\Qx{u}{r}} g_h$ with $\theta=\alpha$, $\kappa=2$ due to \cref{lem:composition-applies} and $g_{h}\in C^{2}_{\omega}$ for $\omega$ as therein and $\beta=2\tilde{m}+1$ by \eqref{eq:ghgrowth}. 
Using \cref{lem:operator}, the embedding $\calC^{\alpha+\epsilon}\hookrightarrow H^{\alpha}$ and \cref{lem:composition}, we arrive at
\begin{align*}
    \|A_{u,v}\|_{L^2}&\leqslant \int_{u}^{v}\|(\operatorname{Id}-\Psi_N \Theta_N)(P^{\mathbb{R}}_{Q(r-u)} g_h)(\mathfrak{a}_{u,r}+\mathfrak{c})\|_{L^2}\, dr
    \\&\lesssim N^{-\alpha} \int_{u}^{v} \norm{(P^{\mathbb{R}}_{Q(r-u)} g_h)(\mathfrak{a}_{u,r}+\mathfrak{c})}_{\calC^{\alpha+\epsilon}} \, dr
    \\&\lesssim   \squeeze[1]{N^{-\alpha} \int_{u}^{v} (1+\norm{\mathfrak{a}_{u,r}+\mathfrak{c}}_{L^{\infty}}^{2\tilde{m}+1})\,(1 +\|\mathfrak{a}_{u,r}+\mathfrak{c}\|_{C^{1}_b}\|\mathfrak{a}_{u,r}+\mathfrak{c}\|_{\cC^{\alpha-1+\varepsilon}}+\norm{\mathfrak{a}_{u,r}+\mathfrak{c}}_{\cC^{\alpha+\varepsilon}}) \, dr}.
\end{align*}
Together with the bounds \eqref{eq:a-bound}, \eqref{eq:a-bound-unif} and \eqref{eq:c-bound}, we thus obtain the following estimate on $A_{u,v}$, using furthermore that $\alpha-1+\epsilon\leq \frac{1}{2}-\epsilon$,
\begin{align*}
    \|A_{u,v}\|_{L^2}&\lesssim   N^{-\alpha} \int_{u}^{v} (1+\norm{O}_{C_TL^{\infty}}^{2\tilde{m}+1}+\norm{z^{M,N}}_{C_T L^{\infty}}^{2\tilde{m}+1})\\
    &\qquad\qquad \times\Big(1+\norm{z^{M,N}}_{C_T \calC^{3/2-\epsilon}}(\norm{z^{M,N}}_{C_T \calC^{3/2-\epsilon}} +\|\mathfrak{a}_{u,r}\|_{\cC^{\alpha+\varepsilon-1}}+ \norm{\mathfrak{a}_{u,r}}_{C^{1}_b}))\\
    &\qquad\qquad\quad +\|\mathfrak{a}_{u,r}\|_{C^{1}_b}\|\mathfrak{a}_{u,r}\|_{\cC^{\alpha+\varepsilon-1}}+\norm{\mathfrak{a}_{u,r}}_{\cC^{\alpha+\varepsilon}}\Big) \, dr
    \\&\lesssim   N^{-\alpha} \int_{u}^{v} (1+\norm{O}_{C_TL^{\infty}}^{2\tilde{m}+1}+\norm{z^{M,N}}_{C_T L^{\infty}}^{2\tilde{m}+1})\\
    &\qquad\qquad\times\Big(1+\norm{z^{M,N}}_{C_T \calC^{3/2-\epsilon}}\big(\norm{z^{M,N}}_{C_T \calC^{3/2-\epsilon}}+(1+(r-u)^{-1/4-\varepsilon})\norm{O}_{C_T\calC^{1/2-\epsilon}}\big)
    \\&\qquad\qquad\quad + (r-u)^{-1/4-\varepsilon}\norm{O}_{C_T\calC^{1/2-\epsilon}}^2+(r-u)^{-\alpha/2+1/4-\varepsilon}\norm{O}_{C_T\calC^{1/2-\epsilon}}\Big) \,dr
    \\&\lesssim   N^{-\alpha} \Big(1+\norm{O}_{C_T\calC^{1/2-\epsilon}}^{2\tilde{m}+3}+\norm{z^{M,N}}_{C_T \calC^{3/2-\eps}}^{2\tilde{m}+3}\Big)(v-u)^{5/4-\alpha/2-\eps}.
\end{align*}
Hence, taking the $L^p(\Omega)$-norm, this gives
\begin{align*}
\big\|\|A_{u,v}\|_{L^2}\big\|_{L^p(\Omega)}&\lesssim N^{-\alpha} (v-u)^{5/4-\alpha/2-\eps}\\
&\qquad\times \Big(1+\big\|\norm{z^{M,N}}_{\mathcal{C}_T \cC^{3/2-\eps}}\big\|_{L^{(2\tilde{m}+3)p}(\Omega)}+\big\|\norm{O}_{C_T\calC^{1/2-\eps}}\big\|_{L^{(2\tilde{m}+3)p}(\Omega)}\Big),
\end{align*}
where the right hand side in the above is finite by \cref{cor:wMN1} and \eqref{eq:Otimespace}. Therefore \eqref{eq:sewingassump2} is fulfilled for $\epsilon_2= \frac{3}{4}-\frac{\alpha}{2}-\epsilon>0$, which is positive due to the choice on $\epsilon$, $\Gamma_2=C N^{-\alpha}$, $\delta_2=0$.
To also verify \eqref{eq:sewingassump1}, we use Lemma~\ref{lem:operator} with $\alpha>1$, tower property of conditional expectation as well as the product estimate \eqref{eq:product} from \cref{lem:product-est}, so that for $\xi \in [u,v]$ and small enough $\varepsilon\in (0,\frac{3}{4}-\frac{\alpha}{2})$, with the notation for $\mathfrak{a},\mathfrak{c}$ from above,
\begin{align*}
    \|&\E_u\delta A_{u,\xi,v}\|_{L^2}\\
    &=\norm[\bigg]{\int_\xi^v \E_u(\operatorname{Id}-\Psi_N \Theta_N) [(P^{\mathbb{R}}_{\Qx{\xi}{r}}g_h)(\E_u z_r^{M,N}+P_{r-\xi}O_r)-(P_{\Qx{\xi}{r}}^{\R}g_h)(\E_\xi z_r^{M,N} + P_{r-\xi}O_r)] \, dr}_{L^2} \\
    &\lesssim N^{-\alpha} \E_u\int_{\xi}^v \Big\|\Big(\int_0^1 (P^{\mathbb{R}}_{\Qx{\xi}{r}}g_h)^\prime (\lambda \mathfrak{c}_u + (1-\lambda) \mathfrak{c}_{\xi}+\mathfrak{a}_{\xi,r}) \, d\lambda\Big)(\mathfrak{c}_u -\mathfrak{c}_{\xi})\Big\|_{H^\alpha} \, dr\\
    &\lesssim N^{-\alpha} \E_u\int_\xi^v \underbrace{\Big\|\int_0^1 (P^{\mathbb{R}}_{\Qx{\xi}{r}}g_h)^\prime (\lambda \mathfrak{c}_u + (1-\lambda) \mathfrak{c}_{\xi}+\mathfrak{a}_{\xi,r}) \, d\lambda\Big\|_{\calC^{\alpha+\epsilon}}}_{\coloneqq P_1 (r)} \|\mathfrak{c}_u - \mathfrak{c}_{\xi}\|_{H^\alpha} \, dr.
\end{align*}
To estimate $P_1(r)$, we apply the composition estimate, \cref{lem:composition}, to $(P^{\mathbb{R}}_{\Qx{\xi}{r}}g_h)^\prime$ and $v=\lambda \mathfrak{c}_u + (1-\lambda) \mathfrak{c}_{\xi}+\mathfrak{a}_{\xi,r}$ and $\theta=\alpha+\epsilon\in (1,2)$, where the composition estimate applies to that setting by \cref{lem:composition-applies} using now that $g_h\in C^{3}_{\omega}$, so that $(P^{\R}_{Q(\xi,r)}g_{h})^{\prime}\in C^{2}_{\omega}$, $\beta=2\tilde{m}+1$. This yields, by similar arguments as for $A_{u,v}$ (in particular using \cref{lem:semigroup1}, \cref{lem:OUreg} and the estimate of $z^{M,N}$, i.e. \cref{cor:wMN1}),
\begin{align*}
&\|P_1(r)\|_{L^{2p}(\Omega)}
\\&\lesssim \squeeze[1]{\norm{(1+ \|z^{M,N}\|_{C_TL^{\infty}}^{2\tilde{m}+1}+\norm{O}_{C_TL^{\infty}}^{2\tilde{m}+1})(1+\|z^{M,N}\|_{C_T\calC^{3/2-\epsilon}}^2+\|\mathfrak{a}_{\xi,r}\|_{C^{1}_b}\norm{\mathfrak{a}_{\xi,r}}_{\calC^{\alpha+\epsilon-1}}+ \|\mathfrak{a}_{\xi,r}\|_{\calC^{\alpha+\epsilon}})}_{L^{2p}(\Omega)}}
\\&\lesssim (1+(r-\xi)^{1/4-\alpha/2-\epsilon}) (1+\norm{O}_{C_T\calC^{1/2-\epsilon}}^{2\tilde{m}+3}+\norm{z^{M,N}}_{C_T \calC^{3/2-\eps}}^{2\tilde{m}+3}).
\end{align*}
Further, we see that for $u\leq\xi\leq r\leq v$, since $\alpha<3/2$,
\begin{align*}
\|\|\mathfrak{c}_u - \mathfrak{c}_{\xi}\|_{H^\alpha}\|_{L^{2p}(\Omega)}&\lesssim \|\|z_r^{M,N} - z_\xi^{M,N}\|_{H^\alpha}\|_{L^{2p}(\Omega)}+\|\|z_r^{M,N} -  z_u^{M,N}\|_{H^\alpha}\|_{L^{2p}(\Omega)}\\&\lesssim  \|\|z^{M,N}\|_{C^{1/2}_T H^{\alpha}}\|_{L^{2p}(\Omega)}(r-u)^{1/2}
\\&\lesssim (r-u)^{1/2}
\end{align*}
using the a priori estimate, Lemma~\ref{cor:wMN1}. Hence, we get
\begin{align*}
    \|\|\E_u \delta A_{u,\xi,v}\|_{L^2}\|_{L^p(\Omega)} 
    &\lesssim N^{-\alpha} \Big(\int_{\xi}^v(1+ (r-\xi)^{1/4-\alpha/2-\epsilon})(r-u)^{1/2}\, dr\Big)\\&\lesssim N^{-\alpha} (v-u)^{7/4-\alpha/2-\epsilon}.
\end{align*}
Due to the choice of $\epsilon$, we have that $\epsilon_1=\frac{3}{4}-\frac{\alpha}{2}-\epsilon>0$ and thus \eqref{eq:sewingassump1} follows with $\Gamma_1= C N^{-\alpha},\delta_1=0$.

It remains to identify the sewing integral $(\mathcal{A}_{v})_{v\in[s,t]}$ as $\mathcal{A}_{v}=\int_{s}^{v} P_{t-r}((\operatorname{Id}-\Psi_N\Theta_N)g_h(z_r^{M,N}+O_r))\, dr$ for $0\leq s\leq v\leq t$. We do so, using that it is the unique process fulfilling \eqref{eq:SSL-conc1} and \eqref{eq:SSL-conc2}, i.e. controlling
\begin{align*}
R_1&\coloneqq\int_{u}^{v} P_{t-r}((\operatorname{Id}-\Psi_N\Theta_N)g_h(z_r^{M,N}+O_r))\, dr- \int_{u}^{v} P_{t-r}((\operatorname{Id}-\Psi_N\Theta_N)g_h(\E_u z_r^{M,N}+P_{r-u}O_u))\, dr
\end{align*}
and $R_2\coloneqq \E_{u} R_1$. Again, we use the growth bound on $g_h$ and the a priori estimates for $z^{M,N}$ and $O$, that we have used multiple times throughout the proof, for each of the two summands of $R_1$ separately, as well as that $P_{t-r}$ is a contraction on $L^2({\T})$  and $\|\Psi_N \Theta_N f\|_{L^2}\leqslant \|f\|_{L^{\infty}}$ by \cref{lem:Theta-Psi} c) and f) for $f\in C(\T)$. Then we obtain that $\Big\|\|R_1\|_{L^2}\Big\|_{L^2(\Omega)}\lesssim (v-u)$ and thus \eqref{eq:SSL-conc1} holds for $C_1=C_2=C(T,K,m,p)$, where the dependence comes from the a priori estimates.
To treat $R_2$, we additionally use the contraction property of conditional expectation to see that
\begin{align*}
    \Big\|\|R_2\|_{L^2}\Big\|_{L^2(\Omega)}&\leqslant \int_{u}^{v} \Big\| \|g_h(z_r^{M,N}+O_r)-g_h(\E_{u} z^{M,N}_r+P_{r-u}O_u)\|_{L^{\infty}}\Big\|_{L^2(\Omega)}\, dr.
\end{align*}
Hence by the growth bound on $g_h$, semigroup estimates and the regularity of $O$ from \cref{lem:OUreg} and $z^{M,N}$ from \cref{cor:wMN1}, we deduce that for $\epsilon\in (0,1/2)$
\begin{equation*}
\Big\|\|R_2\|_{L^2}\Big\|_{L^2(\Omega)}\lesssim (v-u)^{5/4-\epsilon/2} \paren[\big]{\norm{\norm{z^{M,N}}_{C_T^1L^{\infty}}}_{L^{2(2m+1)}(\Omega)}+  \norm{\norm{O}_{C_T\calC^{1/2-\epsilon}}}_{L^{2(2m+1)}(\Omega)}}
\end{equation*}
and thus \eqref{eq:SSL-conc2} follows for $C_3=C(T,m,K,p)$.
Hence, the claim \eqref{eq:claim1} follows by \eqref{eq:resultSewing}, from which the lemma follows by an application of a version of Kolmogorov's continuity theorem from \cref{prop:vKolmogorov}. 
\end{proof}

\begin{lemma}\label{lem:vhatz}
    Let $\alpha \in (1,3/2)$ and $p\geqslant 1$. Then there exists $C=C(\alpha, K, m, T,p)$, that does not depend on $M,N$, such that 
\begin{equation*}
    (\E\sup_{t\in[0,T]}\|v^M_t - \hat{\mathbf{z}}_t^{M,N}\|^{p}_{L^2(\T)})^{1/p}\leqslant C N^{-\alpha}.
\end{equation*}
\end{lemma}
\begin{proof}
We know that $v^M-\hat{\mathbf{z}}^{M,N}$ weakly solves (cf. \cite[Remark 6.1]{DjGK} and \cite{bell})
\begin{align*}
\partial _t (v^M_t -\hat{\mathbf{z}}^{M,N}_t) = \Delta (v^M_t -\hat{\mathbf{z}}^{M,N}_t) + g_{h}(v^M_t + O_t)- g_h (z^{M,N}_t + O_t), \quad v^M_0 -\hat{\mathbf{z}}^{M,N}_0= u_0- \Psi_N\Theta_N u_0.
\end{align*}
Applying Lemma~\ref{lem:operator} for the initial condition, Young's product inequality and the properties of $g_h$ (see Lemma~\ref{lem:g_h-bounds}), we get
\begin{align*}
    \|v^M_t &- \hat{\mathbf{z}}_t^{M,N}\|^2_{L^2(\T)}= \|P_t (u_0-\Psi_N\Theta_Nu_0 )\|_{L^2(\T)}^2 +\int_0^t \langle v_s^M -\hat{z}_s^{M,N},\Delta (v_s^M-\hat{\mathbf{z}}_s^{M,N})\rangle\, ds \\
    &\quad \quad \quad + \int_0^t \langle v^M_s-\hat{\mathbf{z}}^{M,N}_s,g_h(v^M_s +O_s )-g_h (z^{M,N}_s+O_s)\rangle \, ds\\
    &\leqslant C N^{-2\alpha}\|u_0\|_{H^{\alpha}(\T)}^2\underbrace{-\int_0^t \langle \nabla (v_s^M-\hat{\mathbf{z}}^{M,N}_s),\nabla (v_s^M-\hat{\mathbf{z}}^{M,N}_s)\rangle \, ds}_{\leqslant 0} \\
    &\quad  \quad \quad + \int_0^t\langle v^M_s-\hat{\mathbf{z}}_s^{M,N},g_h(v^M_s+O_s)-g_h(z_s^{M,N}+O_s)\rangle \, ds\\
    &\leqslant C N^{-2\alpha}\|u_0\|_{H^{\alpha}(\T)}^2+\int_0^t \langle v^M_s-\hat{\mathbf{z}}_s^{M,N},g_h(v^M_s+O_s)-g_h(\hat{\mathbf{z}}_s^{M,N}+O_s)\rangle\, ds\\
    &\quad \quad \quad +\int_0^t \langle v^M_s-\hat{\mathbf{z}}^{M,N}_s,g_h(\hat{\mathbf{z}}^{M,N}_s+O_s)-g_h(z^{M,N}_s+O_s)\rangle\, ds\\
    &\leqslant C N^{-2\alpha}\|u_0\|_{H^{\alpha}(\T)}^2+K \int_0^t \|v^M_s-\hat{\mathbf{z}}^{M,N}_s\|_{L^2(\T)}^2\, ds + \frac{1}{2} \int_0^t \|v^M_s-\hat{\mathbf{z}}^{M,N}_s\|^2_{L^2(\T)}\, ds \\
    &\quad \quad \quad +\frac{1}{2} \int_0^t \|g_h(\hat{\mathbf{z}}_s^{M,N}+O_s)-g_h(z^{M,N}_s+O_s)\|^2_{L^2(\T)}\, ds \\
    &\leqslant C N^{-2\alpha}\|u_0\|_{H^{\alpha}(\T)}^2+K \int_0^t \|v^M_s-\hat{\mathbf{z}}^{M,N}_s\|_{L^2(\T)}^2\, ds + \frac{1}{2} \int_0^t \|v^M_s-\hat{\mathbf{z}}^{M,N}_s\|^2_{L^2(\T)}\, ds \\
    &\quad \quad \quad \quad + \tilde{K}^2(1+\|\hat{\mathbf{z}}^{M,N}+O\|_{C_TL^\infty(\T)}^{2\tilde{m}+1}+\|z^{M,N}+O\|_{C_TL^\infty(\T)}^{2\tilde{m}+1})^2 \int_0^t \|\hat{\mathbf{z}}^{M,N}_s-z^{M,N}_s\|_{L^2(\T)}^2\, ds.
\end{align*}
Applying Gr\"onwall's inequality, we thus obtain
\begin{align*}
    &\|v^M_t - \hat{\mathbf{z}}_t^{M,N}\|^2_{L^2(\T)}\\
    &\leqslant C\Big(N^{-2\alpha}\|u_0\|^2_{H^\alpha(\T)}+(1+\|\hat{\mathbf{z}}^{M,N}+O\|_{C_TL^\infty(\T)}^{2\tilde{m}+1}+\|z^{M,N}+O\|_{C_TL^\infty(\T)}^{2\tilde{m}+1})^2 \|\hat{\mathbf{z}}^{M,N}-z^{M,N}\|_{C_TL^2(\T)}^2\Big)
\end{align*}
Recall the a priori estimates for $z^{M,N},\hat{\mathbf{z}}^{M,N}$ and $O$ from Corollary~\ref{cor:wMN1} and \cref{lem:OUreg}. Therefore, after taking the $L^{p/2}(\Omega)$ norm and Cauchy-Schwarz, the result follows as $$\|\|\hat{\mathbf{z}}^{M,N}-z^{M,N}\|^2_{C_TL^2(\T)}\|_{L^p(\Omega)}\leqslant C N^{-2\alpha}$$ by Lemma~\ref{lem:ztildez} and Lemma~\ref{lem:zhatztilde}.
\end{proof}

\subsection{Temporal error} \label{subsec:temporal}

This subsection is dedicated to bounding the two terms in the decomposition \eqref{BIG-decomposition} contributing to the temporal error; i.e. $\|v_t-v_t^M\|_{L^2(\T)}$ and $\|z_t^{M,N}-\Psi_N X_t^{M,N}\|_{L^2(\T)}$. The former is directly treated in \cite{DjGK}, see \cref{lem:6.3} below. In order to prove the latter we employ the a priori estimates for $X^{M,N}$ from \cref{prop:apriori-XMN}. Then we proceed along similar lines as in \cite{DjGK} via stochastic sewing and a buckling argument.

\begin{lemma}\label{lem:6.3}
    Let $p\geq 1$. Then there exists a constant $C=C(p,K,T,m,\mathscr{M})$, that does not depend on $M$, such that 
    \begin{equ}
        \Big\|\sup_{t\in[0,T]}\|v- v^M\|_{L^2(\T)}\Big\|_{L^p(\Omega)}\leq C M^{-1}.
    \end{equ}
\end{lemma}

\begin{proof}
    The proof directly corresponds to the proof of \cite[Lemma 6.3]{DjGK} plugging $v+O$ instead of $u$ there and $v^{M}+O$ instead of $X^h$ there and using the regularity bounds on $v$ from \cref{prop:wp}, as well as the regularity bounds on $O$ from \cref{lem:OUreg}. Thus we omit further details.
\end{proof}

Having the a priori estimates for $X^{M,N}$ from \cref{prop:apriori-XMN} at hand, we can tackle the last remaining term in \eqref{BIG-decomposition}. 

\begin{proposition}\label{prop:temporal}
Let $p\geqslant 1$ and $\eps\in (0, 1/2)$. Then there exists a constant $C=C(T,p,\epsilon, K, m,\mathcal{M})$, that does not depend on $M,N$, such that 
\begin{align}
    \Big\|\sup_{t\in[0,T]}\|z^{M,N}_t-\Psi_N X^{M,N}_t\|_{L^2(\T)}\Big\|_{L^p(\Omega)}&=\Big\|\sup_{t\in[0,T]}\|w^{M,N}_t - X^{M,N}_t\|_{L^2(\Pi_N)}\Big\|_{L^p(\Omega)}\nonumber
    \\&\leq C [M^{-1+\epsilon}+N^{-2+\eps}].
\end{align}
\end{proposition}

In order to prove the above, we split the error into three terms, each of which requires different arguments to be bounded. More precisely, we decompose as 
\begin{align}\label{eq:decomposition}
    \|w_t^{M,N}&- X^{M,N}_t\|_{L^2(\Pi_N)}\nonumber\\
    &\leqslant \Big\|\int_0^tP_{t-r}^N\Big(g_h(w_r^{M,N}+\Theta_N O_r)\Big)\, dr-\int_0^t {P}_{t-r}^N \Big(g_h(X_r^{M,N}+\Theta_N O_r)\Big)\, dr\Big\|_{L^2(\Pi_N)}\nonumber\\
    &+\Big\|\int_0^t {P}_{t-r}^N \Big(g_h(X_r^{M,N}+\Theta_N O_r)\Big)\, dr-\int_0^t {P}_{t-r}^N \Big(g_h(X_{k_M(r)}^{M,N}+\Theta_N O_{k_M(r)})\Big)\, dr\Big\|_{L^2(\Pi_N)}\nonumber\\
    &+\Big\|\int_0^t {P}_{t-r}^N \Big(g_h(X_{k_M(r)}^{M,N}+\Theta_N O_{k_M(r)})\Big)\, dr-\int_0^t P^N_{t-k_M(r)}\Big(g_h(X_{k_M(r)}^{M,N}+\Theta_N O_{k_M(r)})\Big)\, dr\Big\|_{L^2(\Pi_N)}\nonumber\\
    &\coloneqq \cE^1_t+\cE^2_t+\cE^3_t.
\end{align}
We bound each in $\|\cdot\|_{L^p(\Omega)}$ norm. First, we state \cref{lem:E3,lem:E2} which allow to bound $\cE^3$ and $\cE^2$. Then we proceed with the proof of \cref{prop:temporal} with a buckling argument. The proof of \cref{lem:E3,lem:E2} is given thereafter. The proof of the former relying solely on the properties of the semigroup $P^N$; the latter follows a stochastic sewing argument.

\begin{lemma}\label{lem:E3}
Let $p\geqslant 1$, $\eps\in (0,1)$ and $\cE^3$ as in \eqref{eq:decomposition}. Then there exists a constant $C=C(T,p,\epsilon, K, m,\mathcal{M})$, that does not depend on $M,N$, such that 
\begin{align*}
\|\sup_{t\in[0,T]}\cE^3_t\|_{L^p(\Omega)}\leqslant C M^{-1+\epsilon}.
\end{align*}
\end{lemma}

\begin{lemma}\label{lem:E2}
Let $p\geqslant 1$, $\eps\in (0, 1/2)$ and $\cE^2$ as in \eqref{eq:decomposition}. Then there exists a constant $C=C(T,p,\epsilon, K,m, \mathcal{M})$, that does not depend on $M,N$, such that 
\begin{align*}
\|\sup_{t\in[0,T]}\cE^2_t\|_{L^p(\Omega)}\leqslant C [M^{-1+\epsilon}+N^{-2+\eps}].
\end{align*}
\end{lemma}

\begin{proof}[Proof of \cref{prop:temporal}]
First we define an auxiliary process for $t\in[0,T]$ by 
$$\hat{X}^{M,N}_t\coloneqq P_{t}^{N}u_0+\int_0^t {P}_{t-r}^N \Big(g_h(X_r^{M,N}+\Theta_N O_r)\Big)\, dr,$$ which differs from $X^{M,N}$ by an integration of continuous time points $r$ instead of $k_{M}(r)$. Further, since $w^{M,N}-\hat{X}^{M,N}$  weakly solves (using again \cite{bell} and \cite[Remark 6.1]{DjGK})
$$\partial_t (w^{M,N}-\hat{X}^{M,N})=\Delta_N (w^{M,N}-\hat{X}^{M,N}) + g_h(w^{M,N}+\Theta_NO)-g_h(X^{M,N}+\Theta_NO), \quad w^{M,N}_0-\hat{X}^{M,N}_0=0,$$
we obtain using \cref{lem:g_h-bounds},
\begin{align*}
    &(\cE^1_t)^2=\|w_t^{M,N}-\hat{X}^{M,N}_t\|_{L^2(\Pi_N)}^2\\
    &\quad \leqslant 2 \int_0^t \langle w_r^{M,N}-X^{M,N}_r,g_h(w_r^{M,N}+\Theta_N O_r)-g_h(X_r^{M,N}+\Theta_N O_r)\rangle\, dr\\
    &\quad \quad + 2 \int_0^t  \langle X_r^{M,N}-\hat{X}^{M,N}_r,g_h(w_r^{M,N}+\Theta_N O_r)-g_h(X_r^{M,N}+\Theta_N O_r)\rangle\, dr\\
    &\quad \leqslant 2K \int_0^t  \|w^{M,N}_r-X^{M,N}_r\|_{L^2(\Pi_N)}^2 \, dr\\
    &\quad \quad +2\tilde{K} \Big(1+\|w^{M,N}\|^{2\tilde{m}+1}_{C_T L^\infty(\Pi_N)}+\|X^{M,N}\|^{2\tilde{m}+1}_{C_T (\Pi_N)}+\|O\|^{2\tilde{m}+1}_{C_T L^\infty(\T)}\Big)\\
    &\qquad \qquad   \times \int_0^t  \|X_r^{M,N}-\hat{X}_r^{M,N}\|_{L^2(\Pi_N)}\|w_r^{M,N}-X_r^{M,N}\|_{L^2(\Pi_N)}\, dr.
\end{align*}
Hence we obtain after an application of Youngs inequality
\begin{align*}
    \|&w^{M,N}_t-\hat{X}^{M,N}_t\|_{L^2(\Pi_N)}^p\lesssim \int_0^t \|w_r^{M,N}-X_r^{M,N}\|_{L^2(\Pi_N)}^2\, dr\\
    &\quad + \Big(1+\|w^{M,N}\|^{2\tilde{m}+1}_{C_T L^\infty(\Pi_N)}+\|X^{M,N}\|^{2\tilde{m}+1}_{C_T L^\infty(\Pi_N)}+\|O\|^{2\tilde{m}+1}_{C_T L^\infty(\T)}\Big)^2 \int_0^t \|X_r^{M,N}-\hat{X}_r^{M,N}\|^2_{L^2(\Pi_N)}\, dr.
\end{align*}
Plugging this into \eqref{eq:decomposition},  we have that
\begin{align*}
    \|&w_t^{M,N}-X_t^{M,N}\|_{L^2(\Pi_N)}^p\lesssim (\cE^2_t)^2+(\cE^3_t)^2+ \int_0^t \|w_r^{M,N}-X_r^{M,N}\|_{L^2(\Pi_N)}^2\, dr\\
    &\quad + \Big(1+\|w^{M,N}\|^{2\tilde{m}+1}_{C_T L^\infty(\Pi_N)}+\|X^{M,N}\|^{2\tilde{m}+1}_{C_T L^\infty(\Pi_N)}+\|O\|^{2\tilde{m}+1}_{C_T L^\infty(\T)}\Big)^2 \int_0^t \|X_r^{M,N}-\hat{X}_r^{M,N}\|^2_{L^2(\Pi_N)}\, dr.
\end{align*}
Applying Gr\"onwall's inequality gives
\begin{align}\label{eq:beforeexp}
    \|&w_t^{M,N}-X_t^{M,N}\|_{L^2(\Pi_N)}^2\lesssim (\cE^2_t)^2+(\cE^3_t)^2\\
    &\quad + \Big(1+\|w^{M,N}\|^{2\tilde{m}+1}_{C_T L^\infty(\Pi_N)}+\|X^{M,N}\|^{2\tilde{m}+1}_{C_T L^\infty(\Pi_N)}+\|O\|^{2\tilde{m}+1}_{C_T L^\infty(\T)}\Big)^2 \int_0^t \|X_r^{M,N}-\hat{X}_r^{M,N}\|^2_{L^2(\Pi_N)}\, dr.\nonumber
\end{align}
Note that we have moment bounds for the bracket in the above by \cref{prop:apriori-XMN,cor:wMN1}. Moreover, $X_t^{M,N}-\hat{X}_t^{M,N}$ can be rewritten as
\begin{align*}
    \int_0^t P_{t-k_M(r)}^N&\Big(g_h(X_{k_M(r)}^{M,N}+\Theta_NO_{k_M(r)})\Big) \, dr-\int_0^t P^N_{t-r}\Big(g_h(X^{M,N}_{k_M(r)}+\Theta_N O_{k_M(r)})\Big)\, dr \\
    &+\int_0^t P^N_{t-r}\Big(g_h(X^{M,N}_{k_M(r)}+\Theta_N O_{k_M(r)})\Big)\, dr-\int_0^t P^N_{t-r}\Big(g_h(X^{M,N}_r+\Theta_N O_r\Big)\, dr
\end{align*}
and its squared $L^2(\Pi_N)$-norm can thus be bounded by $(\cE^2_t)^2+(\cE^3_t)^2$.
Hence, by \cref{lem:E2,lem:E3}, we can bound for any $q\geq 1$
\begin{align*}
\E\Big[\sup_{t \in [0,T]}\|X_t^{M,N}-\hat{X}_t^{M,N}\|_{L^2(\Pi_N)}^q\Big]\leqslant C(q,T,\varepsilon,K,\mathscr{M}) [M^{q(-1+\varepsilon)}+N^{q(-2+\varepsilon)}].
\end{align*}

Overall, using again \cref{lem:E3,lem:E2} to bound the moments of $\cE^2$ and $\cE^3$, this gives 
\begin{equation*}
\Big\|\sup_{t\in[0,T]}\|w_t^{M,N}-X_t^{M,N}\|_{L^2(\Pi_N)}^2\Big\|_{L^{p/2}(\Omega)}^{1/2}\leq C(q,T,\varepsilon,K,\mathscr{M}) [M^{-1+\varepsilon}+N^{-2+\varepsilon}].\qedhere
\end{equation*}
\end{proof}

\begin{proof}[Proof of \cref{lem:E3}]
Applying \cref{lem:semigroup1} (similar as in the end of the proof of \cref{prop:apriori-XMN}), \eqref{eq:ghgrowth} and the a priori estimate \cref{prop:apriori-XMN}, we have
\begin{align*}
    \cE^3_t=&\Big\|\int_0^t {P}_{t-r}^N(\operatorname{Id}-P^N_{r-k_M(r)}) \Big(g_h(X_{k_M(r)}^{M,N}+\Theta_N O_{k_M(r)})\Big)\, dr\Big\|_{L^2(\Pi_N)}\\
    &\lesssim \Big(1+\|X^{M,N}\|_{C_T L^{\infty}(\Pi_N)}^{2\tilde{m}+1}+\|\Theta_N O\|_{C_T L^{\infty}(\Pi_N)}^{2\tilde{m}+1}\Big)\\
    &\quad \times \int_0^t (t-r)^{-1+\varepsilon} (r-k_M(r))^{1-\varepsilon} \, dr\lesssim M^{-1+\varepsilon},
\end{align*}
from which the claim follows after taking the supremum in $t\in[0,T]$ and the $p$-th moment.
\end{proof}

\begin{proof}[Proof of \cref{lem:E2}]
The proof resembles the proof of Lemma~\ref{lem:temp-ss} and follows similar arguments as the one from \cite[Proposition 6.5]{DjGK}.
We prove that for any $p\geq 1$, $0\leq s\leq t \leq T$
\begin{align*}
\MoveEqLeft
  \paren[\bigg]{\E \norm[\bigg]{\int_s^t {P}_{t-r}^N \Big[g_h(\E_u X_r^{M,N}+\Theta_N O_r)-g_h(\E_u X_{k_M(r)}^{M,N}+\Theta_NO_{k_M(r)})\Big]\, dr}_{L^2(\Pi_N)}^p}^{1/p}
  \\&\lesssim [M^{-1+2\epsilon}+N^{-2+2\epsilon}](t-s)^{\frac{1}{4}+\frac{\epsilon}{2}}
\end{align*}
from which the claim follows by an application of Kolmogorov's continuity theorem, \cref{prop:vKolmogorov}. 
Let $0\leq s\leq t\leq T$ be arbitrary and let for $(u,v) \in [s,t]_\leq$ the germ be defined as
\begin{align*}
    A_{u,v}=\E_u \int_u^v {P}_{t-r}^N \Big[g_h(\E_u X_r^{M,N}+\Theta_N O_r)-g_h(\E_u X_{k_M(r)}^{M,N}+\Theta_NO_{k_M(r)})\Big]\, dr.
\end{align*}
In order to verify \eqref{eq:sewingassump2}, we bound $\|\norm{A_{u,v}}_{L^2(\Pi_N)}\|_{L^p(\Omega)}$, which works as in the case of bounded $f$, however additionally taking into account the growth of $g_h$ and the remainder term $$R_{u,r}\coloneqq \E_u X_{r}^{M,N}.$$
Throughout the proof, we use that the latter is $\mathcal{F}_u$-measurable multiple times without always explicitly mentioning it. Moreover we use several times that by \cref{prop:apriori-XMN}, we have the estimate for any $p\geq 1$ and $\eps\in (0,1/2)$,
\begin{align} \label{eq:aprioriR}
    \Big\|\|R_{u,r}-R_{u,\eta}\|_{L^2(\Pi_N)}\Big\|_{L^p(\Omega)}\leq \Big\|\|R_{u,r}-R_{u,\eta}\|_{\cC^{\frac{1}{2}-\eps}(\Pi_N)}\Big\|_{L^p(\Omega)}\leq C |r-\eta|^{1-\varepsilon}.
\end{align}
The above follows from the a priori bound for $X^{M,N}$  in \cref{prop:apriori-XMN} and the contraction property of conditional expectation. Throughout the proof, we bound moments of $R$ by moments of $X^{M,N}$, which works by contraction property of conditional expectation as, for $u$ fixed,
\begin{align}\label{eq:condexp}
\sup_{u\leq r\leq T}\|\|R_{u,r}\|_{L^\infty(\Pi_N)}\|_{L^p(\Omega)} \nonumber
&=\sup_{u\leq r\leq T}\|\|\mathbb{E}_u[X_r^{M,N}]\|_{L^\infty(\Pi_N)} \|_{L^p(\Omega)} \\
&\leqslant\sup_{u\leq r\leq T}\|\mathbb{E}_u[ \|X_r^{M,N}\|_{L^\infty(\Pi_N)}] \|_{L^p(\Omega)} \nonumber\\
&\leqslant\|\|X^{M,N}\|_{C_TL^\infty(\Pi_N)} \|_{L^p(\Omega)}
\end{align}
Additionally, by the composition estimate, \cref{lem:composition}, we have
\begin{align}\label{eq:g_h-comp}
    \Big\|&\sup_{\lambda\in[0,1]}\sup_{r \in [u,T]} \Big\|g_h^\prime\Big(\lambda(R_{u,r}+\Theta_N O_r)+(1-\lambda) (R_{u,k_M(r)}+\Theta_N O_{k_M(r)})\Big) \Big\|_{\mathcal{C}^{1/2-\eps}(\Pi_N)}\Big\|_{L^{2p}(\Omega)}\\
    &\lesssim K\Big(1+\big\|\|X^{M,N}\|_{C_T L^{\infty}(\Pi_N)}^{2\tilde{m}}\big\|_{L^{2p}(\Omega)}+\big\|\|\Theta_N O\|_{C_T L^{\infty}(\Pi_N)}^{2\tilde{m}}\big\|_{L^{2p}(\Omega)}\Big)\nonumber\\
    &\quad \quad \times\Big(1+\big\|\|X^{M,N}\|_{C_T\cC^{1/2-\varepsilon}(\Pi_N)}\big\|\big\|_{L^{2p}(\Omega)}+\big\|\|\Theta_N O\|_{C_T\mathcal{C}^{1/2-\eps}(\Pi_N)}\big\|_{L^{2p}(\Omega)}\Big)\lesssim 1,\nonumber
\end{align}
which is bounded by a constant that does not depend on $N,M$ due to the a priori estimates for $X^{M,N}$ from \cref{prop:apriori-XMN} and the bounds for $O$ from \cref{lem:OUreg}, using that $\norm{\Theta_N O_t}_{\calC^{1/2-\varepsilon}(\Pi_N)}\leq \norm{O_t}_{\calC^{1/2-\varepsilon}(\T)}$ by property f) from \cref{lem:Theta-Psi}.\\
In order to prove \eqref{eq:sewingassump2}, we first assume that $|v-u|\leqslant 3M^{-1}$. In this case, we bound using \cref{lem:semigroup2}, \eqref{eq:aprioriR}, the product estimate \cref{lem:productestnegative}, \eqref{eq:g_h-comp} and \eqref{eq:Ot-Os} from \cref{lem:Otreg} for $\alpha=\frac{1}{2}-2\eps$,
\begin{align*}
\MoveEqLeft
  \norm{\|A_{u,v}\|_{L^2(\Pi_N)}}_{L^p(\Omega)}
   \\&\leq \norm[\bigg]{\int_{u}^{v} (t-r)^{-1/4+\eps/2}\norm[\bigg]{\paren[\bigg]{\int_{0}^{1} g_h^\prime(\lambda(R_{u,r}+\Theta_N O_r)+(1-\lambda) (R_{u,k_M(r)}+\Theta_N O_{k_M(r))} \, d\lambda} 
   \\&\qquad\qquad\times \paren[\big]{R_{u,r}+\Theta_N O_r-R_{u,k_M(r)}-\Theta_N O_{k_M(r)}}}_{H^{-1/2+\eps}(\Pi_N)} \, dr}_{L^{p}(\Omega)}
   \\&\leqslant (t-v)^{-1/4+\eps/2}
   \\&\qquad\qquad\times \norm{\sup_{\lambda\in [0,1]}\sup_{r \in [u,T]} \norm{g_h^\prime\big(\lambda(R_{u,r}+\Theta_N O_r)+(1-\lambda) (R_{u,k_M(r)}+\Theta_N O_{k_M(r)})\big)}_{\calC^{1/2-\eps/2}(\Pi_N)}}_{L^{2p}(\Omega)} 
   \\&\qquad\qquad\times\norm[\bigg]{\int_{u}^{v}\norm{R_{u,r}+\Theta_N O_r-R_{u,k_M(r)}-\Theta_N O_{k_M(r)}}_{H^{-1/2+\eps}(\Pi_N)}\, dr}_{L^{2p}(\Omega)}
   \\&\lesssim (t-v)^{-1/4+\eps/2}(v-u) M^{-1+\varepsilon} 
   \\&\qquad + (t-v)^{-1/4+\eps/2}(v-u) [M^{-1/2+\eps}+N^{-1+\eps}]
   \\&\lesssim (t-v)^{-1/4+\eps/2}\paren[\big]{(v-u)M^{-1+\eps} + (v-u)^{1/2+\epsilon} M^{-1/2+\eps} [M^{-1/2+\epsilon}+N^{-1+\eps}]}
   \\&\lesssim (t-v)^{-1/4+\eps/2}(v-u)^{1/2+\eps}[M^{-1+2\eps}+N^{-2+2\eps}]
\end{align*}
where in the penultimate inequality we employed that $$v-u=(v-u)^{1/2+\eps}(v-u)^{1/2-\eps}\leq (v-u)^{1/2+\eps} M^{-1/2+\eps}.$$
Next we consider the case $|v-u|>3M^{-1}$. Let $t^\prime=k_M(u)+3M^{-1}$. As in the proof of Lemma~\ref{lem:temp-ss}, we split the integral into two integrals on $[u,t^\prime]$ and $[t^\prime,v]$. The former can be treated as the first case because $|t^\prime-u|\leqslant 3M^{-1}$. The latter integral is left to treat.
Before we treat the integral on $[t^\prime,v]$, we introduce the shorthand notation $$\cK_{u,t}:=R_{u,t}+\Theta_N P_{t-u}O_u \text{ for } u\leq t.$$ Note that for $\cK$, we have the following bound, for any $q\geqslant1$,
\begin{align}
    \sup_{0\leq u\leq r\leq T} \|\|\cK_{u,r}\|_{\cC^{1/2-\eps}(\Pi_N)}\|_{L^{q}(\Omega)}\lesssim \|\|X^{M,N}\|_{C_T\cC^{1/2-\eps}(\Pi_N)}\|_{L^q(\Omega)}+\|\|\Theta_NO\|_{C_T\cC^{1/2-\eps}(\Pi_N)}\|_{L^q(\Omega)}\lesssim 1,\label{eq:cK1}
\end{align}
uniformly in $N,M$ by the chain of inequalities \eqref{eq:condexp}, the a priori estimates for $X^{M,N}$ from Proposition~\ref{prop:apriori-XMN} and the bounds for the OU process.
Moreover, by \eqref{eq:aprioriR} and \eqref{eq:Ot-Os} from \cref{lem:Otreg} applied for $\alpha=1/2-2\eps$ and $\beta=1-2\eps$ we obtain for $0\leq u\leq \hat{r}\leq r\leq T$
\begin{align}\label{eq:cK2}
    \|\|\cK_{u,r}&-\cK_{u,\hat{r}}\|_{H^{-1/2+\eps}(\Pi_N)}\|_{L^q(\Omega)}\nonumber\\
    &\leqslant \norm{\norm{R_{u,r}-R_{u,\hat{r}}}_{H^{-1/2+\eps}(\Pi_N)}}_{L^{q}(\Omega)} + \norm{\norm{\Theta_N P_{r-u} O_u-\Theta_N P_{\hat{r}-u}O_u}_{H^{-1/2+\eps}(\Pi_N)}}_{L^{q}(\Omega)}\nonumber
    \\&\lesssim (r-\hat{r})^{1-\varepsilon}+[M^{-1+2\eps}+N^{-2+3\eps}](\hat{r}-u)^{-1/2+\eps}.
\end{align}
 First, we again compute the conditional expectation, similar as in the proof of \cref{lem:temp-ss}, noting that $k_{M}(r)\geq u$ for $r\in[t^\prime, v]$, and obtain that almost surely
\begin{align*}
    \E_u \int_{t^\prime}^v &{P}_{t-r}^N \Big[g_h(R_{u,r}+\Theta_N O_r)-g_h(R_{u,k_M(r)}+\Theta_NO_{k_M(r)})\Big]\, dr\\
    &=\int_{t^\prime}^v {P}_{t-r}^N \Big[(P_{\Qx{u}{r}}^{\mathbb{R}}g_h)(\cK_{u,r})-(P_{\Qx{u}{r}}^\mathbb{R}g_h)(\cK_{u,k_M(r)})\Big]\, dr\\
    &\quad +\int_{t^\prime}^v {P}_{t-r}^N \Big[(P_{\Qx{u}{r}}^\mathbb{R}g_h-P_{\Qx{u}{k_M(r)}}^\mathbb{R}g_h)(\cK_{u,k_M(r)})\Big]\, dr\eqqcolon E_1+E_2.
\end{align*}
We start by estimating $E_1$. We have by \cref{lem:semigroup2} and the product estimate, \cref{lem:productestnegative}, 
\begin{align*}
    \|E_1\|_{L^2(\Pi_N)}&\lesssim (t-v)^{-1/4+\eps/2}\int_{t^\prime}^v \underbrace{\Big\|\int_0^1 (P_{Q(r-u)}^\mathbb{R}g_h)^\prime\big(\lambda \cK_{u,r} +(1-\lambda) \cK_{u,k_M(r)}\big)\,d\lambda\Big\|_{\calC^{1/2-\eps/2}(\Pi_N)}}_{P_1}\\
    &\qquad\qquad \times \underbrace{\Big\|\cK_{u,r}-\cK_{u,k_M(r)}\Big\|_{H^{-1/2+\eps}(\Pi_N)}}_{P_2}\, dr.
\end{align*}
Taking $L^p(\Omega)$ norms, one can bound the norm of $P_1, P_2$ by H\"older's inequality, the composition estimate in \cref{lem:composition} and using \eqref{eq:cK1} and \eqref{eq:cK2} (for $\hat{r}=k_{M}(r)$).
After integration in $r$ and using that $k_M(r)-u\geqslant (r-u)/2$ for $r\in[t^\prime,v]$, this gives overall that
\begin{align*}
\|\|&E_1\|_{L^2(\Pi_N)}\|_{L^p(\Omega)}\lesssim [M^{-1+2\eps}+N^{-2+3\eps}] (t-v)^{-1/4+\eps/2}(v-u)^{1/2+\eps}.
\end{align*}
Next, we bound $E_2$ as
\begin{align*}
    \|\|E_2\|_{L^2(\Pi_N)}\|_{L^p(\Omega)}\leqslant \int_{t^\prime}^v \Big\|\|(P_{\Qx{u}{r}}^\mathbb{R}g_h-P_{\Qx{u}{k_M(r)}}^\mathbb{R}g_h)(\cK_{u,k_M(r)})\|_{L^2(\Pi_N)}\Big\|_{L^p(\Omega)}\, dr.
\end{align*}
Using the weighted spaces $C^k_\omega$ for a polynomial weight $\omega(x)=(1+\abs{x}^2)^{-\beta/2}$ for $\beta=2\tilde{m}+1$, such that $g_h\in C^{3}_{\omega}$, and the estimate \eqref{eq:R-semigroup-weighted} for $P^{\R}$ on weighted spaces, as well as \cref{lem:QQN}, we see that
\begin{align}\label{eq:twosemigroup}
\MoveEqLeft
 \|(P_{\Qx{u}{r}}^\mathbb{R}g_h-P_{\Qx{u}{k_M(r)}}^\mathbb{R}g_h)(\cK_{u,k_M(r)})\|_{L^2(\Pi_N)}
 \\&\leqslant \|(P_{\Qx{u}{r}}^\mathbb{R}g_h-P_{\Qx{u}{k_M(r)}}^\mathbb{R}g_h)(\cK_{u,k_M(r)})\|_{L^\infty(\Pi_N)}   \nonumber
 \\&\leqslant \norm{P_{\Qx{u}{r}}^\mathbb{R}g_h-P_{\Qx{u}{k_M(r)}}^\mathbb{R}g_h}_{C^{0}_{\omega}}\norm{\omega^{-1} (\cK_{u,k_M(r)})}_{L^{\infty}(\Pi_N)}\nonumber
 \\&\lesssim \paren[\Big]{\Qx{u}{r}-\Qx{u}{k_M(r)}} \norm{g_h}_{C^{2}_\omega} (1+ \E_u\norm{X_{k_M(r)}^{M,N}}_{L^{\infty}(\Pi_N)}^{2\tilde{m}+1}+\norm{P_{k_M(r)-u}O_u}_{L^{\infty}(\T)}^{2\tilde{m}+1})\nonumber
\end{align}
Taking the $p$-th moment and contraction property of conditional expectation, we get
\begin{align*}
     \||(&P_{\Qx{u}{r}}^\mathbb{R}g_h-P_{\Qx{u}{k_M(r)}}^\mathbb{R}g_h)(\cK_{u,k_M(r)})\|_{L^2(\Pi_N)}\|_{L^p(\Omega)}\\
     &\lesssim\paren[\Big]{\Qx{u}{r}-\Qx{u}{k_M(r)}} \norm{g_h}_{C^{2}_\omega} (1+ \sup_{M,N}\|\norm{X^{M,N}}_{C_TL^{\infty}(\Pi_N)}^{2\tilde{m}+1}\|_{L^p(\Omega)}+\|\norm{O}_{C_TL^{\infty}(\T)}^{2\tilde{m}+1}\|_{L^p(\Omega)})
\end{align*}
and integrating in $r$ using again that $k_M(r)-u\geqslant (r-u)/2$ and using \cref{lem:OUreg,prop:apriori-XMN} thus yields
\begin{align*}
\|\|&E_2\|_{L^2(\Pi_N)}\|_{L^p(\Omega)}\lesssim M^{-1+\eps}(v-u)^{1/2+\eps}.  
\end{align*}
Together we obtain \eqref{eq:sewingassump2} for $\epsilon_2=\epsilon, \delta_2= \frac{1}{4}-\frac{\eps}{2}, \Gamma_2= C(T,\eps,K,m,\mathcal{M}) (M^{-1+2\eps}+N^{-2+3\eps})$.
It remains to bound $\E_u \delta A_{u,\xi,v}$ for $\xi \in [u,v]$ in order to verify \eqref{eq:sewingassump1}. We have that
\begin{align*}
    \E_u \delta A_{u,\xi,v} &= \int_{\xi}^{v} \E_{s}\E_{\xi}P^{N}_{t-r}[g_h(R_{u,r}+\Theta_N O_r)-g_h(R_{u,k_M(r)}+\Theta_N O_r)
    \\&\qquad-g_h(R_{\xi,r}+\Theta_N O_r)+g_h(R_{\xi,k_M(r)}+\Theta_N O_r)] \, dr.
\end{align*}
To do so, we use that
\begin{align*}
    g_h(&R_{u,r}+\Theta_N O_r)-g_h(R_{u,k_M(r)}+\Theta_N O_r)-g_h(R_{\xi,r}+\Theta_N O_r)+g_h(R_{\xi,k_M(r)}+\Theta_N O_r)\\
    &=(R_{u,r}-R_{\xi,r})\int_0^1 \Big(g_h^\prime(\lambda(R_{u,r}+\Theta_N O_r)+(1-\lambda)(R_{\xi,r}+\Theta_N O_r))\\
    &\quad \quad -g_h^\prime(\lambda(R_{u,k_M(r)}+\Theta_N O_{k_M(r)})+(1-\lambda)(R_{\xi,k_M(r)}+\Theta_N O_{k_M(r)}))\Big)\, d\lambda\\
    &+(R_{u,k_M(r)}-R_{u,r}+R_{\xi,r}-R_{\xi,k_M(r)})\\
    &\quad \quad \times \int_0^1 g_h^\prime(\lambda(R_{u,k_M(r)}+\Theta_N O_{k_M(r)})+(1-\lambda)(R_{\xi,k_M(r)}+\Theta_N O_{k_M(r)}))\,d\lambda\coloneqq J_1+J_2.
\end{align*}
We start by estimating the term corresponding to $J_2$. In the below, there appears a factor almost identical to $P_1$ from before and its moments can be bounded similar as the moments of $P_1$. Additionally, using \eqref{eq:aprioriR} twice, we get
\begin{align*}
   \int_\xi^v \Big\|&\|J_2\|_{L^2(\Pi_N)}\Big\|_{L^p(\Omega)}\, dr\\
   &\lesssim \int_\xi^v \Big\|\|R_{u,k_M(r)}-R_{u,r}+R_{\xi,r}-R_{\xi,k_M(r)}\|_{L^2(\Pi_N)}\Big\|_{L^{2p}(\Omega)}\\
   &\times \Big\|\Big\|\int_0^1 g_h^\prime(\lambda(R_{u,k_M(r)}+\Theta_N O_{k_M(r)})+(1-\lambda)(R_{\xi,k_M(r)}+\Theta_N O_{k_M(r)}))\,d\lambda\Big\|_{L^\infty(\Pi_N)}\Big\|_{L^{2p}(\Omega)}\, dr\\
   &\lesssim \int_\xi^v \min\{(r-u)^{1-\varepsilon},(r-k_M(r))^{1-\varepsilon}\}\, dr\\
   &\lesssim (v-u)^{1+\varepsilon/2} M^{-1+2\varepsilon},
\end{align*}
where we used in the last line that $\min\{a,b\}\leqslant a^{1-\delta} b^\delta$ for $\delta\in [0,1]$.
Next, we estimate the term corresponding to $J_1$; we separate the cases $|v-\xi|\leqslant 3M^{-1}$ and $|v-\xi|>3M^{-1}$.
We start by assuming the latter.
We take $t^\prime=k_M(\xi)+3M^{-1}$ similar to previously in the proof and similarly it will be sufficient to treat an integral from $t^\prime$ to $v$.
Using that $R_{u,r}-R_{\xi,r}$ is $\mathcal{F}_\xi$-measurable and \cref{lem:semigroup1} together with \cref{lem:productestnegative}, we have that
\begin{align}\label{eq:ghdiff}
    \int_{t^\prime}^v& (t-r)^{-1/4+\eps/2}\Big\|\|\E_\xi J_1\|_{H^{-1/2+\eps}(\Pi_N)}\Big\|_{L^p(\Omega)}\, dr\nonumber\\
    &\lesssim (t-v)^{-1/4+\eps/2}\int_{t^\prime}^v \|\|R_{u,r}-R_{\xi,r}\|_{\calC^{1/2-\eps/2}(\Pi_N)}\|_{L^{2p}(\Omega)} \nonumber\\
    & \times \int_0^1 \Big\|\Big\|(P^\R_{\Qx{\xi}{r}}g_h^\prime)(\lambda(R_{u,r}+\Theta_NP_{r-\xi}O_\xi+(1-\lambda)\cK_{\xi,r})\\
    & -(P_{\Qx{\xi}{k_M(r)}}^\R g_h^\prime)(\lambda(R_{u,k_M(r)}+\Theta_N P_{k_M(r)-\xi}O_\xi)+(1-\lambda)\cK_{\xi,k_M(r)})\, \Big\|_{H^{-1/2+\eps}(\Pi_N)}\Big\|_{L^{2p}(\Omega)}\,d\lambda\, dr.\nonumber
\end{align}
We rewrite the difference in \eqref{eq:ghdiff}, calling the respective arguments of the functions $L^{\lambda}_\cdot$ (surpressing the dependence on $\xi$ and $u$),  as
\begin{align}\label{eq:smugglingsemigroup}
    \paren[\Big]{\int_0^1 (P^\R_{\Qx{\xi}{r}} g_h^\prime)^\prime (\tilde{\lambda}L^\lambda_r+(1-\tilde{\lambda})L^\lambda_{k_M(r)})\,d\tilde{\lambda} }\,(L^\lambda_r-L^\lambda_{\kappa_M(r)})+((P^\R_{\Qx{\xi}{r}}-P^\R_{\Qx{\xi}{k_M(r)}})g_h^\prime)(L^\lambda_{k_M(r)}).
\end{align}
By \eqref{eq:cK2}, for any $q\geqslant 1$,
\begin{align} \label{eq:differenceR}
\Big\|\|\cK_{\xi,r}-\cK_{\xi,k_M(r)}\|_{H^{-1/2+\eps}(\Pi_N)}\Big\|_{L^q(\Omega)}&\lesssim [M^{-1+2\eps}+N^{-2+3\eps}](k_M(r)-\xi)^{-1/2+\eps}+(r-k_M(r))^{1-\eps}\nonumber\\
&\lesssim [M^{-1+2\eps}+N^{-2+3\eps}](k_M(r)-\xi)^{-1/2+\eps}
\end{align}
and after going through a similar argument for the other terms, we get
\begin{align*}
    \big\|\|L^\lambda_r-L^\lambda_{k_M(r)}\|_{H^{-1/2+\eps}(\Pi_N)}\big\|_{L^q(\Omega)}\lesssim [M^{-1+2\eps}+N^{-2+3\eps}](k_M(r)-\xi)^{-1/2+\eps}+(r-k_M(r))^{1-\eps}
\end{align*}
and the other factor in the first summand of \eqref{eq:smugglingsemigroup} can be controlled in $\|\cdot\|_{\calC^{1/2-\eps}(\Pi_N)}$ as for $P_1$ previously in the proof (note that $(P^{\R}_{Q(\xi,r)}g_h')'\in C^1_b$). For the second summand in \eqref{eq:smugglingsemigroup} we proceed as in the lines of inequalities \eqref{eq:twosemigroup}.
Using this and the regularity of the ``increments´´ of $R$ and that $k_M(r)-\xi\geqslant (r-\xi)/2$, we can estimate 
\begin{align*}
 \MoveEqLeft
(t-v)^{-1/4+\eps/2}\int_{t^\prime}^v\Big\|\|\E_\xi J_1\|_{H^{-1/2+\eps}(\Pi_N)}\Big\|_{L^p(\Omega)}\, dr
\\&\lesssim   (t-v)^{-1/4+\eps/2}\int_{t^\prime}^v (r-u)^{1-\varepsilon}[M^{-1+2\eps}+N^{-2+3\eps}] (k_M(r)-\xi)^{-1/2+\eps}\, dr\\
&\lesssim  (t-v)^{-1/4+\eps/2}(v-u)^{3/2+\eps}[M^{-1+2\eps}+N^{-2+3\eps}]
\end{align*}
It remains to treat the case $|v-\xi|\leqslant 3M^{-1}$. As we can bound the second factor in \eqref{eq:ghdiff} (with $\xi$ instead of $t^\prime$), we directly infer from \eqref{eq:ghdiff} and \eqref{eq:aprioriR} that
\begin{align*}
\int_{\xi}^v\Big\|\|\E_\xi J_1\|_{L^2(\Pi_N)}\Big\|_{L^p(\Omega)}\, dr&\lesssim  \int_{\xi}^v (r-u)^{1-\varepsilon}\, dr\\
&\lesssim  (v-u)^{1+\epsilon} M^{-1+2\eps},
\end{align*}
where we used that $|v-\xi|\lesssim M^{-1}$ in the last inequality. Together this yield \eqref{eq:sewingassump1} for $\delta_1=\frac{1}{4}-\frac{\eps}{2}, \epsilon_1=\eps, \Gamma_1= C(T,\eps,K,m,\mathcal{M})(M^{-1+2\epsilon}+N^{-2+3\eps})$.
It remains to identify for $v\in[s,t]$
\begin{align*}
\mathcal{A}_v= \int_s^v {P}_{t-r}^N \Big[g_h( X_r^{M,N}+\Theta_N O_r)-g_h(X_{k_M(r)}^{M,N}+\Theta_NO_{k_M(r)})\Big]\, dr. 
\end{align*}
 as the unique one-parameter process for which \eqref{eq:SSL-conc1} and \eqref{eq:SSL-conc2} is fulfilled. The argument for this was already carried out in the proof of \cite[Proposition 6.5]{DjGK} and can be applied similarly here.
\end{proof}

\begin{proof}[Proof of \cref{thm:main2}]
The result follows by combining the decomposition \eqref{BIG-decomposition} with \cref{lem:easyterm,lem:ztildez,lem:zhatztilde,lem:vhatz,lem:6.3,prop:temporal} and Assumption~\ref{ass:tildeO} \ref{en:b}.
\end{proof}

\section{A concrete choice of the input noise}\label{section:input-noise}

In this section we give an example of a process $\tilde O$ that verifies Assumption \ref{ass:tildeO} and that is furthermore implementable. The construction follows closely \cite[Section~2.3]{Davie-Gaines}.
Throughout the section we assume that $N$ is even, so that $J_N=\{-N/2+1,\ldots,0,\ldots,N/2\}$ and $$J_{N-1}=\{-N/2+1,\ldots,0,\ldots,N/2-1\}=\{j\in\Z\mid\abs{j}\leq N/2-1\}.$$ In fact a common choice is $N=2^K$ with various $K\in\N$, so that the different spatial grids are nested and thus the comparison of the different approximations is straightforward.

\subsection{Definition of $\tilde O$ and verifying \cref{ass:tildeO}}\label{section:tildeO}

In order to motivate our choice, we first write the OU process as
\begin{equ}
    O_t=\cQ_{N-1}^- O_t+(1-\cQ_{N-1}^-)O_t,
\end{equ}
where $\cQ_{N-1}^-$ is the orthogonal projection in $L^2(\T)$ to the subspace $L^2_{N-1}=\{e_j:\,j\in J_{N-1}\}$.
The choice of projecting to $L^2_{N-1}(\T)$ instead of $L^2_N(\T)$ is one of convenience, to keep the index set symmetric.
The projected part will be unchanged, as its restriction to the grid is possible to simulate exactly, while the remainder $(1-\cQ_{N-1}^-)O_t$ is replaced by the process
\begin{equ}\label{eq:Znoise}
    Z_t^{M,N}=\bL_M(1-\cQ_{N-1}^-)\int_{(t-h)\vee 0}^t\int_\T p_{t-s}(\cdot-y)\xi(dy,ds),
\end{equ}
where $\bL_M$ is the Fourier multiplier defined as
\begin{equ}
    \widehat{\bL_M f}(j):=\frac{1}{\sqrt{1-e^{-8\pi^2j^2h}}}\hat f(j)
\end{equ}
for $j\neq 0$ and $f\in L^2(\T)$ with $\hat{f}(0)=0$.

First, we collect properties of the random variables $\hat{Z}^{M,N}_{kh}(j)$. In particular, we see that the covariance structure is explicit.

Below we use the convention that a $\C$-valued random variable $Z$ is complex Gaussian, if its real and imaginary parts are independent and identically distributed $\R$-valued Gaussians. In this case the distribution of $Z$ is characterised by its mean and variance $\mathbb{E}|Z|^2$. Note that $\mathbb{E}Z^2=0$ and two centered complex Gaussians $Z,Z'$ are independent if and only if $\E ZZ'=\E Z\bar{Z}'=0$.
\begin{lemma}\label{lem:triv-Z}
The random variables $\hat Z^{M,N}_{kh}(j)$ for $\abs{j}>\frac{N}{2}-1$, $k=1,2,\ldots,M$, are complex Gaussians with mean $0$ and variance $(8\pi^2 j^2)^{-1}$. Further, $\hat Z^{M,N}_{kh}(j),\hat Z^{M,N}_{k'h}(j')$ for $\abs{j},\abs{j'}>\frac{N}{2}-1$ are independent if $k\neq k'$ or if $k=k'$ and $j'\notin\{j,-j\}$, and moreover, if $k=k'$ and $j'=-j$, then $\hat Z^{M,N}_{kh}(j)=\overline{\hat Z^{M,N}_{k'h}(j')}$.
\end{lemma}
\begin{proof}
    It is clear that $\hat{Z}^{M,N}_{kh}(j)$ for $k=1,\dots, M$, $\abs{j}>\frac{N}{2}-1$  are complex Gaussians with vanishing mean and variance that we compute using the isometry property of the white noise integral,
    \begin{align*}
        \E \abs{\hat{Z}^{M,N}_{kh}(j)}^2
        &=\E \abs[\bigg]{\int_{(k-1)h}^{kh}\int_{\T}\frac{\exp(-4\pi^2 j^2 (kh-r))}{\sqrt{1-\exp(-8\pi^2 j^2 h)}}\, e_{-j}(y)\,\xi(dy,dr)}^2
        \\&= \frac{1}{1-\exp(-8\pi^2 j^2 h)}\int_{(k-1)h}^{kh} \exp(-8\pi^2 j^2 (kh-r))\,dr
        \\&=\frac{1}{1-\exp(-8\pi^2 j^2 h)}\int_{0}^{h} \exp(-8\pi^2 j^2 r)\,dr = (8\pi^2 j^2)^{-1}.
    \end{align*}
    Independence of $\hat{Z}^{M,N}_{kh}(j)$ and $\hat{Z}^{M,N}_{k'h}(j)$ for $k\neq k'$ follows due to the integral against the white noise over non-overlapping intervals. The independence of $\hat{Z}^{M,N}_{kh}(j)$ and $\hat{Z}^{M,N}_{kh}(j')$ with $j'\notin \{j,-j\}$ holds, since $(\hat{Z}^{M,N}_{kh}(\ell))_{\ell\in\Z}$ are jointly complex Gaussian and due to the isometry property 
    \begin{align*}
    \MoveEqLeft
        \E \paren[\bigg]{\int_{(k-1)h}^{kh} e^{-4\pi^2 j^2 (kh-r)} \int_{\T} e_{j} (y) \xi(dy,dr)}\paren[\bigg]{\int_{(k-1)h}^{kh} e^{-4\pi^2 j^2 (kh-r)} \int_{\T} e_{j'} (y) \xi(dy,dr)} 
        \\&= \langle e_j, e_{-j'}\rangle_{L^2(\T,\C)}\int_{(k-1)h}^{kh}e^{-8\pi^2 j^2 (kh-r)} \, dr=0
    \end{align*}
    if $j\neq -j'$, and similarly, 
    $\E\big(\hat{Z}^{M,N}_{kh}(j)\overline{\hat{Z}}^{M,N}_{kh}(j')\big)=0$ if $j\neq j'$.
    The conjugate relation is also clear from the conjugate relation of the Fourier basis.
\end{proof}
Our choice for $\tilde O$ is then the following: 
\begin{equ}
    \tilde O=\cQ_{N-1}^-O+Z^{M,N}.
\end{equ}
The exact simulation of $\Theta_N\tilde O$ is
detailed in \cref{sec:implement} below. Next, we verify the required regularity bounds on $\tilde{O}$.
\begin{lemma}\label{lem:Z-est}
Let $c>0$ and assume that $M^{-1}\geq c N^{-2}$ for some. Let $p\geq 1$, $\lambda\in [0,1)$, $\epsilon\in (0,1/2)$. Then there exists $C=C(p,\lambda,\eps,c)$ such that
    \begin{align}
       \E\norm{\cQ_{N-1}^- O}_{C_T^{\frac{\lambda}{2}}\calC^{\frac{1}{2}-\lambda-\epsilon}(\T)}^p&\leq C,\label{eq:aa1}\\
       \E \norm{Z^{M,N}}_{C_T^{\lambda/2}\calC^{1/2-\lambda-\epsilon}(\T)}^p&\leqslant C,\label{eq:aa2}\\
       \E\norm{\tilde O}_{C_T^{\frac{\lambda}{2}}\calC^{\frac{1}{2}-\lambda-\epsilon}(\T)}^p&\leq C.\label{eq:aa3}
    \end{align}
\end{lemma}
\begin{proof}
Recall that the corresponding estimate for $O$ is proven by estimating the variance of each Fourier coefficient, see e.g  \cite[Proposition 3.1]{DjGK}. Then it follows that the same estimate holds for any process $m(\nabla)O$, where $m(\nabla)$ is a Fourier multiplier, that is $m(\nabla) f := \F^{-1} (m\hat{f})$, with a uniformly bounded symbol $m$. 
Set now
$$m(j)=\frac{1}{\sqrt{1-e^{-8\pi^2 j^2 h}}}\mathbf{1}_{\abs{j}>\frac{N}{2}-1}.
$$
It follows from the assumption $M^{-1}\geq c N^{-2}$ that $|m(j)|\lesssim 1$ uniformly in $j$.
Moreover, $Z_t=m(\nabla)O_t-m(\nabla)P_hO_{(t-h)\vee0}$. The bound \eqref{eq:aa2} then readily follows. The case \eqref{eq:aa1} is even easier, since the multiplier symbol is just an indicator in this case. Finally, \eqref{eq:aa3} follows from \eqref{eq:aa1} and \eqref{eq:aa2}.
\end{proof}

We are left to prove that $O$ and $\tilde{O}$ are close with high rate, which is the content of the next lemma.

\begin{lemma}\label{lem:tilde-O-error}
    Assume that for $c,\varepsilon>0$, $M^{-1}\geq c N^{-2+\varepsilon}$. Let $p\geq 2$. Then there exist constants $C=C(T,p),c'=c'(T,p,c)>0$ such that
    \begin{equ}
   \max_{0\leqslant k \leqslant M} \big(\E\|\Theta_N(O_{t_k}-\tilde O_{t_{k}})\|_{L^2(\Pi_N)}^p\big)^{1/p}\leq C e^{-c'N^{\eps}}.
    \end{equ}
\end{lemma}

\begin{proof}
    First note that by isometry property of $\Psi_N$, we have that
\begin{equ}
    \|\Theta_N(O_{t_k}-\tilde O_{t_k})\|_{L^2(\Pi_N)}=\|\Psi_N\Theta_N(O_{t_k}-\tilde O_{t_k})\|_{L^2(\T)}.
\end{equ}
Let $Y\coloneqq O-\tilde{O}^{M,N}$ and $\lambda_j=4\pi^2 j^2$. Then, for $\abs{j} >\frac{N}{2}-1$ and $t\geqslant h>0$,
\begin{align*}
    \hat{Y}_t(j)=\int_0^{t-h}e^{-\lambda_j(t-r)}\int_{\T}e^{-2\pi i j y} \xi(dy,dr)+\Big(1-\frac{1}{\sqrt{1-e^{-2\lambda_j h}}}\Big) \int_{t-h}^t e^{-\lambda_j (t-r)}\int_{\T} e^{-2\pi i j y} \xi(dy,dr)
\end{align*}
and $\hat{Y}_{t}(j)=0$ for $\abs{j}\leq \frac{N}{2}-1$.
For $t=t_0=0$ the claim is trivial and thus we assume wlog that $t=t_{k}=k h \geq h$ for $k\geq 1$.
Hence we find,
\begin{align*}
    \E[|\hat{Y}_{t}(j)|^2]&=\int_0^{t-h} e^{-2\lambda_j(t-r)}\, dr + \Big(1-\frac{1}{\sqrt{1-e^{-2\lambda_j h}}}\Big)^2 \int_{t-h}^t e^{-2\lambda_j(t-r)}\, dr\\
    &=\frac{e^{-2\lambda_j h}-e^{-2\lambda_j t}}{2\lambda_j}+\Big(1-\frac{1}{\sqrt{1-e^{-2\lambda_j h}}}\Big)^2 \frac{1-e^{-2\lambda_jh}}{2\lambda_j}\\
    &\lesssim \frac{e^{-2\lambda_j h}}{2\lambda_j}+\Big(1-\frac{1}{\sqrt{1-e^{-2\lambda_j h}}}\Big)^2 \frac{1-e^{-2\lambda_jh}}{2\lambda_j}\\
    &\lesssim j^{-2}\Big[ e^{-2\lambda_j h}+\Big(1-\frac{1}{\sqrt{1-e^{-2\lambda_j h}}}\Big)^2 (1-e^{-2\lambda_jh})\Big]\\
    &\lesssim j^{-2} e^{-2\lambda_j h},
\end{align*}
where we used that $x+(1-\frac{1}{\sqrt{1-x}})^2(1-x)\lesssim x$ for $x \in (0,1)$ in the last line of the above.
Next, we use similar arguments as in the proof of Lemma~\ref{lem:operator}. We have that
\begin{align*}
    \Psi_N\Theta_N Y_{t}=\sum_{m\in\mathbb{Z}}\sum_{j\in J_N}\hat{Y}_{t}(j+mN)e_j
\end{align*}
and thus by orthogonality and independence
\begin{align*}
    \E\norm{\Psi_N\Theta_N Y_{t}}_{L^2(\T)}^2 = \E\sum_{j\in J_N} \paren[\Big]{\sum_{m\in\Z}\hat{Y}_{t}(j+mN)}^2&=\sum_{j\in J_N} \sum_{m,\tilde{m}}\E\hat{Y}_{t}(j+mN)\hat{Y}_{t}(j+\tilde{m}N)
    \\&=\sum_{j\in J_N} \sum_{m\in\Z}\E\abs{\hat{Y}_{t}(j+mN)}^2.
\end{align*}
Note that the sum is only over $j,m$ with $\abs{j+mN}> \frac{N}{2}-1\geq \frac{N}{4}$ for $N\geq 4$, since its summand vanishes otherwise. In particular, if $j=0$ it follows that $m\neq 0$. In this case, we further have that $e^{-8\pi^2 (j+mN)^2 h}\leq e^{-c' N^{\eps}}$ for a constant $c'=c'(T,c)$ as $h=TM^{-1}\geq c TN^{-2+\eps}$. Thus we find, combining with the above estimate,
\begin{align*}
    \E\norm{\Psi_N\Theta_N Y_{t}}_{L^2(\T)}^2
    &\lesssim \sum_{j\in J_N} \sum_{m\in\Z} (j+mN)^{-2}e^{-8\pi^2 (j+mN)^2 h}
   \\ &= \sum_{j\in J_N\setminus\{0\}} \sum_{m\in\Z} (j+mN)^{-2}e^{-8\pi^2 (j+mN)^2 h} + \sum_{m\in\Z\setminus\{0\}} (mN)^{-2}e^{-8\pi^2 (j+mN)^2 h}
    \\&\lesssim e^{-c'N^{\eps}}  N^{-2} \paren[\big]{N\sum_{m\in\Z} (1+m)^{-2}+\sum_{m\in\Z\setminus\{0\}} m^{-2}}
    \\&\lesssim N^{-1} e^{-c'N^{\eps}}\leq e^{-c'N^{\eps}}.
\end{align*}
Since Gaussian moments are comparable, the claim follows.
\end{proof}

Putting \cref{lem:triv-Z} and \cref{lem:tilde-O-error} together we obtain that our choice for $\tilde{O}$ satisfies \cref{ass:tildeO}.

\subsection{Description of the pseudo-algorithm}\label{sec:implement}

For an $N$-dimensional $\C$-valued vector indexed by $\Pi_N=\{0,1/N,\ldots,(N-1)/N\}\to\C$ we denote by $\mathtt{FT}(g)$ its discrete Fourier transform, which we view as an $N$-dimensional $\C$-valued vector $\hat g$ indexed by $J_N$. Similarly, $\mathtt{FT}^{-1}(\hat g)$ is the inverse Fourier transform.
Note that the computational effort to compute $\mathtt{FT} g$ from $g$ (or $\mathtt{FT} ^{-1}\hat g$ from $\hat g$) is $O(N\log N)$. 
Recall that since we are working with a complex Fourier basis, $i$ denotes $\sqrt{-1}$ and is never used as an index.
By $\mathtt{random}$ we denote a function that generates standard normals that are independent at each call of the function.
For simplicity we also use the notation $\mathtt{Crandom}$ for $\frac{1}{\sqrt{2}}(\mathtt{random}+i\,\mathtt{random)}$, i.e. for the generation of a complex Gaussian with variance $1$.

\textbf{Step 1: Generating the noise from \cref{section:tildeO}}

\textit{Step 1a: The projected component in Fourier}.
First, we would like to simulate the random variables $\mathcal{F}_N\Theta_N\cQ_{N-1}^-O_{kh}$ for $k=0,\dots,M$, where $\mathcal{F}_N$ denotes the discrete Fourier transform $\mathcal{F}_Nf(j)=\langle f, e_j^N\rangle_{L^2(\Pi_N,\mathbb{C})}$. Since each Fourier mode of $O$ is a $1$-dimensional Ornstein-Uhlenbeck process, the recursive exact generation is straightforward

Let $\mathtt{O}^{\mathtt{FT,proj}}_0(j)=0$ for all $j\in J_{N-1}$. For $k=1,\ldots,M$ we proceed inductively: for $j=1,\ldots,N/2-1$ set
\begin{equs}
\mathtt{O}^{\mathtt{FT,proj}}_k(j)&=e^{-4\pi^2j^2h}\mathtt{O}^{\mathtt{FT,proj}}_{k-1}(j)+\Big(\int_0^he^{-8\pi^2 j^2 s}\,ds\Big)^{1/2}
\mathtt{Crandom}
\\
&=e^{-4\pi^2j^2h}\mathtt{O}^{\mathtt{FT,proj}}_{k-1}(j)+\frac{\sqrt{1-e^{-8\pi^2j^2h}}}{\sqrt{8}\pi j}
\mathtt{Crandom},
\end{equs}
and for $j=-N/2+1,\ldots,-1$ set $\mathtt{O}^{\mathtt{FT,proj}}_k(j)=\overline{\mathtt{O}}^{\mathtt{FT,proj}}_k(-j)$.
Note that the $0$-th 
Fourier mode is always purely real and given by the induction
\begin{equ}
\mathtt{O}^{\mathtt{FT,proj}}_k(0)=\mathtt{O}^{\mathtt{FT,proj}}_{k-1}(0)+\sqrt{h}\,\,\mathtt{random}.
\end{equ}
The $j=N/2$ component simply vanishes due to the projection $\mathcal{Q}_{N-1}$.

\textit{Step 1b: The remainder in Fourier.}
Next, we would like to simulate $\mathcal{F}_N\Theta_NZ^{M,N}_{kh}$ for $k=0,\ldots,M$. 
First we set $\mathtt{Z}^{\mathtt{FT}}_0(j)=0$ for all $j\in J_N$, since by our definition above $Z^{M,N}_0=0$.
Recall the identity for $j\in J_{N-1}$ (the case $j=N/2$ will be treated separately below), 
\begin{equ}\label{eq:help-eq-in-implement}
    \cF_N\big(\Theta_NZ^{M,N}_{kh}\big)(j)=\sum_{m\in\Z\setminus\{0\}}\cF(Z^{M,N}_{kh})(j+mN).
\end{equ}
As before, we separate different cases for $j$. For $j=1,\ldots,N/2-1$, notice that by \cref{lem:triv-Z}, the random variables in the sum are independent for all $k=1,2,\ldots,M$,  $m\in\N\setminus\{0\}$, and the variance can be computed in a closed form as 
\begin{equ}
    \E\Big|\cF_N\big(\Theta_NZ^{M,N}_{kh}\big)(j)\Big|^2=-\frac{1}{8\pi^2j^2}+\sum_{m\in\Z}\frac{1}{8\pi^2(j+mN)^2}
    =\frac{1}{8N^2\sin^2(\pi j/N)}-\frac{1}{8\pi^2j^2}=:\mathtt{B}_N^2(j).
\end{equ}
Therefore we set for each $k=1,\ldots,M$ and $j=1,\ldots,N/2-1$
\begin{equ}
    \mathtt{Z}^{\mathtt{FT}}_k(j)=\mathtt{B}_N(j)\,\mathtt{Crandom}
\end{equ}
and for $j=-N/2+1,\ldots,-1$ set $\mathtt{Z}^{\mathtt{FT}}_k(j)=\overline{\mathtt{Z}}^{\mathtt{FT}}_k(-j)$. In the case $j=0$ the sum in \eqref{eq:help-eq-in-implement} is made up of conjugate pairs so the imaginary part cancels, while the real part doubles, leaving a real Gaussian with variance 
\begin{equ}
    \E\Big|\cF_N\big(\Theta_NZ^{M,N}_{kh}\big)(0)\Big|^2=\sum_{m\in\Z\setminus\{0\}}\frac{1}{8\pi^2(mN)^2}=\frac{1}{24N^2}=:\mathtt{B}_N^2(0).
\end{equ}
Therefore we set for each $k=1,\ldots,M$
\begin{equ}
\mathtt{Z}^{\mathtt{FT}}_k(0)=\mathtt{B}_N(0)\,\mathtt{random}.
\end{equ}
Finally, the situation for $j=N/2$ is similar to $j=0$, but now the equality \eqref{eq:help-eq-in-implement} holds only with including the $m=0$ term. We thus obtain a real-valued Gaussian with variance
\begin{equ}
    \E\Big|\cF_N\big(\Theta_NZ^{M,N}_{kh}\big)(N/2)\Big|^2=\mathtt{B}_{N/2}^2(0)-\mathtt{B}_N^2(0)=\frac{1}{8N^2}=:\mathtt{B}_N^2(N/2)
\end{equ}
Therefore we set for each $k=1,\ldots,M$
\begin{equ}
\mathtt{Z}^{\mathtt{FT}}_k(N/2)=\mathtt{B}_N(N/2)\,\mathtt{random}.
\end{equ}

\textit{Step 1c: Putting the noise parts together.}
Summing the two components and inverting Fourier yields the input noise of the scheme:
\begin{equ}
    \mathtt{O}_k=\mathtt{FT}^{-1}\big(\mathtt{O}_k^{\mathtt{FT,proj}}+\mathtt{Z}_k^{\mathtt{FT}}\Big).
\end{equ}

\textbf{Step 2: The schemes}

 Once $\mathtt{O}$ is given, the schemes are easy to implement. Define the operators appearing in \eqref{eq:V-recursion} and \eqref{eq:XMN} on the Fourier side as
Let
\begin{equ}
    \mathtt{P^{FT,1}}(j)=e^{-4\pi^2j^2 h},\qquad\mathtt{Q^{FT,2}}(j)=(4\pi^2j^2)^{-1}\big(\mathtt{SG^{FT,1}}(j)-1\big)
\end{equ}
with the convention $\mathtt{Q^{FT,2}}(0)=h$.
In real space, we can write for $i=1,2$
\begin{equ}
    \mathtt{P} g=\mathtt{FT}^{-1}\big(j\mapsto \mathtt{FT}g(j)\mathtt{P^{FT}}(j)\big),\qquad
    \mathtt{Q} g=\mathtt{FT}^{-1}\big(j\mapsto \mathtt{FT}g(j)\mathtt{Q^{FT}}(j)\big)
\end{equ}

\textit{Step 2a: The case of bounded $f$}.
With the above notation, \eqref{eq:V-recursion} can be written as the inductive procedure with $\mathtt{V}_0=u_0$ and, for $k=1,\dots M$,
\begin{equ}
    \mathtt{V}_k=\mathtt{P} \mathtt{V}_{k-1}+\mathtt{Q} \big(f(\mathtt{V}_{k-1}+\mathtt{O}_{k-1})\big).
\end{equ}
And finally, the full scheme is 
\begin{equ}
\mathtt{U}_k=\mathtt{V}_k+\mathtt{O}_k.
\end{equ}
\textit{Step 2b: The case of polynomially growing $f$}.
In the second case we further recall the notation \eqref{eq:gh} (see also \eqref{eq:Phi}. Then \eqref{eq:XMN} can be written as the inductive procedure with $\mathtt{X}_0=u_0$ and, for $k=1,\dots M$,
\begin{align*}
     \mathtt{X}_k=\mathtt{P} \mathtt{X}_{k-1}+h\,\mathtt{P} \big(g_h(\mathtt{X}_{k-1}+\mathtt{O}_{k-1})\big).
\end{align*}
Again, the full scheme is given by 
\begin{equ}
\mathtt{U}_k=\mathtt{X}_k+\mathtt{O}_k.
\end{equ}
The total computational effort in both cases is $O(MN\log N)$.

\section{Numerical results}\label{sec:numerics}

\paragraph{The setup} We report on the numerical results on implementing the doubly accelerated exponential Euler scheme (DAEE) for a concrete equation, including comparisons to the performance of other schemes: the implicit finite difference (FD) scheme of \cite{Gy} and the accelerated exponential Euler scheme (AEE) of \cite{Jentzen-Kloden}. 
The numerical evidence aligns with the theoretical results and confirms a substantial speedup of the convergence.

In the numerical example we take $T=1$, $f(r)=\sin(r)$ and $u_0(x)\equiv 1$, i.e. the equation reads as
\begin{equation}\label{eq:num-spde}
  \partial_t u = \Delta u + \sin(u) + \xi, \qquad u(0,x)\equiv 1.
\end{equation}

For all methods,
 $M$ denotes the number of time steps, $h=1/M$, and
$N$ (in the case of AEE, $N-1$) denotes the number of spatial degrees of freedom.
The output of both the FD and the DAEE is a function on $I_M\times \Pi_N$.
The output of AEE would be by default a function on $I_M\times \T$; at each timestep the approximation is a linear combination of Fourier modes with frequency $|k|\leq N/2-1$. Evaluating this trigonometric polynomial on $\Pi_N$, the outputs of all three methods
are functions on $I_M\times \Pi_N$.

\paragraph{Balanced refinement, noise consintency, and error measure}
In each method the time step is coupled to the spatial resolution so that both temporal and spatial contributions to
the proven bound are of the same order,
\[
  \text{FD: } M=N^2,\qquad \text{AEE: } M=N^{1/2},\qquad \text{DAEE: } M=N^{3/2},
\]
so that the total error is expected to decay like $N^{-1/2}$, $N^{-1/2}$ and $N^{-3/2}$,
respectively, owing to the results of \cite{Gy}, \cite{DjGK}, and Theorem \ref{thm:main1}, respectively. For DAEE this also guarantees the condition $h\gtrsim N^{-2+\varepsilon}$ imposed in Section \ref{section:input-noise}. All errors are measured in the same norm,
\[
  e_N=\Big(\mathbb{E}\,\big\|U_N(1)-\Theta_N U_{\mathrm{ref}}(1)\big\|^2_{L^2(\Pi_N)}\Big)^{1/2}, 
\]
where $U_{\mathrm{ref}}$ is the reference solution given by the same method on the finest level
(Table~\ref{tab:num}).
All levels of a method are computed simultaneously from one
noise realisation:
for FD, coarse space-time box integrals of $\xi$ are
sums of fine ones on nested dyadic grids; for AEE the exact Ornstein--Uhlenbeck
increments over a coarse step are obtained from the fine ones using the mild additivity of the Ornstein-Uhlenbeck process
and coarser spatial levels use a subset of the same modes. For DAEE, accounting for the contribution of the high modes is more subtle, we postpone the discussion to the next paragraph. Expectations are replaced by Monte Carlo averages over $S$ independent paths; we report
95\% confidence intervals for $e_N$ (central limit theorem) and for the fitted rates (least-squares
slope in log--log scale, nonparametric bootstrap over paths \cite{bootstrap}, which respects the correlation between
levels). As a reference-free check we also record the successive differences
$\|U_N(1)-\Theta_N U_{4N}(1)\|_{L^2(\Pi_N)}$ (for FD: $U_{2N}$).

\paragraph{Consistent noise for DAEE}
For DAEE the coupling of the noise at different levels has to respect the construction of $\tilde O$, so that the high-frequency part of each is a functional of the same noise $\xi$.
All levels are therefore built from one realisation on the reference resolution: the modes
$|l|<N_{\mathrm{ref}}/2$ of the Ornstein--Uhlenbeck process are simulated exactly on the
reference step $h_{\mathrm{ref}}$,
and the contribution of the high-frequency part is accounted for as in \cref{section:input-noise}.
 For a level with $N$ grid points
and step $h=m\,h_{\mathrm{ref}}$, the noise input at a time $t_k$ of its grid is obtained as follows.
The modes $|l|<N/2$ are obtained from the modes of $\tilde O_{N_{\mathrm{ref}}}$ directly. 
The contribution of high frequencies is separated in two parts.
For $N/2\le|l|<N_{\mathrm{ref}}/2$ we Fourier multiply the corresponding modes (which are exact) of $\tilde O_{N_{\mathrm{ref}}}$.
Finally, for modes $|l|\ge N_{\mathrm{ref}}/2$ the same high-frequency contribution is taken for $\tilde{O}_N$ as for $\tilde O_{N_{\mathrm{ref}}}$.
This differs\footnote{One could also adapt \cref{section:input-noise} to let the time integral horizon in \eqref{eq:Znoise} depend on the frequency, for the sake of brevity we refrain from this.} from \cref{section:input-noise} up to a computationally invisible amount: therein in \eqref{eq:Znoise} one would take these
modes from the integral over the last step of length $h$ rather than $h_{\mathrm{ref}}$. Since they
relax on the time scale $\lambda_l^{-1}\le\lambda_{N_{\mathrm{ref}}/2}^{-1}\ll h_{\mathrm{ref}}$,
the two differ by a factor of order $e^{-\lambda_{N_{\mathrm{ref}}/2}h_{\mathrm{ref}}}\approx e^{-316}$,
which is computationally irrelevant.

\paragraph{Cost}
We measure the cost by the number $MN$ of space-time updates; since every method costs up to
$\mathcal{O}(N\log N)$ per step (using FFTs), this is proportional to the operation count up to
logarithmic factors. The balanced choices give $MN=N^3,N^{3/2},N^{5/2}$, hence the predicted
rates $\tfrac16,\tfrac13,\tfrac35$ with respect to cost; in terms of accuracy $\delta$ the cost grows like
$\delta^{-6}$, $\delta^{-3}$ and $\delta^{-5/3}$, respectively. In addition we report the wall-clock time per
sample path, measured by running each level on its own (with noise generated at its own
resolution) in a single thread, vectorised over 32 paths, median of three repetitions
(Python/NumPy/SciPy, one core).

\paragraph{Results}
Table~\ref{tab:num} and Figure~\ref{fig:num} summarise the results. For all three methods the
fitted rates of the total error with respect to $N$ agree with the proven ones, $0.496$ (FD),
$0.518$ (AEE) and $1.535$ (DAEE), and the 95\% confidence intervals contain $\tfrac12,\tfrac12$ and
$\tfrac32$, respectively. The reference-free successive differences give $0.448$, $0.488$ and $1.492$;
the somewhat smaller value for FD is due to the coarsest level, where the local rate is $0.41$,
while the local rates on the finer levels are $0.50$, $0.50$ and $0.57$. With respect to the cost $MN$
the observed rates are $0.165$, $0.345$ and $0.614$, in agreement with the theoretical values $\tfrac16,\tfrac13,\tfrac35$. The
differences between the methods are substantial: at the same cost $MN=2^{15}$, AEE ($N=1024$)
reaches an error of $1.3\cdot10^{-2}$ and DAEE ($N=64$) one of $7.2\cdot10^{-4}$, whereas FD needs
$MN=2^{18}$ ($N=64$) to reach $5.1\cdot10^{-2}$. The fitted wall-clock rates are higher ($0.208$, $0.456$, $0.795$), this is due to the fact that on coarse grids the
run time is dominated by the per-step overhead of the implementation, and accordingly the local wall-clock rates decrease towards the cost rates
as $N$ grows (for AEE from $1.16$ to $0.37$). As a consistency check that the three methods
approximate the same solution, $8000$ independent paths at moderate resolution give
$\mathbb{E}\big[N^{-1}\sum_{x\in\Pi_N}u(1,x)\big]\approx1.727\pm0.023$, $1.710\pm0.024$, $1.700\pm0.024$
and $\mathbb{E}\|u(1)\|_{L^2(\Pi_N)}^2\approx4.15\pm0.07$, $4.15\pm0.07$, $4.13\pm0.07$ for FD, AEE and DAEE,
respectively (95\% confidence intervals).

\begin{remark}
Although the observed rate of the DAEE matches the proven one, this does not constitute as evidence for the proven rate being sharp, as the chosen balancing $M=N^{3/2}$ and the sharpness of the temporal rate (\cite[Proposition~2.9]{DjGK}) makes observing higher spatial rate impossible. Additional experiments with $M\asymp N^{2}$ indicate a spatial error decay of order $N^{-2}$.
\end{remark}

\begin{table}[ht]
\centering\small
\input{results_table}
\caption{Balanced refinement: levels, reference, number of paths, and fitted rates of the total
error (95\% bootstrap confidence intervals) with respect to $N$ and to the cost $MN$.}
\label{tab:num}
\end{table}

\begin{figure}[h]
\centering
\includegraphics[width=\textwidth]{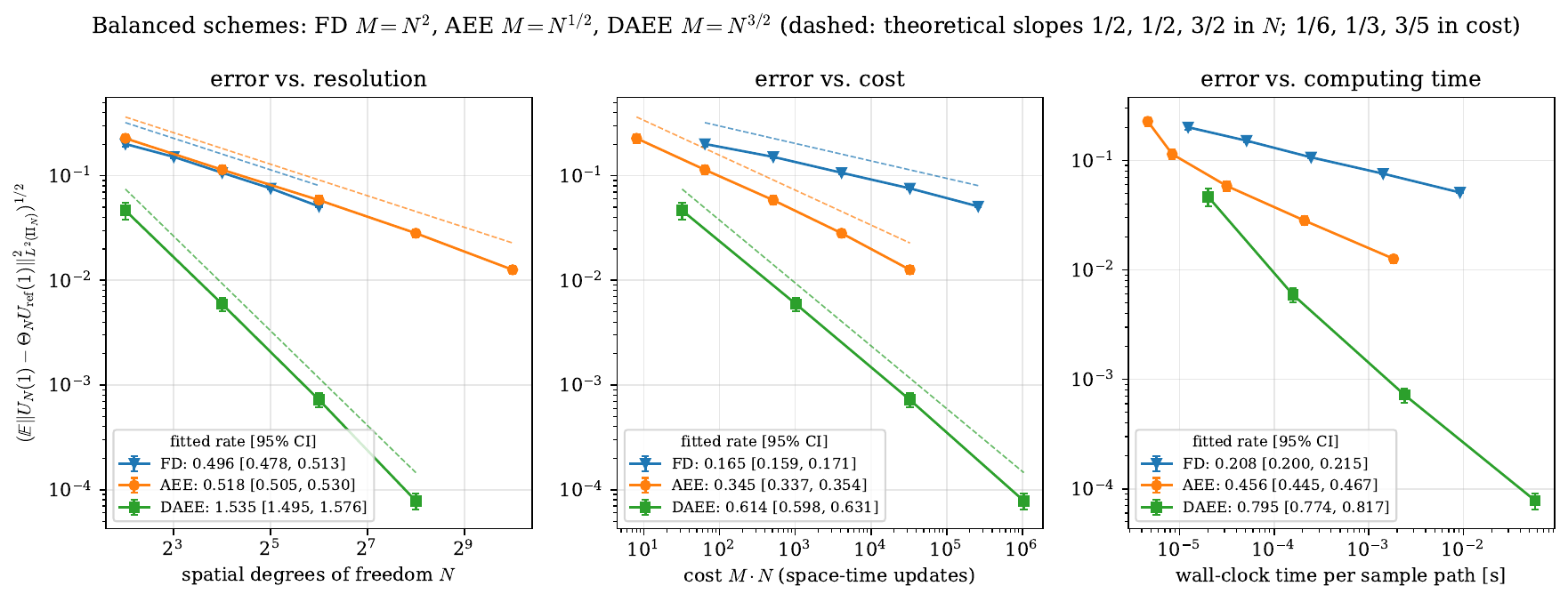}
\caption{Total strong error at $t=1$ in $L^2(\Pi_N)$ for the three balanced schemes, against the
resolution $N$ (left), the cost $MN$ (middle) and the wall-clock time per path (right). Error bars:
95\% confidence intervals; dashed: theoretical slopes.}
\label{fig:num}
\end{figure}

\begin{appendices}
\begin{section}{Appendix A}\label{appendix}

\begin{proof}[Proof of \cref{lem:semigroup-w}]
Let $\omega(z)=(1+z^2)^{-\beta/2}$ and $\omega_{-\beta}(z)=(1+z^2)^{\beta/2}$ for $z\in\R$.
First, we note that by an easy computation for $x,y\in\R$, it holds that $$\frac{\omega(x+y)}{\omega(y)}= \paren[\bigg]{\frac{1+y^2}{1+(x+y)^2}}^{\beta/2}\leq 2^{\beta/2} (1+x^2)^{\beta/2}= 2^{\beta/2}\omega_{-\beta}(x).$$ By definition for $v\in C^{0}_\omega$, $\norm{v}_{C^{0}_\omega}=\norm{vw}_{L^\infty}$. For $t=0$ the claims are trivial, so that we may assume $t>0$. Then using the bound on $\omega$, we find that with the weighted Young inequality from \cite[Theorem 2.1]{MW},
    \begin{align*}
        \norm{P^{\R}_t v}_{C^0_\omega}= \norm{(p_{t}\ast v) w}_{L^\infty}\leq \norm{p_{t}\omega_{-\beta}}_{L^1} \norm{vw}_{L^{\infty}}\leq C \norm{v}_{C^0_\omega},
    \end{align*}
    since $\norm{p_{t}\omega_{-\beta}}_{L^1}=\E(1+B_t^2)^{\beta/2}\leq c(\beta) (1+t^{\beta/2})\leq c(\beta)(1+T^{\beta/2})$ for $t\in (0,T]$ and a constant $c(\beta)>0$ that does not depend on $t$, where $B$ denotes a standard Brownian motion. Next, we note that, since $\partial_t P_{t}^{\R} v = P_{t}^\R \Delta v$,
    \begin{align*}
        (P_t^\R - \operatorname{Id}) v=\int_{0}^{t} P^{\R}_r\Delta v \, dr.
    \end{align*}
    Hence we have that, using the semigroup estimate above,
    \begin{align*}
        \norm{(P_t^\R - \operatorname{Id}) v}_{C^{0}_\omega}\lesssim t  \norm{\Delta v}_{C^{0}_\omega}\lesssim t \norm{v}_{C^2_\omega},
    \end{align*}
    where  the last inequality follows from the fact that $\abs{w'(z)},\abs{w''(z)}\lesssim  w(z)$.
\end{proof}

\begin{proof}[Proof of \cref{lem:semigroup1}]
    The $L^2$ bounds are trivial from the definition of $P_{t},P_{t}^{N},\tilde{P}_{t}^{N}$ since $e^{\lambda_{k} t}\leq 1$ for $\lambda_{k}\leq 0$ being equal to the eigenvalues for $\Delta$, $\Delta_N$ or $\tilde{\Delta}_N$ on $\mathbb{T}$, respectively $\Pi_N$, for $k\in\Z$, respectively $k\in J_N$. 
    The semigroup estimates \eqref{eq:heatkernelH}, \eqref{eq:heatLinfty} and \eqref{eq:heatalpha2} for $(P_{t})$ follow from \cite[Lemma A7]{GIP} and \cite[Proposition 5 and 6]{MW}, 
    whose estimates also apply for the Besov space $\calC^{\alpha}(\mathbb{T})$ on the torus $\T$, since for $\alpha\in\R$, $f\in\calC^{\alpha}(\mathbb{T})$ if and only if its periodic extension $f^{\R}\in\calC^{\alpha}(\R)$ (where $f^{\R}(\phi):= f(\sum_{n\in\Z} \phi(\cdot+n))$, $\phi\in \mathcal{S}(\R)$) and $P_{t}f=P_{t}^{\R} f^{\R}$ for the heat-semigorup $P_{t}^{\R}$ on $\R$.
    The estimate \eqref{eq:heatLinfty2} follows from \eqref{eq:heatalpha2} applied for $\delta'=\delta -\epsilon$ and $\alpha=\epsilon$, respectively $\delta' =\delta+\epsilon$ and $\alpha=-\epsilon$, for small $\epsilon>0$ together with (below $\Delta_j u :=\F^{-1}(\phi^j \hat{u})$ denotes the $j$-the Littlewood Paley block)
    \begin{align*}
        \norm{ P_{t}f-f}_{L^{\infty}(\T)}&\leqslant \sum_{j\geq 0} \norm{\Delta_j (P_{t}f-f)}_{L^{\infty}(\T)}\\&\leqslant \norm{f}_{\calC^{\delta}(\T)}\paren[\Big]{\sum_{j: 2^{-j}<t} 2^{-j\epsilon} t^{(\delta-\epsilon)/2} + \sum_{j: 2^{-j}\geq t} 2^{j\epsilon} t^{(\delta+\epsilon)/2}}
        \\&\lesssim \norm{f}_{\calC^{\delta}(\T)} t^{\delta/2}
    \end{align*}
    To see the semigroup estimates for $(P_t^{N})$, we follow the proof of \cite[Lemma 2.11 and 2.12]{GS}. 
    Note that their proof only applies for $t\in I_M\cap [0,1]$ and we extend it to $t\in [0,1]$. To that aim, we define for $z\in\R$, $t\in [0,1]$
    \begin{align*}
        \tilde{\lambda}^N (z)=-4 N^2\sin^2 (\frac{z\pi}{2N}), \quad \tilde{\mu}_{t,N}(z)=e^{t\tilde{\lambda}^{N}(z)}
    \end{align*}
    so that $\tilde{\lambda}^{N}(j)=\lambda^{N}_{j}$ for $j\in\Z$ and for $f\in C(\Pi_N)$, $$P_{t}^{N}f = \paren[\big]{\sum_{k\in J_{N}}\tilde{\mu}_{t,N}(k) e_{k}^{N}}\ast_{N} f = \sum_{k\in J_{N}} e^{t\tilde{\lambda}^{N}(k)} \langle f,e^{N}_{k}\rangle_{L^{2}(\Pi_N)} e_{k}^{N}$$ with $(f \ast_{N} g) (y) = \frac{1}{N}\sum_{x\in \Pi_{N}} f(x) g(y-x) = \langle f , g(y-\cdot)\rangle_{L^{2}(\pi_N)}, \quad y\in\Pi_N$. We have that 
    \begin{align*}
        \partial_{z} &\tilde{\mu}_{t,N}(z) = -t 4\pi N \sin (\frac{z\pi}{2N})\cos(\frac{z\pi}{2N})\tilde{\mu}_{t,N}(z),\\ \partial_{z}^2& \tilde{\mu}_{t,N}(z)= 2  \pi^2 t \paren[\big]{\sin^2(\frac{\pi z}{2 N}) + \cos^2(\frac{\pi z}{2N}) (-1 + 8 N^2 t \sin^2(\frac{\pi z}{2 N}))} \tilde{\mu}_{t,N}(z)
    \end{align*}
    and thus the following bounds hold true
    \begin{align}\label{eq:density-est}
        \abs{\partial_{z} \tilde{\mu}_{t,N}(z)}\lesssim t \abs{z}\tilde{\mu}_{t,N}(z), \quad \abs{\partial_{z}^2 \tilde{\mu}_{t,N}(z)}\lesssim (t+t^2 \abs{z}^2)\tilde{\mu}_{t,N}(z).
    \end{align}
    The bounds \eqref{eq:density-est} are the same bounds as in \cite[equation (2.11)]{GS}. From here on, the same proof as for \cite[Lemma 2.11 and 2.12]{GS} applies replacing their definition of $\tilde{\mu}_{t,N}(z)$ with our definition above in order to deduce \eqref{eq:Calpha} and \eqref{eq:Calpha2}. To see \eqref{eq:Linfty} for $f\in L^{\infty}$, the same argument as above together with the proof of \cite[Lemma 2.11]{GS} applies. Furthermore the estimate \eqref{eq:Linfty2} again follows from \eqref{eq:Calpha2}.
\end{proof}   

\begin{proof}[Proof of \cref{lem:semigroup2}]
    The bounds for $\tilde{P}_{t}^{N}$ follow from the ones for $P_{t}$. Indeed, we find that due to the isometry property of $\Psi_N$ and the commutativity $\Psi_N\tilde{P}^{N}_t=P_t\Psi_N$
    \begin{align*}
        \norm{\tilde{P}_{t}^{N}f}_{H^{\alpha+\delta}(\Pi_N)}=\norm{\Psi_N\tilde{P}_{t}^{N}f}_{H^{\alpha+\delta}(\T)}=\norm{P_{t}\Psi_N f}_{H^{\alpha+\delta }(\T)}\leq C t^{-\delta/2}\norm{\Psi_N f}_{H^{\alpha}(\T)}=\norm{f}_{H^{\alpha}(\Pi_N)}
    \end{align*}
    and similar for the other bound. The bounds for $(P_t)$ on Sobolev spaces are classical and follow from $e^{-4\pi^2 t \abs{k}^2}\leq C t^{-\delta/2}(1+\abs{k}^2)^{-\delta}$ and $1-e^{-4\pi^2 t \abs{k}^2}\leq C t^{\delta/2}\abs{k}^{\delta}$ for $\delta\in [0,2]$.
\end{proof}

\begin{proof}[Proof of \cref{lem:composition}]
    We treat the case of $\theta \in (1,2)$; the case of $\theta \in (0,1)$ follows along similar but easier arguments. By equivalence of the H\"older norm and Besov norm,
    \begin{align*} \|g(v)\|_{\cC^\theta}= \|g(v)\|_{L^\infty}+\|g'(v)v'\|_{L^\infty}+ \sup_{x \neq y}\frac{|g'(v)(x)(v')(x)-g'(v)(y)(v')(y)|}{|x-y|^{\theta-1}}\eqqcolon E_1+E_2+E_3.
    \end{align*}
By assumption on $g$, we have that
\begin{align*}
    E_1+E_2&\leqslant K (1+\|v\|_{L^\infty}^{m})(1+\|v'\|_{L^\infty})\\
    &\leqslant K (1+\|v\|_{L^\infty}^{m})(1+\|v\|_{C^1_b}).
\end{align*}
Note that for $x,y\in\T$ and $f \in C^1$, $$f(v(x))-f(v(y))=(v(x)-v(y))\int_{0}^{1} f' (\lambda v(x)+(1-\lambda)v(y)) \,d\lambda.$$
    Moreover, for $w(x,y)\coloneqq \lambda v(x)+(1-\lambda) v(y)$, $x,y\in\T$, $\lambda\in[0,1]$, we have that $\norm{w}_{L^{\infty}(\T\times\T)}\leq 2 \norm{v}_{L^{\infty}}$. Using this,
\begin{align*}
    E_3&\leqslant \sup_{x\neq y}\frac{|g'(v)(x)v'(x)-g'(v)(y)v'(x)|+|g'(v)(y)v'(x)-g'(v)(y)v'(y)|}{|x-y|^{\theta-1}}\\
    &\leqslant \|v^\prime\|_{L^\infty} \sup_{x\neq y}\frac{\Big|(v(x)-v(y))\int_0^1 g''(\lambda v(x)+(1-\lambda)v(y))\, d\lambda\Big|}{|x-y|^{\theta-1}}+\|g'(v)\|_{L^\infty}\|v'\|_{\cC^{\theta-1}}\\
    &\leqslant K(1+\|v\|_{L^\infty}^{m})\Big(\|v'\|_{L^\infty}\|v\|_{\cC^{\theta-1}}+\|v\|_{C^{\theta}}\Big)\\
    &\lesssim K(1+\|v\|_{L^\infty}^{m})\Big(\|v\|_{C^1_b}\|v\|_{\cC^{\theta-1}}+\|v\|_{\cC^\theta}\Big)
\end{align*}
and therefore \eqref{eq:comp-B} holds.
    Now, we consider the case of $v \in H^{\theta}(\T)\cap L^\infty(\T)$ for $\theta \in (0,1)$, i.e. proving \eqref{eq:comp-H}.
    For $s \in (0,1)$, recall the Slobodeckij norm defined by
    \begin{align} \label{eq:SoboSlobo}
  \|f\|_{H^s}\coloneqq \|f\|_{L^2}+\Big(\int_{\T} \int_{\T} \frac{|\Delta_hf(x)|^2}{|h|^{1+2s}}\, dx\, dh\Big)^{1/2},
    \end{align}
    with $\Delta_h f (x) := f(x+h)-f(x)$, which is equivalent to $B^s_{2,2}(\T)$ (see \cref{lem:spaces-equivalence}). 
Then we have that
\begin{align*}
\|g(v)\|_{H^\theta}&\lesssim K(1+\|v\|^{m}_{L^\infty})+\Big(\int_{\T} \int_{\T} \frac{|\Delta_h(g(v))(x)|^2}{|h|^{1+2\theta}}\, dx\, dh\Big)^{1/2}\\
&\lesssim K(1+\|v\|^{m}_{L^\infty})+\Big(\int_{\T} \int_{\T} h^{-1+2\theta}\Big(\Delta_h v(x)\int_0^1 g'(\lambda v(x+h)+(1-\lambda) v(x))\, d\lambda \Big)^2\, dx\, dh\Big)^{1/2}\\
&\lesssim K(1+\|v\|^{m}_{L^\infty})(1+\|v\|_{H^\theta}).\qedhere
\end{align*}
\end{proof}
\end{section}
\end{appendices}

\bigskip
\textbf{Acknowledgments}
\begin{itemize}
    \item[(LA)] This research was funded in whole or in part by the Austrian Science Fund (FWF) [10.55776/STA119]. For open access purposes, the author has applied a CC BY public copyright license to any author-accepted manuscript version arising from this submission.
\item [(MG)] Funded by the European Union (ERC, SPDE, 101117125). Views and opinions expressed
are however those of the author(s) only and do not necessarily reflect those of the European Union
or the European Research Council Executive Agency. Neither the European Union nor the granting
authority can be held responsible for them. 
\item [(HK)] Acknowledges support from the Deutsche Forschungsgemeinschaft (DFG) CRC/TRR 388 ``Rough Analysis, Stochastic Dynamics and Related Fields" - Project ID 516748464 within sub-project A02.
\end{itemize}

\bigskip
\textbf{Declaration on AI usage}

Except for \cref{sec:numerics}, AI tools were used only for minor tasks such as reference check. The numerical experiments in \cref{sec:numerics} were carried out using Claude Opus 5.5.

\bibliographystyle{Martin}
\bibliography{numerics}

\end{document}

%% file: results_table.tex
\begin{tabular}{lcccccc}
\toprule
 & $M(N)$ & levels $N$ & ref.\ $(N,M)$ & paths & rate in $N$ (theory) & rate in $MN$ (theory) \\
\midrule
FD & $N^2$ & $4,\dots,64$ & $(512,262144)$ & 200 & $0.496\,[0.478,0.513]$ $(\tfrac12)$ & $0.165\,[0.159,0.171]$ $(\tfrac16)$ \\
AEE & $N^{1/2}$ & $4,\dots,1024$ & $(16384,128)$ & 400 & $0.518\,[0.505,0.530]$ $(\tfrac12)$ & $0.345\,[0.337,0.354]$ $(\tfrac13)$ \\
DAEE & $N^{3/2}$ & $4,\dots,256$ & $(1024,32768)$ & 250 & $1.535\,[1.495,1.576]$ $(\tfrac32)$ & $0.614\,[0.598,0.631]$ $(\tfrac35)$ \\
\bottomrule
\end{tabular}